\documentclass[11pt,a4paper,leqno]{article}
\usepackage{a4wide}
\usepackage{latexsym}
\usepackage{amsmath}
\usepackage{amssymb}
\newtheorem{defin}{Definition}
\newtheorem{lemma}{Lemma}
\newtheorem{prop}{Proposition}
\newtheorem{theo}{Theorem}

\newenvironment{proof}{\medskip\par\noindent{\bf Proof}}{\hfill $\Box$
\medskip\par}
\begin{document}
\title{On complex singularity analysis for some linear partial differential equations in $\mathbb{C}^{3}$}
\author{{\bf A. Lastra\footnote{The author is partially supported by the project MTM2012-31439 of Ministerio de Ciencia e
Innovacion, Spain}, S. Malek\footnote{The author is partially supported by the french ANR-10-JCJC 0105 project and the PHC
Polonium 2013 project No. 28217SG.}, C. Stenger,}\\
University of Alcal\'{a}, Departamento de Matem\'{a}ticas,\\
Ap. de Correos 20, E-28871 Alcal\'{a} de Henares (Madrid), Spain,\\
University of Lille 1, Laboratoire Paul Painlev\'e,\\
59655 Villeneuve d'Ascq cedex, France,\\
University of La Rochelle, Avenue Michel Cr\'epeau,\\
17042 La Rochelle cedex France\\
{\tt alberto.lastra@uah.es}\\
{\tt Stephane.Malek@math.univ-lille1.fr }\\
{\tt catherine.stenger@univ-lr.fr }}
\date{March, 28 2013}
\maketitle
\thispagestyle{empty}
{ \small \begin{center}
{\bf Abstract}
\end{center}
We investigate the existence of local holomorphic solutions $Y$ of linear partial differential equations in three complex variables
whose coefficients are singular along an analytic variety $\Theta$ in $\mathbb{C}^{2}$. The coefficients are written as linear
combinations of powers of a solution $X$ of some first order nonlinear partial differential equation following an idea we have initiated in
a previous work \cite{mast}. The solutions $Y$ are shown to develop
singularities along $\Theta$ with estimates of exponential type depending on the growth's rate of $X$ near the singular variety.
We construct these solutions with the help of series of functions with infinitely many variables which involve 
derivatives of all orders of $X$ in one variable. Convergence and bounds estimates of these series are studied
using a majorant series method which leads to an auxiliary functional equation that contains differential operators in
infinitely many variables. Using a fixed point argument, we show that these functional equations actually have solutions in some
Banach spaces of formal power series.\medskip

} \bigskip \bigskip

\section{Introduction}

In this paper, we study a family of linear partial differential equations of the form
\begin{multline}
\partial_{w}^{S}Y(t,z,w) = \sum_{k \in \mathcal{S}} (a_{1,k}(t,z,w)\partial_{t}\partial_{w}^{k}Y(t,z,w)\\
+ a_{2,k}(t,z,w)\partial_{z}\partial_{w}^{k}Y(t,z,w)+ a_{3,k}(t,z,w)\partial_{w}^{k}Y(t,z,w)) \label{lin_pde_Y_intro}
\end{multline}
for given initial data $\partial_{w}^{j}Y(t,z,0) = \varphi_{j}(t,z)$, $0 \leq j \leq S-1$, where $\mathcal{S}$ is a subset of
$\mathbb{N}^{2}$ and $S$ is an integer which satisfy the constraints
(\ref{shape_lin_PDE_in_C3}). The coefficients $a_{m,k}(t,z,w)$ are holomorphic functions on some domain
$(D(0,r)^{2} \setminus \Theta) \times D(0,\bar{w})$ where $\Theta$ is some analytic variety of $D(0,r)^{2}$ (where $D(0,\delta)$ denotes
the disc centered at 0 in $\mathbb{C}$ with radius $\delta>0$) and the initial data $\varphi_{j}(t,z)$ are assumed to be holomorphic
functions on the polydisc $D(0,r)^{2}$. 

In order to avoid cumbersome statements and tedious computations, the authors have chosen to restrict their study to equations
(\ref{lin_pde_Y_intro}) that involve at most first order derivatives with respect to $t$ and $z$ but the method proposed in this work can
also be extended to higher order derivatives too.

In this work, we plan to construct holomorphic solutions of the problem (\ref{lin_pde_Y_intro}) on
$(D(0,r)^{2} \setminus \Theta) \times D(0,\bar{w})$ and we will give precise growth estimates for these solutions
near the singular variety $\Theta$ of the coefficients $a_{m,k}(t,z,w)$ (Theorem 1).

There exists a huge literature on the study of complex singularities and analytic continuation of solutions to linear partial differential
equations starting from the fundamental contributions of J. Leray in \cite{le}. Many important results are known for singular
initial data and concern equations with bounded holomorphic coefficients. In that context, the singularities of the solution are generally
contained in characteristic hypersurfaces issued from the singular locus of the initial conditions. For meromorphic initial
data, we may refer to \cite{copata}, \cite{ha}, \cite{ou1}, \cite{ou2} and for more general ramified multivalued initial data, we may cite
\cite{halewa}, \cite{ig}, \cite{stsh}, \cite{wa}. In our framework, the initial data are assumed to be non singular and the coefficients
of the equation now carry the singularities. To the best knowledge of the authors, few results have been worked out in that case.
For instance, the research of so-called \emph{fuchsian singularities} in the context of partial differential equations is widely
developed, we provide \cite{al}, \cite{bolepa}, \cite{geta}, \cite{man} as examples of references in this direction.
It turns out that the situation we consider is actually close to a singular perturbation problem since the nature
of the equation changes nearby the singular locus of it's coefficients.

This work is a continuation of our previous study \cite{mast}. In the paper \cite{mast}, the authors focused on linear
partial differential equations in $\mathbb{C}^{2}$. They have constructed
local holomorphic solutions with a careful study of their asymptotic behaviour near the singular locus of the initial data.
These initial data were chosen to be polynomial in $t$,$z$ and a function $u(t)$ satisfying some nonlinear differential equation of
first order on some punctured disc $D(t_{0},r) \setminus \{ t_{0} \} \subset \mathbb{C}$ and owning an isolated singularity at
$t_{0}$ which is either a pole or an algebraic branch point according to a result of P. Painlev\'e.
Inspired by the classical {\it tanh method} introduced in \cite{malf}, they have considered formal series solutions
of the form
\begin{equation}
u(t,z) = \sum_{l \geq 0} u_{l}(t,z) (u(t))^{l} \label{sol_PDE_th}
\end{equation}
where $u_{l}$ are holomorphic functions on $D(t_{0},r) \times D$ where $D \subset \mathbb{C}$ is a small disc centered at 0. They
have given suitable conditions for these series to be well defined and holomorphic for $t$ in a sector $S$ with vertex $t_{0}$ and
moreover as $t$ tends to $t_{0}$ the solutions $u(t,z)$ are shown to carry at most exponential bounds estimates of the form
$C\exp(M|t-t_{0}|^{-\mu})$ for some constants $C,M,\mu>0$.

In this work, the coefficients $a_{m,k}(t,z,w)$ are constructed as polynomials in some function $X(t,z)$ with holomorphic coefficients
in $(t,z,w)$, where $X(t,z)$ is now assumed to solve some nonlinear partial differential equation of first order and is asked to be
holomorphic on a domain $D(0,r)^{2} \setminus \Theta$ and to be singular along the analytic variety $\Theta$. For some specific
choice of $X(t,z)$, one can build the coefficients $a_{m,k}(t,z,w)$ for instance to be some rational functions of $(t,z)$
(see Example 1 of Section 2.1).

In our setting, one cannot achieve the goal only dealing with formal expansions involving the function $X(t,z)$ like (\ref{sol_PDE_th}) since the
derivatives of $X(t,z)$ with respect to $t$ or $z$ cannot be expressed only in term of $X(t,z)$. In order to get suitable recursion
formulas, it turns out that we need to deal with series expansions that take into account all the derivatives of $X(t,z)$ with respect
to $z$. For this reason, the construction of the solutions will follow the one introduced in a recent
work of H. Tahara and will involve Banach spaces of holomorphic functions with infinitely many variables.

In the paper \cite{ta}, H. Tahara introduced a new equivalence problem connecting two given nonlinear partial differential equations of
first order in the complex domain. He showed that the equivalence maps
have to satisfy so called coupling equations which are nonlinear partial differential equations of first order but
with infinitely many variables. It is worthwhile saying that within the framework of mathematical physics, spaces of functions of infinitely
many variables play a fundamental role in the study of nonlinear integrable partial differential equations known as solitons equations
as described in the theory of M. Sato. See \cite{ohsatato} for an introduction.

The layout of the paper is a follows. In a first step described in Section 2.2, we construct formal series of the form
\begin{equation}
 U(t,z,w)=\sum_{\alpha \geq 0} \phi_{\alpha}(t,z,(\frac{\partial_{z}^{h}X(t,z)}{h! \nu^{h}})_{0 \leq h \leq \alpha})
\frac{w^{\alpha}}{\alpha!}, \label{U_defin_intro}
\end{equation}
solutions of some auxiliary non-homogeneous integro-differential equation (\ref{ID_U}) with polynomial coefficients in $X(t,z)$. The
coefficients $\phi_{\alpha}$, $\alpha \geq 0$, are holomorphic functions on some polydisc in $\mathbb{C}^{\alpha+3}$ that satisfy
some differential recursion (Proposition 1).

In Section 2.3, we establish a sequence of inequalities for the modulus of the differentials of arbitrary order of the functions
$\phi_{\alpha}$ denoted $\varphi_{\alpha,n_{0},n_{1},(l_{h})_{0 \leq h \leq \alpha}}$ for all non-negative integers
$\alpha,n_{0},n_{1},l_{h}$ with $0 \leq h \leq \alpha$ (Proposition 2). In the next section, we construct a sequence of coefficients
$\psi_{\alpha,n_{0},n_{1},(l_{h})_{0 \leq h \leq \alpha}}$ which is larger than the latter sequence
$$ \varphi_{\alpha,n_{0},n_{1},(l_{h})_{0 \leq h \leq \alpha}} \leq \psi_{\alpha,n_{0},n_{1},(l_{h})_{0 \leq h \leq \alpha}} $$
for any non-negative integers $\alpha,n_{0},n_{1},l_{h}$ with $0 \leq h \leq \alpha$ and whose generating formal series satisfies some
integro-differential functional equation (\ref{functional_eq_defin_psi}) that involves differential operators with infinitely many
variables (Propositions 3 and 4). The idea of considering recursions over the complete family of derivatives and the use of majorant series
which lead to auxiliary Cauchy problems were already applied in former papers by the authors of this work, see
\cite{lamasa}, \cite{ma1}, \cite{ma2}, \cite{ma3}, \cite{mast}.

In Section 3, we solve the functional equation (\ref{functional_eq_defin_psi}) by applying a fixed point argument in some Banach
space of formal series with infinitely many variables (Proposition 10). The definition of these Banach spaces (Definition 2) is inspired
from formal series spaces introduced in our previous work \cite{mast}. The core of the proof is based on continuity properties of linear
integro-differential operators in infinitely many variables explained in Section 3.1 and constitutes the most technical part of the paper.

Finally, in Section 4, we prove the main result of our work. Namely, we construct analytic functions $Y(t,z,w)$, solutions of
(\ref{lin_pde_Y_intro}) for the prescribed initial data, defined on sets $K \times D(0,\bar{w})$ for any compact set
$K \subset D(0,r)^{2} \setminus \Theta$ with precise bounds of exponential type in term of the maximum value of $|X(t,z)|$ over $K$
(Theorem 1). The proof puts together all the constructions performed in the previous sections. More precisely, for some specific choice
of the non-homogeneous term in the equation (\ref{ID_U}), a formal solution (\ref{U_defin_intro}) of (\ref{ID_U}) gives rise to a
formal solution
$Y(t,z,w)$ of (\ref{lin_pde_Y_intro}) with the given initial data that can be written as the sum of the integral $\partial_{w}^{-S}U(t,z,w)$
and a polynomial in $w$ having the initial data $\varphi_{j}$ as coefficients. Owing to the fact that the generating series
of the sequence $\psi_{\alpha,n_{0},n_{1},(l_{h})_{0 \leq h \leq \alpha}}$, solution of (\ref{functional_eq_defin_psi}), belongs
to the Banach spaces mentioned above, we get estimates for the holomorphic functions $\phi_{\alpha}$ with precise bounds
of exponential type in term of the radii of the polydiscs where they are defined, see (\ref{sup_phi_alpha_v_u<exp_rho}). As a result,
the formal solution $U(t,z,w)$ is actually
convergent for $w$ near the origin and for $(t,z)$ belonging to any compact set of $D(0,r) \setminus \Theta$. Moreover,
exponential bounds are achieved, see (\ref{sup_U_tzw<exp_rho}). The same properties then hold for $Y(t,z,w)$.

\section{Formal series solutions of linear integro-differential equations}
\subsection{Some nonlinear partial differential equation}
We consider the following nonlinear partial differential equation
\begin{equation}
\partial_{t}X(t,z) = a(t,z)\partial_{z}X(t,z) + \sum_{p=0}^{d} a_{p}(t,z)X^{p}(t,z) \ \ , \ \ X(0,z)=f(z) \label{PDE_nl}
\end{equation}
where $d \geq 2$ is some integer, the coefficients $a(t,z)$, $a_{p}(t,z)$ are holomorphic functions on some polydisc
$D(0,R')^{2} \subset \mathbb{C}^{2}$ such that $a_{d}(t,z)$ is not identically equal to zero on $D(0,R')^{2}$ and the initial data
$f(z)$ is holomorphic on $D(0,R')^{2}$. Notice that the problem (\ref{PDE_nl}) can be solved by using the classical method of
characteristics which is described in some classical textbooks like \cite{de}, p. 118 or \cite{ev}, p. 100.

We make the assumption that (\ref{PDE_nl}) has a holomorphic solution $X(t,z)$ on $D(0,R')^{2} \setminus \Theta$ where
$\Theta$ is some \emph{analytic variety} of $D(0,R')^{2}$ (i.e. for any point $M \in D(0,R')^{2}$, there exists a neighborhood
$\mathcal{U}$ of $M$ in $D(0,R')^{2}$ such that $\mathcal{U} \cap \Theta$ is the common zero locus of a finite set of holomorphic
functions $\{f_{1},\ldots,f_{m}\}$ on $\mathcal{U}$ for some integer $m \geq 1$).\medskip

\noindent {\bf Example 1:} The solution of the problem
$$ \partial_{t}X(t,z) = \partial_{z}X(t,z) + X^{2}(t,z) \ \ , \ \ X(0,z)=f(z) $$
where $f(z)$ is some polynomial on $\mathbb{C}$, writes
$$ X(t,z) = \frac{f(t+z)}{1 - tf(t+z)}.$$
Therefore, $X(t,z)$ is holomorphic on $\mathbb{C}^{2} \setminus \Theta$ where
$\Theta = \{(t,z) \in \mathbb{C}^{2} / 1 - tf(t+z)=0\}$ is an algebraic variety.

\noindent {\bf Example 2:} The solution of the problem
$$ \partial_{t}X(t,z) = z\partial_{z}X(t,z) + X^{2}(t,z) \ \ , \ \ X(0,z)=f(z)$$
where $f(z)$ is an holomorphic function on $\mathbb{C}$, writes
$$ X(t,z) = \frac{f(\exp(t)z)}{1 - tf(\exp(t)z)}.$$
The solution $X(t,z)$ is holomorphic on $\mathbb{C}^{2} \setminus \Theta$ where
the singular variety $\Theta$ is the zero locus of one analytic function $\Theta = \{(t,z) \in \mathbb{C}^{2} / 1 - tf(\exp(t)z)=0 \}$.

\subsection{Composition series}

Let $X$ be as in the previous subsection. In the following, we choose a compact subset $K_{0}$ with non-empty interior of
$D(0,R)^{2} \setminus \Theta$ for some $R<R'$ and we consider a real number $\rho>1$ such that
$$
 \sup_{(t,z) \in K_{0}} |X(t,z)| \leq \rho/2.
$$
Let $K \varsubsetneq K_{0}$ be a compact set with non-empty interior $\mathrm{Int}(K)$. From the Cauchy formula, there exists a real
number $\nu>0$ such that
\begin{equation}
\sup_{(t,z) \in \mathrm{Int}(K)} \frac{|\partial_{z}^{h}X(t,z)|}{h! \nu^{h}} \leq \rho/2 \label{sup_diff_z_X<rho}
\end{equation}
for all integers $h \geq 0$. For all integers $\alpha \geq 0$, we denote $I(\alpha) = \{0,\ldots,\alpha \}$. We consider a sequence of
functions $\phi_{\alpha}(v_{0},v_{1},(u_{h})_{h \in I(\alpha)})$ which are holomorphic and bounded 
on the polydisc $D(0,R)^{2} \Pi_{h \in I(\alpha)} D(0,\rho)$, for all $\alpha \geq 0$.

We define the formal series in the $w$ variable,
\begin{equation}
U(t,z,w) = \sum_{\alpha \geq 0} \phi_{\alpha}(t,z,(\frac{\partial_{z}^{h}X(t,z)}{h! \nu^{h}})_{h \in I(\alpha)})
\frac{{w}^{\alpha}}{\alpha!}. \label{U_defin}
\end{equation}
For all $\alpha \geq 0$, we consider a holomorphic and bounded function
$\tilde{\omega}_{\alpha}(v_{0},v_{1},(u_{h})_{h \in I(\alpha)})$ on the product $D(0,R')^{2} \Pi_{h \in I(\alpha)} D(0,\rho)$.
We define the formal series
\begin{equation}
\tilde{\omega}(t,z,w) = \sum_{\alpha \geq 0} \tilde{\omega}_{\alpha}(t,z,(\frac{\partial_{z}^{h}X(t,z)}{h! \nu^{h}})_{h \in I(\alpha)})
\frac{w^{\alpha}}{\alpha!}. \label{tilde_omega_defin}
\end{equation}
Let $\mathcal{S}$ be a finite subset of $\mathbb{N}$ and let $S \geq 1$ be an integer which satisfies the property that
\begin{equation}
S > k 
\end{equation}
for all $k \in \mathcal{S}$. For all $k \in \mathcal{S}$, $m=1,2,3$, we consider holomorphic functions
\begin{equation}
b_{m,k}(t,z,u_{0},w) = \sum_{\alpha \geq 0} b_{m,k,\alpha}(t,z,u_{0}) \frac{w^{\alpha}}{\alpha !} \label{b_mk_t_z_u0_w_defin}
\end{equation}
on $D(0,R')^{2} \times \mathbb{C} \times D(0,\bar{w})$, for some $\bar{w}>0$, which are moreover polynomial with respect to $u_{0}$
of degree $d_{m,k} \geq 0$.

\begin{prop} Assume that the sequence of functions $(\phi_{\alpha})_{\alpha \geq 0}$ satisfies the following recursion
\begin{multline}
\frac{\phi_{\alpha}(v_{0},v_{1},(u_{h})_{h \in I(\alpha)})}{\alpha!} = \sum_{k \in \mathcal{S}}
\sum_{\alpha_{1}+\alpha_{2}=\alpha,\alpha_{2} \geq S-k}
\frac{b_{1,k,\alpha_{1}}(v_{0},v_{1},u_{0})}{\alpha_{1}!}\\
\times \left( \frac{\partial_{v_0}\phi_{\alpha_{2}+k-S}(v_{0},v_{1},(u_{h})_{h \in I(\alpha_{2}+k-S)})}{\alpha_{2}!}
+ \sum_{j \in I(\alpha_{2}+k-S)} ( \sum_{l_{1}+l_{2}=j} \frac{\partial_{v_1}^{l_1}a(v_{0},v_{1})}{l_{1}!\nu^{l_1}}(l_{2}+1)\nu
u_{l_{2}+1} \right. \\
\left. + \sum_{p=0}^{d} \sum_{j_{0}+\ldots+j_{p}=j} \frac{\partial_{v_1}^{j_0}a_{p}(v_{0},v_{1})}{j_{0}!\nu^{j_0}}
\Pi_{l=1}^{p} u_{j_{l}}) \frac{\partial_{u_j}\phi_{\alpha_{2}+k-S}(v_{0},v_{1},(u_{h})_{h \in I(\alpha_{2}+k-S)})}{\alpha_{2}!} \right)\\
+ \sum_{k \in \mathcal{S}} \sum_{\alpha_{1}+\alpha_{2}=\alpha,\alpha_{2} \geq S-k}
\frac{b_{2,k,\alpha_{1}}(v_{0},v_{1},u_{0})}{\alpha_{1}!}\\
\times \left( \frac{\partial_{v_1}\phi_{\alpha_{2}+k-S}(v_{0},v_{1},(u_{h})_{h \in I(\alpha_{2}+k-S)})}{\alpha_{2}!} \right.\\
\left. + \sum_{j \in I(\alpha_{2}+k-S)} (j+1)\nu u_{j+1} \frac{\partial_{u_j}\phi_{\alpha_{2}+k-S}(v_{0},v_{1},(u_{h})_{h \in
I(\alpha_{2}+k-S)})}{\alpha_{2}!}  \right) \\+ \sum_{k \in \mathcal{S}}
\sum_{\alpha_{1}+\alpha_{2}=\alpha,\alpha_{2} \geq S-k}
\frac{b_{3,k,\alpha_{1}}(v_{0},v_{1},u_{0})}{\alpha_{1}!}
\times \frac{\phi_{\alpha_{2}+k-S}(v_{0},v_{1},(u_{h})_{h \in I(\alpha_{2}+k-S)})}{\alpha_{2}!}\\
+ \frac{\tilde{\omega}_{\alpha}(v_{0},v_{1},(u_{h})_{h \in I(\alpha)})}{\alpha!}
\label{recursion_phi_alpha}
\end{multline}
for all $\alpha \geq 0$, all $v_{0},v_{1} \in D(0,R)$, all $u_{h} \in D(0,\rho)$, for $h \in I(\alpha)$. 
Then, the formal series $U(t,z,w)$ satisfies the following integro-differential equation
\begin{multline}
 U(t,z,w) = \sum_{k \in \mathcal{S}} (b_{1,k}(t,z,X(t,z),w)\partial_{t}\partial_{w}^{-S+k}U(t,z,w) \\
+ b_{2,k}(t,z,X(t,z),w)\partial_{z}\partial_{w}^{-S+k}U(t,z,w) + b_{3,k}(t,z,X(t,z),w)\partial_{w}^{-S+k}U(t,z,w))
+ \tilde{\omega}(t,z,w) \label{ID_U}
\end{multline}
for all $(t,z) \in \mathrm{Int}(K)$, where $\partial_{w}^{-m}$ denotes the $m$-iterate of the usual integration operator
$\int_{0}^{w}[.]ds$
\end{prop}
\begin{proof} We have that
\begin{multline*}
b_{3,k}(t,z,X(t,z),w)\partial_{w}^{-S+k}U(t,z,w) \\
= \sum_{\alpha \geq 0}
(\sum_{\alpha_{1}+\alpha_{2}=\alpha, \alpha_{2} \geq S-k} \alpha! \frac{b_{3,k,\alpha_{1}}(t,z,X(t,z))}{\alpha_{1}!}
\frac{\phi_{\alpha_{2}+k-S}(t,z,(\frac{\partial_{z}^{h}X(t,z)}{h! \nu^{h}})_{h \in I(\alpha_{2}+k -S)})}{\alpha_{2}!})
\frac{w^{\alpha}}{\alpha!}
\end{multline*}
and we also see that
\begin{multline*}
b_{2,k}(t,z,X(t,z),w)\partial_{z}\partial_{w}^{-S+k}U(t,z,w) \\
= \sum_{\alpha \geq 0}
(\sum_{\alpha_{1}+\alpha_{2}=\alpha, \alpha_{2} \geq S-k} \alpha! \frac{b_{2,k,\alpha_{1}}(t,z,X(t,z))}{\alpha_{1}!}
\frac{\partial_{z}(\phi_{\alpha_{2}+k-S}(t,z,(\frac{\partial_{z}^{h}X(t,z)}{h! \nu^{h}})_{h \in I(\alpha_{2}+k -S)}))}{\alpha_{2}!})
\frac{w^{\alpha}}{\alpha!}
\end{multline*}
with
\begin{multline*}
\partial_{z}(\phi_{\alpha_{2}+k-S}(t,z,(\frac{\partial_{z}^{h}X(t,z)}{h! \nu^{h}})_{h \in I(\alpha_{2}+k -S)}))=
(\partial_{v_1}\phi_{\alpha_{2}+k -S})(t,z,(\frac{\partial_{z}^{h}X(t,z)}{h! \nu^{h}})_{h \in I(\alpha_{2}+k -S)}) \\
+ \sum_{j \in I(\alpha_{2}+k -S)} (j+1)\nu \frac{\partial_{z}^{j+1}X(t,z)}{(j+1)!\nu^{j+1}}
(\partial_{u_j}\phi_{\alpha_{2}+k - S})(t,z,(\frac{\partial_{z}^{h}X(t,z)}{h! \nu^{h}})_{h \in I(\alpha_{2}+k -S)}),
\end{multline*}
for all $(t,z) \in \mathrm{Int}(K)$. We also get that
\begin{multline*}
b_{1,k}(t,z,X(t,z),w)\partial_{t}\partial_{w}^{-S+k}U(t,z,w)\\
= \sum_{\alpha \geq 0} (\sum_{\alpha_{1}+\alpha_{2}=\alpha,\alpha_{2}\geq S-k}
\alpha!  \frac{b_{1,k,\alpha_{1}}(t,z,X(t,z))}{\alpha_{1}!}
\frac{\partial_{t}(\phi_{\alpha_{2}+k-S}(t,z,(\frac{\partial_{z}^{h}X(t,z)}{h! \nu^{h}})_{h \in I(\alpha_{2}+k -S)})}{\alpha_{2}!})
\frac{w^{\alpha}}{\alpha!}
\end{multline*}
with
\begin{multline*}
\partial_{t}(\phi_{\alpha_{2}+k-S}(t,z,(\frac{\partial_{z}^{h}X(t,z)}{h! \nu^{h}})_{h \in I(\alpha_{2}+k -S)})=
(\partial_{v_0}\phi_{\alpha_{2}+k-S})(t,z,(\frac{\partial_{z}^{h}X(t,z)}{h! \nu^{h}})_{h \in I(\alpha_{2}+k-S)}) \\
+ \sum_{j \in I(\alpha_{2}+k -S)} \frac{\partial_{t}\partial_{z}^{j}X(t,z)}{j!\nu^{j}}
(\partial_{u_j}\phi_{\alpha_{2}+k-S})(t,z,(\frac{\partial_{z}^{h}X(t,z)}{h! \nu^{h}})_{h \in I(\alpha_{2}+k-S)}),
\end{multline*}
for all $(t,z) \in \mathrm{Int}(K)$. Now, from (\ref{PDE_nl}) and the classical Schwarz's result on equality of mixed partial
derivatives, we get that
$$
\frac{\partial_{t}\partial_{z}^{j}X(t,z)}{j!\nu^{j}} = \frac{\partial_{z}^{j}\partial_{t}X(t,z)}{j!\nu^{j}} =
\frac{1}{j!\nu^{j}}\partial_{z}^{j}(a(t,z)\partial_{z}X(t,z) + \sum_{p=0}^{d} a_{p}(t,z)X^{p}(t,z))
$$
and from the Leibniz formula, we can write
$$ \frac{1}{j!\nu^{j}}\partial_{z}^{j}(a(t,z)\partial_{z}X(t,z)) =
\sum_{l_{1}+l_{2}=j} \frac{\partial_{z}^{l_1}a(t,z)}{l_{1}!\nu^{l_1}}(l_{2}+1)\nu
\frac{\partial_{z}^{l_{2}+1}X(t,z)}{(l_{2}+1)!\nu^{l_{2}+1}} $$ 
and
$$ \frac{1}{j!\nu^{j}}\partial_{z}^{j}(a_{p}(t,z)X^{p}(t,z)) = \sum_{j_{0}+\ldots+j_{p} = j}
\frac{\partial_{z}^{j_0}a_{p}(t,z)}{j_{0}!\nu^{j_0}}\Pi_{l=1}^{p}\frac{\partial_{z}^{j_l}X(t,z)}{j_{l}!\nu^{j_l}},
$$
for all $(t,z) \in \mathrm{Int}(K)$. Finally, gathering all the equalities above and using the recursion
(\ref{recursion_phi_alpha}), one gets the integro-differential equation (\ref{ID_U}).
\end{proof}

\subsection{Recursion for the derivatives of the functions $\phi_{\alpha}$, $\alpha \geq 0$}

We consider a sequence of functions $\phi_{\alpha}(v_{0},v_{1},(u_{h})_{h \in I(\alpha)})$, $\alpha \geq 0$, which are holomorphic
and bounded on some polydisc $D(0,R)^{2} \Pi_{h \in I(\alpha)} D(0,\rho)$
for some real numbers $R>0$ and $\rho>1$ and which satisfy the equalities
(\ref{recursion_phi_alpha}). We introduce the sequences
\begin{equation}
\varphi_{\alpha,n_{0},n_{1},(l_{h})_{h \in I(\alpha)}} = \sup_{|v_{0}| < R,|v_{1}| < R, |u_{h}| < \rho, h \in I(\alpha)}
|\partial_{v_0}^{n_0}\partial_{v_1}^{n_1} \Pi_{h \in I(\alpha)}\partial_{u_h}^{l_h}\phi_{\alpha}(v_{0},v_{1},(u_{h})_{h \in I(\alpha)})|
\label{varphi_alpha_n_l_defin}
\end{equation}
for all $n_{0},n_{1} \geq 0$, all $l_{h} \geq 0$, $h \in I(\alpha)$, for all $\alpha \geq 0$. We define also the following sequences
\begin{multline}
b_{m,k,\alpha,n_{0},n_{1},l_{0}} = \sup_{|v_{0}| < R,|v_{1}| < R, |u_{0}| < \rho}
|\partial_{v_0}^{n_{0}}\partial_{v_1}^{n_{1}}\partial_{u_0}^{l_{0}}b_{m,k,\alpha}(v_{0},v_{1},u_{0})|,\\
\tilde{\omega}_{\alpha,n_{0},n_{1},(l_{h})_{h \in I(\alpha)}} = \sup_{|v_{0}| < R,|v_{1}| < R, |u_{h}| < \rho, h \in I(\alpha)}
|\partial_{v_0}^{n_0}\partial_{v_1}^{n_1} \Pi_{h \in I(\alpha)}\partial_{u_h}^{l_h}
\tilde{\omega}_{\alpha}(v_{0},v_{1},(u_{h})_{h \in I(\alpha)})|
\label{b_mk_alpha_n_l_and_tilde_omega_alpha_n_l_defin}
\end{multline}
for $m=1,2,3$ and $k \in \mathcal{S}$. We put
\begin{multline}
A_{j}(v_{0},v_{1},(u_{h})_{h \in I(\alpha+1)})
=  \sum_{l_{1}+l_{2}=j}
\frac{\partial_{v_1}^{l_1}a(v_{0},v_{1})}{l_{1}!\nu^{l_1}}(l_{2}+1)\nu
u_{l_{2}+1} \\
 + \sum_{p=0}^{d} \sum_{j_{0}+\ldots+j_{p}=j} \frac{\partial_{v_1}^{j_0}a_{p}(v_{0},v_{1})}{j_{0}!\nu^{j_0}}
\Pi_{l=1}^{p} u_{j_{l}} \label{defin_A_j}
\end{multline}
and
\begin{equation}
B_{j}(v_{0},v_{1},(u_{h})_{h \in I(\alpha+1)}) = (j+1)\nu u_{j+1} \label{defin_B_j}
\end{equation}
for all $j \in I(\alpha)$, $v_{0},v_{1} \in D(0,R')$ and $u_{h} \in \mathbb{C}$, $h \in I(\alpha)$. We define the sequences
$$ A_{j,\alpha,n_{0},n_{1},(l_{h})_{h \in I(\alpha+1)}} =  \sup_{|v_{0}| < R,|v_{1}| < R, |u_{h}| < \rho, h \in I(\alpha)}
|\partial_{v_0}^{n_0}\partial_{v_1}^{n_1} \Pi_{h \in I(\alpha)}\partial_{u_h}^{l_h}A_{j}(v_{0},v_{1},(u_{h})_{h \in I(\alpha+1)})| $$
and
$$ B_{j,\alpha,n_{0},n_{1},(l_{h})_{h \in I(\alpha+1)}} =  \sup_{|v_{0}| < R,|v_{1}| < R, |u_{h}| < \rho, h \in I(\alpha)}
|\partial_{v_0}^{n_0}\partial_{v_1}^{n_1} \Pi_{h \in I(\alpha)}\partial_{u_h}^{l_h}B_{j}(v_{0},v_{1},(u_{h})_{h \in I(\alpha+1)})| $$
for all $j \in I(\alpha)$, all $n_{0},n_{1} \geq 0$, all $l_{h} \geq 0$, $h \in I(\alpha+1)$, for all $\alpha \geq 0$. We also recall the
definition of the Kronecker symbol $\delta_{0,l}$ which is equal to 0 if $l \neq 0$ and
equal to 1 if $l=0$.
\begin{prop} The sequence $\varphi_{\alpha,n_{0},n_{1},(l_{h})_{h \in I(\alpha)}}$ satisfies the following inequality:
\begin{multline}
\frac{\varphi_{\alpha,n_{0},n_{1},(l_{h})_{h \in I(\alpha)}}}{\alpha!} \leq  \sum_{k \in \mathcal{S}} \sum_{\scriptscriptstyle{
\stackrel{\alpha_{1}+\alpha_{2} = \alpha}{\alpha_{2} \geq S-k}}}
\sum_{\scriptscriptstyle{\stackrel{n_{0,1}+n_{0,2}=n_{0},n_{1,1}+n_{1,2}=n_{1}}{l_{h,1}+l_{h,2}=l_{h},h \in I(\alpha)}}}
\frac{n_{0}!n_{1}!\Pi_{h \in I(\alpha)}l_{h}!}{n_{0,1}!n_{0,2}!n_{1,1}!n_{1,2}!\Pi_{h \in I(\alpha)}l_{h,1}!l_{h,2}!}\\
\frac{b_{1,k,\alpha_{1},n_{0,1},n_{1,1},l_{0,1}}}{\alpha_{1}!}\Pi_{h \in I(\alpha)\setminus\{0 \}} \delta_{0,l_{h,1}} \times
\frac{\varphi_{\alpha_{2}+k-S,n_{0,2}+1,n_{1,2},(l_{h,2})_{h \in I(\alpha_{2}+k-S)}}}{\alpha_{2}!}
\Pi_{h \in I(\alpha) \setminus I(\alpha_{2}+k-S)} \delta_{0,l_{h,2}}\\
+ \sum_{j \in I(\alpha_{2}+k-S)}\sum_{\scriptscriptstyle{\stackrel{n_{0,1}+n_{0,2}+n_{0,3}=n_{0},
n_{1,1}+n_{1,2}+n_{1,3}=n_{1}}{l_{h,1}+l_{h,2}+l_{h,3}=l_{h},h \in I(\alpha)}}}
\frac{n_{0}!n_{1}!\Pi_{h \in I(\alpha)}l_{h}!}{n_{0,1}!n_{0,2}!n_{0,3}!n_{1,1}!n_{1,2}!n_{1,3}!
\Pi_{h \in I(\alpha)}l_{h,1}!l_{h,2}!l_{h,3}!}\\
\frac{b_{1,k,\alpha_{1},n_{0,1},n_{1,1},l_{0,1}}}{\alpha_{1}!}\Pi_{h \in I(\alpha)\setminus\{0 \}} \delta_{0,l_{h,1}} \times
A_{j,\alpha_{2}+k-S+1,n_{0,2},n_{1,2},(l_{h,2})_{h \in I(\alpha_{2}+k-S+1)}}\\
\times \Pi_{h \in I(\alpha) \setminus I(\alpha_{2}+k-S+1)} \delta_{0,l_{h,2}} \times
\frac{\varphi_{\alpha_{2}+k-S,n_{0,3},n_{1,3},(l_{h,3})_{h \in I(\alpha_{2}+k-S),h \neq j},l_{j,3}+1}}{\alpha_{2}!} \times
\Pi_{h \in I(\alpha) \setminus I(\alpha_{2}+k-S)} \delta_{0,l_{h,3}}\\
+ \sum_{k \in \mathcal{S}} \sum_{\scriptscriptstyle{\stackrel{\alpha_{1}+\alpha_{2} = \alpha}{\alpha_{2} \geq S-k}}}
\sum_{\scriptscriptstyle{\stackrel{n_{0,1}+n_{0,2}=n_{0},n_{1,1}+n_{1,2}=n_{1}}{l_{h,1}+l_{h,2}=l_{h},h \in I(\alpha)}}}
\frac{n_{0}!n_{1}!\Pi_{h \in I(\alpha)}l_{h}!}{n_{0,1}!n_{0,2}!n_{1,1}!n_{1,2}!\Pi_{h \in I(\alpha)}l_{h,1}!l_{h,2}!}\\
\frac{b_{2,k,\alpha_{1},n_{0,1},n_{1,1},l_{0,1}}}{\alpha_{1}!}\Pi_{h \in I(\alpha)\setminus\{0 \}} \delta_{0,l_{h,1}} \times
\frac{\varphi_{\alpha_{2}+k-S,n_{0,2},n_{1,2}+1,(l_{h,2})_{h \in I(\alpha_{2}+k-S)}}}{\alpha_{2}!}
\Pi_{h \in I(\alpha)\setminus I(\alpha_{2}+k-S) } \delta_{0,l_{h,2}}\\
+ \sum_{j \in I(\alpha_{2}+k-S)}\sum_{\scriptscriptstyle{\stackrel{n_{0,1}+n_{0,2}+n_{0,3}=n_{0},
n_{1,1}+n_{1,2}+n_{1,3}=n_{1}}{l_{h,1}+l_{h,2}+l_{h,3}=l_{h},h \in I(\alpha)}}}
\frac{n_{0}!n_{1}!\Pi_{h \in I(\alpha)}l_{h}!}{n_{0,1}!n_{0,2}!n_{0,3}!n_{1,1}!n_{1,2}!n_{1,3}!
\Pi_{h \in I(\alpha)}l_{h,1}!l_{h,2}!l_{h,3}!}\\
\frac{b_{2,k,\alpha_{1},n_{0,1},n_{1,1},l_{0,1}}}{\alpha_{1}!}\Pi_{h \in I(\alpha)\setminus\{0 \}} \delta_{0,l_{h,1}} \times
B_{j,\alpha_{2}+k-S+1,n_{0,2},n_{1,2},(l_{h,2})_{h \in I(\alpha_{2}+k-S+1)}}\\
\times \Pi_{h \in I(\alpha) \setminus I(\alpha_{2}+k-S+1)} \delta_{0,l_{h,2}} \times
\frac{\varphi_{\alpha_{2}+k-S,n_{0,3},n_{1,3},(l_{h,3})_{h \in I(\alpha_{2}+k-S),h \neq j},l_{j,3}+1}}{\alpha_{2}!} \times
\Pi_{h \in I(\alpha) \setminus I(\alpha_{2}+k-S)} \delta_{0,l_{h,3}} \\
+ \sum_{k \in \mathcal{S}} \sum_{\scriptscriptstyle{
\stackrel{\alpha_{1}+\alpha_{2} = \alpha}{\alpha_{2} \geq S-k}}}
\sum_{\scriptscriptstyle{\stackrel{n_{0,1}+n_{0,2}=n_{0},n_{1,1}+n_{1,2}=n_{1}}{l_{h,1}+l_{h,2}=l_{h},h \in I(\alpha)}}}
\frac{n_{0}!n_{1}!\Pi_{h \in I(\alpha)}l_{h}!}{n_{0,1}!n_{0,2}!n_{1,1}!n_{1,2}!\Pi_{h \in I(\alpha)}l_{h,1}!l_{h,2}!}\\
\frac{b_{3,k,\alpha_{1},n_{0,1},n_{1,1},l_{0,1}}}{\alpha_{1}!}\Pi_{h \in I(\alpha)\setminus\{0 \}} \delta_{0,l_{h,1}} \times
\frac{\varphi_{\alpha_{2}+k-S,n_{0,2},n_{1,2},(l_{h,2})_{h \in I(\alpha_{2}+k-S)}}}{\alpha_{2}!}
\Pi_{h \in I(\alpha) \setminus I(\alpha_{2}+k-S)} \delta_{0,l_{h,2}}\\
+ \frac{\tilde{\omega}_{\alpha,n_{0},n_{1},(l_{h})_{h \in I(\alpha)}}}{\alpha!} \label{ineq_recursion_varphi_alpha_n_l}
\end{multline}
for all $\alpha \geq 0$, all $n_{0},n_{1},l_{h} \geq 0$ for $h \in I(\alpha)$.
\end{prop}
\begin{proof} In order to get the inequality (\ref{ineq_recursion_varphi_alpha_n_l}), we apply the differential operator
$\partial_{v_0}^{n_0}\partial_{v_1}^{n_1}\Pi_{h \in I(\alpha)}\partial_{u_h}^{l_h}$ on the left and right handside of the
recursion (\ref{recursion_phi_alpha}) and we use the expansions that are computed below.

>From the Leibniz formula, we deduce that
\begin{multline}
\partial_{v_0}^{n_0}\partial_{v_1}^{n_1} \Pi_{h \in I(\alpha)}\partial_{u_h}^{l_h}(b_{3,k,\alpha_{1}}(v_{0},v_{1},u_{0})
\phi_{\alpha_{2}+k-S}(v_{0},v_{1},(u_{h})_{h \in I(\alpha_{2}+k-S)}))=\\
\sum_{\scriptscriptstyle{\stackrel{n_{0,1}+n_{0,2}=n_{0},n_{1,1}+n_{1,2}=n_{1}}{l_{h,1}+l_{h,2}=l_{h},h \in I(\alpha)}}}
\frac{n_{0}!n_{1}!\Pi_{h \in I(\alpha)}l_{h}!}{n_{0,1}!n_{0,2}!n_{1,1}!n_{1,2}!\Pi_{h \in I(\alpha)}l_{h,1}!l_{h,2}!}
\partial_{v_0}^{n_{0,1}}\partial_{v_1}^{n_{1,1}}\Pi_{h \in I(\alpha)} \partial_{u_h}^{l_{h,1}}(b_{3,k,\alpha_{1}}(v_{0},v_{1},u_{0}))\\
\times
\partial_{v_0}^{n_{0,2}}\partial_{v_1}^{n_{1,2}}\Pi_{h \in I(\alpha)} \partial_{u_h}^{l_{h,2}}
(\phi_{\alpha_{2}+k-S}(v_{0},v_{1},(u_{h})_{h \in I(\alpha_{2}+k-S)}))
\end{multline}
and
\begin{multline}
\partial_{v_0}^{n_0}\partial_{v_1}^{n_1} \Pi_{h \in I(\alpha)}\partial_{u_h}^{l_h}(b_{1,k,\alpha_{1}}(v_{0},v_{1},u_{0})
\partial_{v_0}\phi_{\alpha_{2}+k-S}(v_{0},v_{1},(u_{h})_{h \in I(\alpha_{2}+k-S)}))=\\
\sum_{\scriptscriptstyle{\stackrel{n_{0,1}+n_{0,2}=n_{0},n_{1,1}+n_{1,2}=n_{1}}{l_{h,1}+l_{h,2}=l_{h},h \in I(\alpha)}}}
\frac{n_{0}!n_{1}!\Pi_{h \in I(\alpha)}l_{h}!}{n_{0,1}!n_{0,2}!n_{1,1}!n_{1,2}!\Pi_{h \in I(\alpha)}l_{h,1}!l_{h,2}!}
\partial_{v_0}^{n_{0,1}}\partial_{v_1}^{n_{1,1}}\Pi_{h \in I(\alpha)} \partial_{u_h}^{l_{h,1}}(b_{1,k,\alpha_{1}}(v_{0},v_{1},u_{0}))\\
\times
\partial_{v_0}^{n_{0,2}+1}\partial_{v_1}^{n_{1,2}}\Pi_{h \in I(\alpha)} \partial_{u_h}^{l_{h,2}}
(\phi_{\alpha_{2}+k-S}(v_{0},v_{1},(u_{h})_{h \in I(\alpha_{2}+k-S)})).
\end{multline}
Moreover, we can write
\begin{multline}
 \partial_{v_0}^{n_{0,1}}\partial_{v_1}^{n_{1,1}}\Pi_{h \in I(\alpha)} \partial_{u_h}^{l_{h,1}}(b_{3,k,\alpha_{1}}(v_{0},v_{1},u_{0}))\\
= \partial_{v_0}^{n_{0,1}}\partial_{v_1}^{n_{1,1}}\partial_{u_0}^{l_{0,1}}b_{3,k,\alpha_{1}}(v_{0},v_{1},u_{0}) \times
\Pi_{h \in I(\alpha) \setminus \{ 0 \}} \delta_{0,l_{h,1}} \label{derivative_b3kalpha1}
\end{multline}
with
\begin{multline}
\partial_{v_0}^{n_{0,2}}\partial_{v_1}^{n_{1,2}}\Pi_{h \in I(\alpha)} \partial_{u_h}^{l_{h,2}}
(\phi_{\alpha_{2}+k-S}(v_{0},v_{1},(u_{h})_{h \in I(\alpha_{2}+k-S)}))\\
= \partial_{v_0}^{n_{0,2}}\partial_{v_1}^{n_{1,2}}\Pi_{h \in I(\alpha_{2}+k-S)} \partial_{u_h}^{l_{h,2}}
(\phi_{\alpha_{2}+k-S}(v_{0},v_{1},(u_{h})_{h \in I(\alpha_{2}+k-S)}))\\
\times \Pi_{h \in I(\alpha) \setminus I(\alpha_{2}+k-S)} \delta_{0,l_{h,2}}
\end{multline}
and
\begin{multline}
 \partial_{v_0}^{n_{0,1}}\partial_{v_1}^{n_{1,1}}\Pi_{h \in I(\alpha)} \partial_{u_h}^{l_{h,1}}(b_{1,k,\alpha_{1}}(v_{0},v_{1},u_{0}))\\
= \partial_{v_0}^{n_{0,1}}\partial_{v_1}^{n_{1,1}}\partial_{u_0}^{l_{0,1}}b_{1,k,\alpha_{1}}(v_{0},v_{1},u_{0}) \times
\Pi_{h \in I(\alpha) \setminus \{ 0 \}} \delta_{0,l_{h,1}} \label{derivative_b1kalpha1}
\end{multline}
with
\begin{multline}
\partial_{v_0}^{n_{0,2}+1}\partial_{v_1}^{n_{1,2}}\Pi_{h \in I(\alpha)} \partial_{u_h}^{l_{h,2}}
(\phi_{\alpha_{2}+k-S}(v_{0},v_{1},(u_{h})_{h \in I(\alpha_{2}+k-S)}))\\
= \partial_{v_0}^{n_{0,2}+1}\partial_{v_1}^{n_{1,2}}\Pi_{h \in I(\alpha_{2}+k-S)} \partial_{u_h}^{l_{h,2}}
(\phi_{\alpha_{2}+k-S}(v_{0},v_{1},(u_{h})_{h \in I(\alpha_{2}+k-S)}))\\
\times \Pi_{h \in I(\alpha) \setminus I(\alpha_{2}+k-S)} \delta_{0,l_{h,2}}.
\end{multline}
By construction, we have 
\begin{multline}
A_{j}(v_{0},v_{1},(u_{h})_{h \in I(\alpha_{2}+k-S+1)})
=  \sum_{l_{1}+l_{2}=j}
\frac{\partial_{v_1}^{l_1}a(v_{0},v_{1})}{l_{1}!\nu^{l_1}}(l_{2}+1)\nu
u_{l_{2}+1} \\
 + \sum_{p=0}^{d} \sum_{j_{0}+\ldots+j_{p}=j} \frac{\partial_{v_1}^{j_0}a_{p}(v_{0},v_{1})}{j_{0}!\nu^{j_0}}
\Pi_{l=1}^{p} u_{j_{l}}
\end{multline}
for all $j \in I(\alpha_{2}+k-S)$. Again, by the Leibniz formula, we get that
\begin{multline}
\partial_{v_0}^{n_0}\partial_{v_1}^{n_1} \Pi_{h \in I(\alpha)}\partial_{u_h}^{l_h}(b_{1,k,\alpha_{1}}(v_{0},v_{1},u_{0})
A_{j}(v_{0},v_{1},(u_{h})_{h \in I(\alpha_{2}+k-S+1)}) \\
\times \partial_{u_j}\phi_{\alpha_{2}+k-S}(v_{0},v_{1},(u_{h})_{h \in I(\alpha_{2}+k-S)})) \\
=
\sum_{\scriptscriptstyle{\stackrel{n_{0,1}+n_{0,2}+n_{0,3}=n_{0},
n_{1,1}+n_{1,2}+n_{1,3}=n_{1}}{l_{h,1}+l_{h,2}+l_{h,3}=l_{h},h \in I(\alpha)}}}
\frac{n_{0}!n_{1}!\Pi_{h \in I(\alpha)}l_{h}!}{n_{0,1}!n_{0,2}!n_{0,3}!n_{1,1}!n_{1,2}!n_{1,3}!
\Pi_{h \in I(\alpha)}l_{h,1}!l_{h,2}!l_{h,3}!}\\
\partial_{v_0}^{n_{0,1}}\partial_{v_1}^{n_{1,1}}\Pi_{h \in I(\alpha)} \partial_{u_h}^{l_{h,1}}(b_{1,k,\alpha_{1}}(v_{0},v_{1},u_{0}))\\
\times \partial_{v_0}^{n_{0,2}}\partial_{v_1}^{n_{1,2}}\Pi_{h \in I(\alpha)}
\partial_{u_h}^{l_{h,2}}(A_{j}(v_{0},v_{1},(u_{h})_{h \in I(\alpha_{2}+k-S+1)}))\\
\times \partial_{v_0}^{n_{0,3}}\partial_{v_1}^{n_{1,3}}(\Pi_{h \in I(\alpha),h \neq j}
\partial_{u_h}^{l_{h,3}})\partial_{u_j}^{l_{j,3}+1}\phi_{\alpha_{2}+k-S}(v_{0},v_{1},(u_{h})_{h \in I(\alpha_{2}+k-S)}).
\label{derivative_triple_product_Aj}
\end{multline}
Inside the formula (\ref{derivative_triple_product_Aj}), we can rewrite the relations (\ref{derivative_b1kalpha1}) and
\begin{multline}
 \partial_{v_0}^{n_{0,2}}\partial_{v_1}^{n_{1,2}}\Pi_{h \in I(\alpha)}
\partial_{u_h}^{l_{h,2}}A_{j}(v_{0},v_{1},(u_{h})_{h \in I(\alpha_{2}+k-S+1)}) \\
=
\partial_{v_0}^{n_{0,2}}\partial_{v_1}^{n_{1,2}}\Pi_{h \in I(\alpha_{2}+k-S+1)}
\partial_{u_h}^{l_{h,2}}A_{j}(v_{0},v_{1},(u_{h})_{h \in I(\alpha_{2}+k-S+1)}) \times \Pi_{h \in I(\alpha) \setminus I(\alpha_{2}+k-S+1)}
\delta_{0,l_{h,2}}
\end{multline}
with
\begin{multline}
  \partial_{v_0}^{n_{0,3}}\partial_{v_1}^{n_{1,3}}(\Pi_{h \in I(\alpha),h \neq j}
\partial_{u_h}^{l_{h,3}})\partial_{u_j}^{l_{j,3}+1}\phi_{\alpha_{2}+k-S}(v_{0},v_{1},(u_{h})_{h \in I(\alpha_{2}+k-S)}) \\
=  \partial_{v_0}^{n_{0,3}}\partial_{v_1}^{n_{1,3}}(\Pi_{h \in I(\alpha_{2}+k-S),h \neq j}
\partial_{u_h}^{l_{h,3}})\partial_{u_j}^{l_{j,3}+1}\phi_{\alpha_{2}+k-S}(v_{0},v_{1},(u_{h})_{h \in I(\alpha_{2}+k-S)})\\
\times \Pi_{h \in I(\alpha) \setminus I(\alpha_{2}+k-S)} \delta_{0,l_{h,3}}. \label{derivative_partial_uj_phi_alpha2kS}
\end{multline}
In the same way, one gets the next equalities
\begin{multline}
\partial_{v_0}^{n_0}\partial_{v_1}^{n_1} \Pi_{h \in I(\alpha)}\partial_{u_h}^{l_h}(b_{2,k,\alpha_{1}}(v_{0},v_{1},u_{0})
\partial_{v_1}\phi_{\alpha_{2}+k-S}(v_{0},v_{1},(u_{h})_{h \in I(\alpha_{2}+k-S)}))=\\
\sum_{\scriptscriptstyle{\stackrel{n_{0,1}+n_{0,2}=n_{0},n_{1,1}+n_{1,2}=n_{1}}{l_{h,1}+l_{h,2}=l_{h},h \in I(\alpha)}}}
\frac{n_{0}!n_{1}!\Pi_{h \in I(\alpha)}l_{h}!}{n_{0,1}!n_{0,2}!n_{1,1}!n_{1,2}!\Pi_{h \in I(\alpha)}l_{h,1}!l_{h,2}!}
\partial_{v_0}^{n_{0,1}}\partial_{v_1}^{n_{1,1}}\Pi_{h \in I(\alpha)} \partial_{u_h}^{l_{h,1}}(b_{2,k,\alpha_{1}}(v_{0},v_{1},u_{0}))\\
\times
\partial_{v_0}^{n_{0,2}}\partial_{v_1}^{n_{1,2}+1}\Pi_{h \in I(\alpha)} \partial_{u_h}^{l_{h,2}}
(\phi_{\alpha_{2}+k-S}(v_{0},v_{1},(u_{h})_{h \in I(\alpha_{2}+k-S)}))
\end{multline}
with the factorizations
\begin{multline}
 \partial_{v_0}^{n_{0,1}}\partial_{v_1}^{n_{1,1}}\Pi_{h \in I(\alpha)} \partial_{u_h}^{l_{h,1}}(b_{2,k,\alpha_{1}}(v_{0},v_{1},u_{0}))\\
= \partial_{v_0}^{n_{0,1}}\partial_{v_1}^{n_{1,1}}\partial_{u_0}^{l_{0,1}}b_{2,k,\alpha_{1}}(v_{0},v_{1},u_{0}) \times
\Pi_{h \in I(\alpha) \setminus \{ 0 \}} \delta_{0,l_{h,1}} \label{derivative_b2kalpha1}
\end{multline}
and
\begin{multline}
\partial_{v_0}^{n_{0,2}}\partial_{v_1}^{n_{1,2}+1}\Pi_{h \in I(\alpha)} \partial_{u_h}^{l_{h,2}}
(\phi_{\alpha_{2}+k-S}(v_{0},v_{1},(u_{h})_{h \in I(\alpha_{2}+k-S)}))\\
= \partial_{v_0}^{n_{0,2}}\partial_{v_1}^{n_{1,2}+1}\Pi_{h \in I(\alpha_{2}+k-S)} \partial_{u_h}^{l_{h,2}}
(\phi_{\alpha_{2}+k-S}(v_{0},v_{1},(u_{h})_{h \in I(\alpha_{2}+k-S)}))\\
\times \Pi_{h \in I(\alpha) \setminus I(\alpha_{2}+k-S)} \delta_{0,l_{h,2}}.
\end{multline}
We recall that
\begin{equation}
B_{j}(v_{0},v_{1},(u_{h})_{h \in I(\alpha_{2}+k-S+1)}) = (j+1)\nu u_{j+1}
\end{equation}
for all $j \in I(\alpha_{2}+k-S)$ and we deduce that
\begin{multline}
\partial_{v_0}^{n_0}\partial_{v_1}^{n_1} \Pi_{h \in I(\alpha)}\partial_{u_h}^{l_h}(b_{2,k,\alpha_{1}}(v_{0},v_{1},u_{0})
B_{j}(v_{0},v_{1},(u_{h})_{h \in I(\alpha_{2}+k-S+1)}) \\
\times \partial_{u_j}\phi_{\alpha_{2}+k-S}(v_{0},v_{1},(u_{h})_{h \in I(\alpha_{2}+k-S)})) \\
=
\sum_{\scriptscriptstyle{\stackrel{n_{0,1}+n_{0,2}+n_{0,3}=n_{0},
n_{1,1}+n_{1,2}+n_{1,3}=n_{1}}{l_{h,1}+l_{h,2}+l_{h,3}=l_{h},h \in I(\alpha)}}}
\frac{n_{0}!n_{1}!\Pi_{h \in I(\alpha)}l_{h}!}{n_{0,1}!n_{0,2}!n_{0,3}!n_{1,1}!n_{1,2}!n_{1,3}!
\Pi_{h \in I(\alpha)}l_{h,1}!l_{h,2}!l_{h,3}!}\\
\partial_{v_0}^{n_{0,1}}\partial_{v_1}^{n_{1,1}}\Pi_{h \in I(\alpha)} \partial_{u_h}^{l_{h,1}}(b_{2,k,\alpha_{1}}(v_{0},v_{1},u_{0}))\\
\times \partial_{v_0}^{n_{0,2}}\partial_{v_1}^{n_{1,2}}\Pi_{h \in I(\alpha)}
\partial_{u_h}^{l_{h,2}}(B_{j}(v_{0},v_{1},(u_{h})_{h \in I(\alpha_{2}+k-S+1)}))\\
\times \partial_{v_0}^{n_{0,3}}\partial_{v_1}^{n_{1,3}}(\Pi_{h \in I(\alpha),h \neq j}
\partial_{u_h}^{l_{h,3}})\partial_{u_j}^{l_{j,3}+1}\phi_{\alpha_{2}+k-S}(v_{0},v_{1},(u_{h})_{h \in I(\alpha_{2}+k-S)}).
\label{derivative_triple_product_Bj}
\end{multline}
Inside the formula (\ref{derivative_triple_product_Bj}), we can rewrite the relations (\ref{derivative_b2kalpha1}) and
\begin{multline}
 \partial_{v_0}^{n_{0,2}}\partial_{v_1}^{n_{1,2}}\Pi_{h \in I(\alpha)}
\partial_{u_h}^{l_{h,2}}B_{j}(v_{0},v_{1},(u_{h})_{h \in I(\alpha_{2}+k-S+1)}) \\
=
\partial_{v_0}^{n_{0,2}}\partial_{v_1}^{n_{1,2}}\Pi_{h \in I(\alpha_{2}+k-S+1)}
\partial_{u_h}^{l_{h,2}}B_{j}(v_{0},v_{1},(u_{h})_{h \in I(\alpha_{2}+k-S+1)}) \times \Pi_{h \in I(\alpha) \setminus I(\alpha_{2}+k-S+1)}
\delta_{0,l_{h,2}}
\end{multline}
with the factorization (\ref{derivative_partial_uj_phi_alpha2kS}).
\end{proof}

\subsection{Majorant series and a functional equation with infinitely many variables}

\begin{defin}
We denote $\mathbb{G}[[V_{0},V_{1},(U_{h})_{h \geq 0},W]]$ the vector space of formal series in the variables
$V_{0},V_{1},(U_{h})_{h \geq 0},W$ of the form
\begin{equation}
\Psi(V_{0},V_{1},(U_{h})_{h \geq 0},W) =
\sum_{\alpha \geq 0} \Psi_{\alpha}(V_{0},V_{1},(U_{h})_{h \in I(\alpha)})
\frac{W^\alpha}{\alpha!}
\end{equation}
where $\Psi_{\alpha} \in \mathbb{C}[[V_{0},V_{1},(U_{h})_{h \in I(\alpha)}]]$ for all $\alpha \geq 0$.
\end{defin}
We keep the notations of the previous section and we introduce the following formal series:
\begin{multline}
B_{m,k}(V_{0},V_{1},U_{0},W) = \sum_{\alpha \geq 0}
\left( \sum_{n_{0},n_{1},l_{0} \geq 0} b_{m,k,\alpha,n_{0},n_{1},l_{0}}\frac{V_{0}^{n_0}}{n_{0}!}
\frac{V_{1}^{n_1}}{n_{1}!}\frac{U_{0}^{l_0}}{l_{0}!} \right) \frac{W^{\alpha}}{\alpha!},\\
\tilde{\Omega}(V_{0},V_{1},(U_{h})_{h \geq 0},W) = \sum_{\alpha \geq 0}
\left( \sum_{n_{0},n_{1},l_{h} \geq 0, h \in I(\alpha)} \tilde{\omega}_{\alpha,n_{0},n_{1},(l_{h})_{h \in I(\alpha)}}
\frac{V_{0}^{n_0}}{n_{0}!}\frac{V_{1}^{n_1}}{n_{1}!}\Pi_{h \in I(\alpha)}\frac{U_{h}^{l_h}}{l_{h}!} \right)
\frac{W^{\alpha}}{\alpha!}
\label{B_mk_V_U_W_and_tildeOmega_V_U_W_defin}
\end{multline}
for $m=1,2,3$, all $k \in \mathcal{S}$, and
\begin{multline*}
\mathbf{A}_{j,\alpha}(V_{0},V_{1},(U_{h})_{h \in I(\alpha)}) =
\sum_{n_{0},n_{1},l_{h} \geq 0, h \in I(\alpha)} A_{j,\alpha,n_{0},n_{1},(l_{h})_{h \in I(\alpha)}}
\frac{V_{0}^{n_0}}{n_{0}!}\frac{V_{1}^{n_1}}{n_{1}!}\Pi_{h \in I(\alpha)} \frac{U_{h}^{l_h}}{l_{h}!},\\
\mathbf{B}_{j,\alpha}(V_{0},V_{1},(U_{h})_{h \in I(\alpha)}) =
\sum_{n_{0},n_{1},l_{h} \geq 0, h \in I(\alpha)} B_{j,\alpha,n_{0},n_{1},(l_{h})_{h \in I(\alpha)}}
\frac{V_{0}^{n_0}}{n_{0}!}\frac{V_{1}^{n_1}}{n_{1}!}\Pi_{h \in I(\alpha)} \frac{U_{h}^{l_h}}{l_{h}!}
\end{multline*}
for all $\alpha \geq 0$, all $j \in I(\alpha)$. We also introduce the following linear operators acting on
$\mathbb{G}[[V_{0},V_{1},(U_{h})_{h \geq 0},W]]$. Let
\begin{multline*}
\mathbb{D}_{\mathbf{A}}\Psi(V_{0},V_{1},(U_{h})_{h \geq 0},W)\\
= \sum_{\alpha \geq 0} (\sum_{j \in I(\alpha)} 
\mathbf{A}_{j,\alpha+1}(V_{0},V_{1},(U_{h})_{h \in I(\alpha+1)})(\partial_{U_j}\Psi_{\alpha})(V_{0},V_{1},(U_{h})_{h \in I(\alpha)}))
\frac{W^{\alpha}}{\alpha!},\\
\mathbb{D}_{\mathbf{B}}\Psi(V_{0},V_{1},(U_{h})_{h \geq 0},W) \\
= \sum_{\alpha \geq 0} (\sum_{j \in I(\alpha)}
\mathbf{B}_{j,\alpha+1}(V_{0},V_{1},(U_{h})_{h \in I(\alpha+1)})(\partial_{U_j}\Psi_{\alpha})(V_{0},V_{1},(U_{h})_{h \in I(\alpha)}))
\frac{W^{\alpha}}{\alpha!}
\end{multline*}
for all $\Psi \in \mathbb{G}[[V_{0},V_{1},(U_{h})_{h \geq 0},W]]$. We stress the fact that although these operators act on
$\mathbb{G}[[V_{0},V_{1},(U_{h})_{h \geq 0},W]]$ their image does not have to belong to this space.
\begin{prop} A formal series
$$ \Psi(V_{0},V_{1},(U_{h})_{h \geq 0},W) = \sum_{\alpha \geq 0}
\left( \sum_{n_{0},n_{1},l_{h} \geq 0, h \in I(\alpha)} \psi_{\alpha,n_{0},n_{1},(l_{h})_{h \in I(\alpha)}}
\frac{V_{0}^{n_0}}{n_{0}!}\frac{V_{1}^{n_1}}{n_{1}!}\Pi_{h \in I(\alpha)}\frac{U_{h}^{l_h}}{l_{h}!} \right)
\frac{W^{\alpha}}{\alpha!} $$
satisfies the following functional equation
\begin{multline}
\Psi(V_{0},V_{1},(U_{h})_{h \geq 0},W)
= \sum_{k \in \mathcal{S}} ( B_{1,k}(V_{0},V_{1},U_{0},W)\partial_{W}^{-S+k}\partial_{V_0}
\Psi(V_{0},V_{1},(U_{h})_{h \geq 0},W)\\
+ B_{1,k}(V_{0},V_{1},U_{0},W)\partial_{W}^{-S+k}\mathbb{D}_{\mathbf{A}}\Psi(V_{0},V_{1},(U_{h})_{h \geq 0},W) )\\
+ \sum_{k \in \mathcal{S}} ( B_{2,k}(V_{0},V_{1},U_{0},W)\partial_{W}^{-S+k}\partial_{V_1}
\Psi(V_{0},V_{1},(U_{h})_{h \geq 0},W)\\
+ B_{2,k}(V_{0},V_{1},U_{0},W)\partial_{W}^{-S+k}\mathbb{D}_{\mathbf{B}}\Psi(V_{0},V_{1},(U_{h})_{h \geq 0},W) )\\
+ \sum_{k \in \mathcal{S}} B_{3,k}(V_{0},V_{1},U_{0},W)\partial_{W}^{-S+k}\Psi(V_{0},V_{1},(U_{h})_{h \geq 0},W) \\+
\tilde{\Omega}(V_{0},V_{1},(U_{h})_{h \geq 0},W) \label{functional_eq_defin_psi}
\end{multline}
if and only if its coefficients $\psi_{\alpha,n_{0},n_{1},(l_{h})_{h \in I(\alpha)}}$ satisfy the following recursion
\begin{multline}
\frac{\psi_{\alpha,n_{0},n_{1},(l_{h})_{h \in I(\alpha)}}}{\alpha!} =  \sum_{k \in \mathcal{S}} \sum_{\scriptscriptstyle{
\stackrel{\alpha_{1}+\alpha_{2} = \alpha}{\alpha_{2} \geq S-k}}}
\sum_{\scriptscriptstyle{\stackrel{n_{0,1}+n_{0,2}=n_{0},n_{1,1}+n_{1,2}=n_{1}}{l_{h,1}+l_{h,2}=l_{h},h \in I(\alpha)}}}
\frac{n_{0}!n_{1}!\Pi_{h \in I(\alpha)}l_{h}!}{n_{0,1}!n_{0,2}!n_{1,1}!n_{1,2}!\Pi_{h \in I(\alpha)}l_{h,1}!l_{h,2}!}\\
\frac{b_{1,k,\alpha_{1},n_{0,1},n_{1,1},l_{0,1}}}{\alpha_{1}!}\Pi_{h \in I(\alpha)\setminus\{0 \}} \delta_{0,l_{h,1}} \times
\frac{\psi_{\alpha_{2}+k-S,n_{0,2}+1,n_{1,2},(l_{h,2})_{h \in I(\alpha_{2}+k-S)}}}{\alpha_{2}!}
\Pi_{h \in I(\alpha)\setminus I(\alpha_{2}+k-S)} \delta_{0,l_{h,2}}\\
+ \sum_{j \in I(\alpha_{2}+k-S)}\sum_{\scriptscriptstyle{\stackrel{n_{0,1}+n_{0,2}+n_{0,3}=n_{0},
n_{1,1}+n_{1,2}+n_{1,3}=n_{1}}{l_{h,1}+l_{h,2}+l_{h,3}=l_{h},h \in I(\alpha)}}}
\frac{n_{0}!n_{1}!\Pi_{h \in I(\alpha)}l_{h}!}{n_{0,1}!n_{0,2}!n_{0,3}!n_{1,1}!n_{1,2}!n_{1,3}!
\Pi_{h \in I(\alpha)}l_{h,1}!l_{h,2}!l_{h,3}!}\\
\frac{b_{1,k,\alpha_{1},n_{0,1},n_{1,1},l_{0,1}}}{\alpha_{1}!}\Pi_{h \in I(\alpha)\setminus\{0 \}} \delta_{0,l_{h,1}} \times
A_{j,\alpha_{2}+k-S+1,n_{0,2},n_{1,2},(l_{h,2})_{h \in I(\alpha_{2}+k-S+1)}}\\
\times \Pi_{h \in I(\alpha) \setminus I(\alpha_{2}+k-S+1)} \delta_{0,l_{h,2}} \times
\frac{\psi_{\alpha_{2}+k-S,n_{0,3},n_{1,3},(l_{h,3})_{h \in I(\alpha_{2}+k-S),h \neq j},l_{j,3}+1}}{\alpha_{2}!} \times
\Pi_{h \in I(\alpha) \setminus I(\alpha_{2}+k-S)} \delta_{0,l_{h,3}}\\
+ \sum_{k \in \mathcal{S}} \sum_{\scriptscriptstyle{\stackrel{\alpha_{1}+\alpha_{2} = \alpha}{\alpha_{2} \geq S-k}}}
\sum_{\scriptscriptstyle{\stackrel{n_{0,1}+n_{0,2}=n_{0},n_{1,1}+n_{1,2}=n_{1}}{l_{h,1}+l_{h,2}=l_{h},h \in I(\alpha)}}}
\frac{n_{0}!n_{1}!\Pi_{h \in I(\alpha)}l_{h}!}{n_{0,1}!n_{0,2}!n_{1,1}!n_{1,2}!\Pi_{h \in I(\alpha)}l_{h,1}!l_{h,2}!}\\
\frac{b_{2,k,\alpha_{1},n_{0,1},n_{1,1},l_{0,1}}}{\alpha_{1}!}\Pi_{h \in I(\alpha)\setminus\{0 \}} \delta_{0,l_{h,1}} \times
\frac{\psi_{\alpha_{2}+k-S,n_{0,2},n_{1,2}+1,(l_{h,2})_{h \in I(\alpha_{2}+k-S)}}}{\alpha_{2}!}
\Pi_{h \in I(\alpha) \setminus I(\alpha_{2}+k-S)} \delta_{0,l_{h,2}}\\
+ \sum_{j \in I(\alpha_{2}+k-S)}\sum_{\scriptscriptstyle{\stackrel{n_{0,1}+n_{0,2}+n_{0,3}=n_{0},
n_{1,1}+n_{1,2}+n_{1,3}=n_{1}}{l_{h,1}+l_{h,2}+l_{h,3}=l_{h},h \in I(\alpha)}}}
\frac{n_{0}!n_{1}!\Pi_{h \in I(\alpha)}l_{h}!}{n_{0,1}!n_{0,2}!n_{0,3}!n_{1,1}!n_{1,2}!n_{1,3}!
\Pi_{h \in I(\alpha)}l_{h,1}!l_{h,2}!l_{h,3}!}\\
\frac{b_{2,k,\alpha_{1},n_{0,1},n_{1,1},l_{0,1}}}{\alpha_{1}!}\Pi_{h \in I(\alpha)\setminus\{0 \}} \delta_{0,l_{h,1}} \times
B_{j,\alpha_{2}+k-S+1,n_{0,2},n_{1,2},(l_{h,2})_{h \in I(\alpha_{2}+k-S+1)}}\\
\times \Pi_{h \in I(\alpha) \setminus I(\alpha_{2}+k-S+1)} \delta_{0,l_{h,2}} \times
\frac{\psi_{\alpha_{2}+k-S,n_{0,3},n_{1,3},(l_{h,3})_{h \in I(\alpha_{2}+k-S),h \neq j},l_{j,3}+1}}{\alpha_{2}!} \times
\Pi_{h \in I(\alpha) \setminus I(\alpha_{2}+k-S)} \delta_{0,l_{h,3}} \\
+  \sum_{k \in \mathcal{S}} \sum_{\scriptscriptstyle{
\stackrel{\alpha_{1}+\alpha_{2} = \alpha}{\alpha_{2} \geq S-k}}}
\sum_{\scriptscriptstyle{\stackrel{n_{0,1}+n_{0,2}=n_{0},n_{1,1}+n_{1,2}=n_{1}}{l_{h,1}+l_{h,2}=l_{h},h \in I(\alpha)}}}
\frac{n_{0}!n_{1}!\Pi_{h \in I(\alpha)}l_{h}!}{n_{0,1}!n_{0,2}!n_{1,1}!n_{1,2}!\Pi_{h \in I(\alpha)}l_{h,1}!l_{h,2}!}\\
\frac{b_{3,k,\alpha_{1},n_{0,1},n_{1,1},l_{0,1}}}{\alpha_{1}!}\Pi_{h \in I(\alpha)\setminus\{0 \}} \delta_{0,l_{h,1}} \times
\frac{\psi_{\alpha_{2}+k-S,n_{0,2},n_{1,2},(l_{h,2})_{h \in I(\alpha_{2}+k-S)}}}{\alpha_{2}!}
\Pi_{h \in I(\alpha)\setminus I(\alpha_{2}+k-S)} \delta_{0,l_{h,2}}\\
+ \frac{\tilde{\omega}_{\alpha,n_{0},n_{1},(l_{h})_{h \in I(\alpha)}}}{\alpha!} \label{recursion_psii_alpha_n_l}
\end{multline}
for all $\alpha \geq 0$, all $n_{0},n_{1},l_{h} \geq 0$ with $h \in I(\alpha)$.
\end{prop}
\begin{proof} We proceed by identification of the coefficients in the Taylor expansion with respect to the variables
$V_{0},V_{1},(U_{h})_{h \in I(\alpha)}$ and $W$ for all $\alpha \geq 0$. By definition, we have that
$$ B_{1,k}(V_{0},V_{1},U_{0},W) \partial_{W}^{-S+k} \partial_{V_0} \Psi(V_{0},V_{1},(U_{h})_{h \geq 0},W)
= \sum_{\alpha \geq 0} \sum_{\stackrel{\alpha_{1}+\alpha_{2}=\alpha}{\alpha_{2} \geq S-k}}
\mathcal{C}_{\alpha_{1},\alpha_{2}}^{1} W^{\alpha}
$$
where the coefficients $\mathcal{C}_{\alpha_{1},\alpha_{2}}^{1}$ can be rewritten, using the Kronecker symbol $\delta_{0,m}$, in
the form
\begin{multline*}
 \mathcal{C}_{\alpha_{1},\alpha_{2}}^{1} = (\sum_{n_{0},n_{1},l_{h} \geq 0,h \in I(\alpha)}
\frac{b_{1,k,\alpha_{1},n_{0},n_{1},l_{0}}}{\alpha_{1}!}\Pi_{h \in I(\alpha) \setminus \{ 0 \}} \delta_{0,l_h}
\frac{V_{0}^{n_0}}{n_{0}!}\frac{V_{1}^{n_1}}{n_{1}!}\Pi_{h \in I(\alpha)} \frac{U_{h}^{l_h}}{l_{h}!})\\
\times (\sum_{n_{0},n_{1},l_{h} \geq 0, h \in I(\alpha)}
\frac{\psi_{\alpha_{2}+k-S,n_{0}+1,n_{1},(l_{h})_{h \in I(\alpha_{2}+k-S)}}}{\alpha_{2}!}
\Pi_{h \in I(\alpha) \setminus I(\alpha_{2}+k-S)} \delta_{0,l_h}
\frac{V_{0}^{n_0}}{n_{0}!}\frac{V_{1}^{n_1}}{n_{1}!}\Pi_{h \in I(\alpha)}\frac{U_{h}^{l_h}}{l_{h}!})
\end{multline*}
Hence,
\begin{multline}
\mathcal{C}_{\alpha_{1},\alpha_{2}}^{1} = \sum_{n_{0},n_{1},l_{h} \geq 0,h \in I(\alpha)}
( \sum_{\scriptscriptstyle{\stackrel{n_{0,1}+n_{0,2}=n_{0},n_{1,1}+n_{1,2}=n_{1}}{l_{h,1}+l_{h,2}=l_{h},h \in I(\alpha)}}}
\frac{b_{1,k,\alpha_{1},n_{0,1},n_{1,1},l_{0,1}}}{\alpha_{1}!n_{0,1}!n_{1,1}! \Pi_{h \in I(\alpha)}l_{h,1}!}
\Pi_{h \in I(\alpha) \setminus \{ 0 \}} \delta_{0,l_{h,1}}\\
\times \frac{\psi_{\alpha_{2}+k-S,n_{0,2}+1,n_{1,2},
(l_{h,2})_{h \in I(\alpha_{2}+k-S)}}}{\alpha_{2}!n_{0,2}!n_{1,2}! \Pi_{h \in I(\alpha)}l_{h,2}!}
\Pi_{h \in I(\alpha) \setminus I(\alpha_{2}+k-S)} \delta_{0,l_{h,2}}) V_{0}^{n_0}V_{1}^{n_1} \Pi_{h \in I(\alpha)} U_{h}^{l_h}.
\label{mathcalC1}
\end{multline}
We also have that
$$ B_{1,k}(V_{0},V_{1},U_{0},W) \partial_{W}^{-S+k}\mathbb{D}_{\mathbf{A}}\Psi(V_{0},V_{1},(U_{h})_{h \geq 0},W)
= \sum_{\alpha \geq 0} \sum_{\stackrel{\alpha_{1}+\alpha_{2}=\alpha}{\alpha_{2} \geq S-k}}
\mathcal{F}_{\alpha_{1},\alpha_{2}}^{1} W^{\alpha}
$$
where the coefficients $\mathcal{F}_{\alpha_{1},\alpha_{2}}^{1}$ can be rewritten in the form
\begin{multline*}
\mathcal{F}_{\alpha_{1},\alpha_{2}}^{1} = \sum_{j \in I(\alpha_{2}-S+k)}( \sum_{n_{0},n_{1},l_{h} \geq 0,h \in I(\alpha)}
\frac{b_{1,k,\alpha_{1},n_{0},n_{1},l_{0}}}{\alpha_{1}!}\Pi_{h \in I(\alpha) \setminus \{ 0 \}} \delta_{0,l_h}
\frac{V_{0}^{n_0}}{n_{0}!}\frac{V_{1}^{n_1}}{n_{1}!}\Pi_{h \in I(\alpha)} \frac{U_{h}^{l_h}}{l_{h}!} )\\
\times (\sum_{n_{0},n_{1},l_{h} \geq 0,h \in I(\alpha)}
A_{j,\alpha_{2}-S+k+1,n_{0},n_{1},(l_{h})_{h \in I(\alpha_{2}-S+k+1)}}
\Pi_{h \in I(\alpha) \setminus I(\alpha_{2}-S+k+1)} \delta_{0,l_h}\\
\times \frac{V_{0}^{n_0}}{n_{0}!}\frac{V_{1}^{n_1}}{n_{1}!}\Pi_{h \in I(\alpha)} \frac{U_{h}^{l_h}}{l_{h}!} )\\
\times ( \sum_{n_{0},n_{1},l_{h} \geq 0,h \in I(\alpha)}
\frac{\psi_{\alpha_{2}-S+k,n_{0},n_{1},(l_{h})_{h \in I(\alpha_{2}-S+k),h \neq j},l_{j}+1}}{\alpha_{2}!}
\Pi_{h \in I(\alpha) \setminus I(\alpha_{2} - S+k)} \delta_{0,l_h}\\
\times \frac{V_{0}^{n_0}}{n_{0}!}\frac{V_{1}^{n_1}}{n_{1}!}\Pi_{h \in I(\alpha)} \frac{U_{h}^{l_h}}{l_{h}!} )
\end{multline*}
Therefore,
\begin{multline}
\mathcal{F}_{\alpha_{1},\alpha_{2}}^{1} =\sum_{j \in I(\alpha_{2}-S+k)}( \sum_{n_{0},n_{1},l_{h} \geq 0,h \in I(\alpha)}\\
( \sum_{\scriptscriptstyle{\stackrel{n_{0,1}+n_{0,2}+n_{0,3}=n_{0},n_{1,1}+n_{1,2}+n_{1,3}=n_{1}}{l_{h,1}+l_{h,2}+l_{h,3}
=l_{h},h \in I(\alpha)}}}
\frac{b_{1,k,\alpha_{1},n_{0,1},n_{1,1},l_{0,1}}}{\alpha_{1}!n_{0,1}!n_{1,1}! \Pi_{h \in I(\alpha)}l_{h,1}!}
\Pi_{h \in I(\alpha) \setminus \{ 0 \}} \delta_{0,l_{h,1}}\\
\times \frac{A_{j,\alpha_{2}-S+k+1,n_{0,2},n_{1,2},(l_{h,2})_{h \in I(\alpha_{2}-S+k+1)}}}{n_{0,2}!n_{1,2}!
\Pi_{h \in I(\alpha)}l_{h,2}!} \Pi_{h \in I(\alpha) \setminus I(\alpha_{2}-S+k+1)} \delta_{0,l_{h,2}}\\
\times \frac{\psi_{\alpha_{2}-S+k,n_{0,3},n_{1,3},(l_{h,3})_{h \in I(\alpha_{2}-S+k),h \neq j},l_{j,3}+1}}{\alpha_{2}!
n_{0,3}!n_{1,3}!\Pi_{h \in I(\alpha)}l_{h,3}!} \Pi_{h \in I(\alpha) \setminus I(\alpha_{2} - S+k)} \delta_{0,l_{h,3}})\\
\times V_{0}^{n_0}V_{1}^{n_1} \Pi_{h \in I(\alpha)} U_{h}^{l_h} ) \label{mathcalF1}
\end{multline}

On the other hand, using similar computations we get
$$ B_{2,k}(V_{0},V_{1},U_{0},W) \partial_{W}^{-S+k} \partial_{V_1} \Psi(V_{0},V_{1},(U_{h})_{h \geq 0},W)
= \sum_{\alpha \geq 0} \sum_{\stackrel{\alpha_{1}+\alpha_{2}=\alpha}{\alpha_{2} \geq S-k}}
\mathcal{C}_{\alpha_{1},\alpha_{2}}^{2} W^{\alpha}
$$
where
\begin{multline}
\mathcal{C}_{\alpha_{1},\alpha_{2}}^{2} = \sum_{n_{0},n_{1},l_{h} \geq 0,h \in I(\alpha)}
( \sum_{\scriptscriptstyle{\stackrel{n_{0,1}+n_{0,2}=n_{0},n_{1,1}+n_{1,2}=n_{1}}{l_{h,1}+l_{h,2}=l_{h},h \in I(\alpha)}}}
\frac{b_{2,k,\alpha_{1},n_{0,1},n_{1,1},l_{0,1}}}{\alpha_{1}!n_{0,1}!n_{1,1}! \Pi_{h \in I(\alpha)}l_{h,1}!}
\Pi_{h \in I(\alpha) \setminus \{ 0 \}} \delta_{0,l_{h,1}}\\
\times \frac{\psi_{\alpha_{2}+k-S,n_{0,2},n_{1,2}+1,
(l_{h,2})_{h \in I(\alpha_{2}+k-S)}}}{\alpha_{2}!n_{0,2}!n_{1,2}! \Pi_{h \in I(\alpha)}l_{h,2}!}
\Pi_{h \in I(\alpha) \setminus I(\alpha_{2}+k-S)} \delta_{0,l_{h,2}}) V_{0}^{n_0}V_{1}^{n_1} \Pi_{h \in I(\alpha)} U_{h}^{l_h}.
\label{mathcalC2}
\end{multline}
We also have that
$$ B_{2,k}(V_{0},V_{1},U_{0},W) \partial_{W}^{-S+k}\mathbb{D}_{\mathbf{B}}\Psi(V_{0},V_{1},(U_{h})_{h \geq 0},W)
= \sum_{\alpha \geq 0} \sum_{\stackrel{\alpha_{1}+\alpha_{2}=\alpha}{\alpha_{2} \geq S-k}}
\mathcal{F}_{\alpha_{1},\alpha_{2}}^{2} W^{\alpha}
$$
where
\begin{multline}
\mathcal{F}_{\alpha_{1},\alpha_{2}}^{2} =\sum_{j \in I(\alpha_{2}-S+k)}( \sum_{n_{0},n_{1},l_{h} \geq 0,h \in I(\alpha)}\\
( \sum_{\scriptscriptstyle{\stackrel{n_{0,1}+n_{0,2}+n_{0,3}=n_{0},n_{1,1}+n_{1,2}+n_{1,3}=n_{1}}{l_{h,1}+l_{h,2}+l_{h,3}
=l_{h},h \in I(\alpha)}}}
\frac{b_{2,k,\alpha_{1},n_{0,1},n_{1,1},l_{0,1}}}{\alpha_{1}!n_{0,1}!n_{1,1}! \Pi_{h \in I(\alpha)}l_{h,1}!}
\Pi_{h \in I(\alpha) \setminus \{ 0 \}} \delta_{0,l_{h,1}}\\
\times \frac{B_{j,\alpha_{2}-S+k+1,n_{0,2},n_{1,2},(l_{h,2})_{h \in I(\alpha_{2}-S+k+1)}}}{n_{0,2}!n_{1,2}!
\Pi_{h \in I(\alpha)}l_{h,2}!} \Pi_{h \in I(\alpha) \setminus I(\alpha_{2}-S+k+1)} \delta_{0,l_{h,2}}\\
\times \frac{\psi_{\alpha_{2}-S+k,n_{0,3},n_{1,3},(l_{h,3})_{h \in I(\alpha_{2}-S+k),h \neq j},l_{j,3}+1}}{\alpha_{2}!
n_{0,3}!n_{1,3}!\Pi_{h \in I(\alpha)}l_{h,3}!} \Pi_{h \in I(\alpha) \setminus I(\alpha_{2} - S+k)} \delta_{0,l_{h,3}})\\
\times V_{0}^{n_0}V_{1}^{n_1} \Pi_{h \in I(\alpha)} U_{h}^{l_h} ) \label{mathcalF2}
\end{multline}
and
$$ B_{3,k}(V_{0},V_{1},U_{0},W) \partial_{W}^{-S+k}\Psi(V_{0},V_{1},(U_{h})_{h \geq 0},W)
= \sum_{\alpha \geq 0} \sum_{\stackrel{\alpha_{1}+\alpha_{2}=\alpha}{\alpha_{2} \geq S-k}}
\mathcal{C}_{\alpha_{1},\alpha_{2}}^{3} W^{\alpha}
$$
where
\begin{multline}
\mathcal{C}_{\alpha_{1},\alpha_{2}}^{3} = \sum_{n_{0},n_{1},l_{h} \geq 0,h \in I(\alpha)}
( \sum_{\scriptscriptstyle{\stackrel{n_{0,1}+n_{0,2}=n_{0},n_{1,1}+n_{1,2}=n_{1}}{l_{h,1}+l_{h,2}=l_{h},h \in I(\alpha)}}}
\frac{b_{3,k,\alpha_{1},n_{0,1},n_{1,1},l_{0,1}}}{\alpha_{1}!n_{0,1}!n_{1,1}! \Pi_{h \in I(\alpha)}l_{h,1}!}
\Pi_{h \in I(\alpha) \setminus \{ 0 \}} \delta_{0,l_{h,1}}\\
\times \frac{\psi_{\alpha_{2}+k-S,n_{0,2},n_{1,2},
(l_{h,2})_{h \in I(\alpha_{2}+k-S)}}}{\alpha_{2}!n_{0,2}!n_{1,2}! \Pi_{h \in I(\alpha)}l_{h,2}!}
\Pi_{h \in I(\alpha) \setminus I(\alpha_{2}+k-S)} \delta_{0,l_{h,2}}) V_{0}^{n_0}V_{1}^{n_1} \Pi_{h \in I(\alpha)} U_{h}^{l_h}.
\label{mathcalC3}
\end{multline}
Finally, gathering the expansions (\ref{mathcalC1}), (\ref{mathcalF1}), (\ref{mathcalC2}) and (\ref{mathcalF2}) with
(\ref{mathcalC3}) yields the result.
\end{proof}
\begin{prop} The sequences $\varphi_{\alpha,n_{0},n_{1},(l_{h})_{h \in I(\alpha)}}$ and
$\psi_{\alpha,n_{0},n_{1},(l_{h})_{h \in I(\alpha)}}$ satisfy the following inequalities
\begin{equation}
\varphi_{\alpha,n_{0},n_{1},(l_{h})_{h \in I(\alpha)}} \leq \psi_{\alpha,n_{0},n_{1},(l_{h})_{h \in I(\alpha)}}
\end{equation}
for all $\alpha \geq 0$, all $n_{0},n_{1} \geq 0$, all $l_{h} \geq 0$, $h \in I(\alpha)$.
\end{prop}
\begin{proof} For $\alpha=0$, using the recursions (\ref{recursion_phi_alpha}) and (\ref{recursion_psii_alpha_n_l}), we get that
$$ \varphi_{0,n_{0},n_{1},(l_{h})_{h \in I(0)}} = \tilde{w}_{0,n_{0},n_{1},(l_{h})_{h \in I(0)}}=
\psi_{0,n_{0},n_{1},(l_{h})_{h \in I(0)}} $$
for all $n_{0},n_{1},l_{0} \geq 0$. By induction on $\alpha$ and using the inequalities (\ref{ineq_recursion_varphi_alpha_n_l})
together with the equalities (\ref{recursion_psii_alpha_n_l}), one gets the result.
\end{proof}

\section{Convergent series solutions for a functional equation with infinitely many variables}

\subsection{Banach spaces of formal series}

Let $\rho > 1$ and $\sigma,\bar{V}_{0},\bar{V}_{1},\bar{W},\bar{\delta}>0$ be real numbers.
For any given real number $b > 1$, we define the sequences $r_{b}(\alpha) = \sum_{n=0}^{\alpha} 1/(n+1)^{b}$ for all
$\alpha \geq 0$ and $\bar{U}_{h} = \bar{\delta}/(h^{b}+1)$ for all $h \geq 0$.
\begin{defin} Let $\alpha \geq 0$ be an integer. We denote $E_{\rho,\alpha,\bar{V}_{0},\bar{V}_{1},(\bar{U}_{h})_{h \in I(\alpha)}}$ the
vector space of formal series
$$ \Psi(V_{0},V_{1},(U_{h})_{h \in I(\alpha)}) = \sum_{n_{0},n_{1},l_{h} \geq 0, h \in I(\alpha)}
\psi_{n_{0},n_{1},(l_{h})_{h \in I(\alpha)}}
\frac{V_{0}^{n_0}}{n_{0}!}\frac{V_{1}^{n_1}}{n_{1}!}\Pi_{h \in I(\alpha)}\frac{U_{h}^{l_h}}{l_{h}!} $$
that belong to $\mathbb{C}[[V_{0},V_{1},(U_{h})_{h \in I(\alpha)}]]$ such that the series
\begin{multline*}
|| \Psi(V_{0},V_{1},(U_{h})_{h \in I(\alpha)}) ||_{\rho,\alpha,\bar{V}_{0},\bar{V}_{1},(\bar{U}_{h})_{h \in I(\alpha)}} \\
= \sum_{n_{0},n_{1},l_{h} \geq 0, h \in I(\alpha)} \frac{|\psi_{n_{0},n_{1},(l_{h})_{h \in I(\alpha)}}|}{\exp(\sigma r_{b}(\alpha) \rho)}
\frac{\bar{V}_{0}^{n_0}\bar{V}_{1}^{n_1}
\Pi_{h \in I(\alpha)}\bar{U}_{h}^{l_h}}{(n_{0}+n_{1}+\sum_{h \in I(\alpha)} l_{h} + \alpha)! }
\end{multline*}
is convergent. We denote also $G_{(\rho,\bar{V}_{0},\bar{V}_{1},(\bar{U}_{h})_{h \geq 0},\bar{W})}$ the vector space of
formal series
$$ \Psi(V_{0},V_{1},(U_{h})_{h \geq 0},W) = \sum_{\alpha \geq 0} \Psi_{\alpha}(V_{0},V_{1},(U_{h})_{h \in I(\alpha)})
\frac{W^{\alpha}}{\alpha!} $$
where $\Psi_{\alpha}(V_{0},V_{1},(U_{h})_{h \in I(\alpha)})$ belong to
$E_{\rho,\alpha,\bar{V}_{0},\bar{V}_{1},(\bar{U}_{h})_{h \in I(\alpha)}}$ for all $\alpha \geq 0$, such that the series
$$
||\Psi(V_{0},V_{1},(U_{h})_{h \geq 0},W)||_{(\rho,\bar{V}_{0},\bar{V}_{1},(\bar{U}_{h})_{h \geq 0},\bar{W})}
= \sum_{\alpha \geq 0} || \Psi_{\alpha}||_{\rho,\alpha,\bar{V}_{0},\bar{V}_{1},(\bar{U}_{h})_{h \in I(\alpha)}} \bar{W}^{\alpha}
$$
is convergent. One checks that the space $G_{(\rho,\bar{V}_{0},\bar{V}_{1},(\bar{U}_{h})_{h \geq 0},\bar{W})}$ equipped with
the norm \\ $||.||_{(\rho,\bar{V}_{0},\bar{V}_{1},(\bar{U}_{h})_{h \geq 0},\bar{W})}$ is a Banach space.
\end{defin}

In the next two propositions, we study norm estimates for linear operators acting on the Banach spaces 
$E_{\rho,\alpha,\bar{V}_{0},\bar{V}_{1},(\bar{U}_{h})_{h \in I(\alpha)}}$ constructed above.
\begin{prop} Consider a formal series
$$ b(V_{0},V_{1},(U_{h})_{h \in I(\alpha)}) = \sum_{n_{0},n_{1},l_{h} \geq 0, h \in I(\alpha)}
b_{n_{0},n_{1},(l_{h})_{h \in I(\alpha)}}
\frac{V_{0}^{n_0}}{n_{0}!}\frac{V_{1}^{n_1}}{n_{1}!}\Pi_{h \in I(\alpha)}\frac{U_{h}^{l_h}}{l_{h}!} $$
which is absolutely convergent on the polydisc
$D(0,\bar{V}_{0}) \times D(0,\bar{V}_{1}) \times_{h \in I(\alpha)} D(0,\bar{U}_{h})$. We use the notation
$$ |b|(\bar{V}_{0},\bar{V}_{1},(\bar{U}_{h})_{h \in I(\alpha)}) = \sum_{n_{0},n_{1},l_{h} \geq 0, h \in I(\alpha)}
|b_{n_{0},n_{1},(l_{h})_{h \in I(\alpha)}}|
\frac{\bar{V}_{0}^{n_0}}{n_{0}!}\frac{\bar{V}_{1}^{n_1}}{n_{1}!}\Pi_{h \in I(\alpha)}\frac{\bar{U}_{h}^{l_h}}{l_{h}!}.$$
Let $\Psi(V_{0},V_{1},(U_{h})_{h \in I(\alpha)})$ belonging to $E_{\rho,\alpha,\bar{V}_{0},\bar{V}_{1},(\bar{U}_{h})_{h \in I(\alpha)}}$.
Then, the following inequality
\begin{multline}
||b(V_{0},V_{1},(U_{h})_{h \in I(\alpha)})
\Psi(V_{0},V_{1},(U_{h})_{h \in I(\alpha)})||_{\rho,\alpha,\bar{V}_{0},\bar{V}_{1},(\bar{U}_{h})_{h \in I(\alpha)}}\\
 \leq |b|(\bar{V}_{0},\bar{V}_{1},(\bar{U}_{h})_{h \in I(\alpha)})
||\Psi(V_{0},V_{1},(U_{h})_{h \in I(\alpha)})||_{\rho,\alpha,\bar{V}_{0},\bar{V}_{1},(\bar{U}_{h})_{h \in I(\alpha)}}
\label{norm_product_b_psi_1<}
\end{multline}
holds.
\end{prop}
\begin{proof} Let
$$ \Psi(V_{0},V_{1},(U_{h})_{h \in I(\alpha)}) = \sum_{n_{0},n_{1},l_{h} \geq 0, h \in I(\alpha)}
\psi_{n_{0},n_{1},(l_{h})_{h \in I(\alpha)}}
\frac{V_{0}^{n_0}}{n_{0}!}\frac{V_{1}^{n_1}}{n_{1}!}\Pi_{h \in I(\alpha)}\frac{U_{h}^{l_h}}{l_{h}!} $$
which belongs to $E_{\rho,\alpha,\bar{V}_{0},\bar{V}_{1},(\bar{U}_{h})_{h \in I(\alpha)}}$. By definition, we have that
\begin{multline*}
||b(V_{0},V_{1},(U_{h})_{h \in I(\alpha)})
\Psi(V_{0},V_{1},(U_{h})_{h \in I(\alpha)})||_{\rho,\alpha,\bar{V}_{0},\bar{V}_{1},(\bar{U}_{h})_{h \in I(\alpha)}}\\
= \sum_{n_{0},n_{1},l_{h} \geq 0, h \in I(\alpha)} |
\sum_{\scriptscriptstyle{\stackrel{n_{0,1}+n_{0,2}=n_{0},n_{1,1}+n_{1,2}=n_{1}}{l_{h,1}+l_{h,2}=l_{h},h \in I(\alpha)}}}
\frac{n_{0}!n_{1}!\Pi_{h \in I(\alpha)}l_{h}!}{n_{0,1}!n_{0,2}!n_{1,1}!n_{1,2}!\Pi_{h \in I(\alpha)}l_{h,1}!l_{h,2}!}\\
b_{n_{0,1},n_{1,1},(l_{h,1})_{h \in I(\alpha)}} \psi_{n_{0,2},n_{1,2},(l_{h,2})_{h \in I(\alpha)}}|
\frac{1}{\exp( \sigma r_{b}(\alpha) \rho )} \frac{\bar{V}_{0}^{n_0}\bar{V}_{1}^{n_1}
\Pi_{h \in I(\alpha)}\bar{U}_{h}^{l_h}}{(n_{0}+n_{1}+\sum_{h \in I(\alpha)} l_{h} + \alpha)! }.
\end{multline*}
We can give upper bounds for this latter expression
\begin{multline}
||b(V_{0},V_{1},(U_{h})_{h \in I(\alpha)})
\Psi(V_{0},V_{1},(U_{h})_{h \in I(\alpha)})||_{\rho,\alpha,\bar{V}_{0},\bar{V}_{1},(\bar{U}_{h})_{h \in I(\alpha)}}\\
\leq  \sum_{n_{0},n_{1},l_{h} \geq 0, h \in I(\alpha)}
\sum_{\scriptscriptstyle{\stackrel{n_{0,1}+n_{0,2}=n_{0},n_{1,1}+n_{1,2}=n_{1}}{l_{h,1}+l_{h,2}=l_{h},h \in I(\alpha)}}}
\left( \frac{n_{0}!n_{1}!\Pi_{h \in I(\alpha)}l_{h}!}{n_{0,2}!n_{1,2}!\Pi_{h \in I(\alpha)}l_{h,2}!} \right.\\
\left. \times \frac{(n_{0,2}+n_{1,2}+\sum_{h \in I(\alpha)} l_{h,2} + \alpha)!}{(n_{0}+n_{1}+\sum_{h \in I(\alpha)} l_{h} + \alpha)!}
\right) \frac{|b_{n_{0,1},n_{1,1},(l_{h,1})_{h \in I(\alpha)}}|}{n_{0,1}!n_{1,1}!\Pi_{h \in I(\alpha)} l_{h,1}!}
\bar{V}_{0}^{n_{0,1}}\bar{V}_{1}^{n_{1,1}}\Pi_{h \in I(\alpha)}\bar{U}_{h}^{l_{h,1}}\\
\times |\psi_{n_{0,2},n_{1,2},(l_{h,2})_{h \in I(\alpha)}}|\frac{1}{\exp( \sigma r_{b}(\alpha) \rho)}
\frac{\bar{V}_{0}^{n_{0,2}}\bar{V}_{1}^{n_{1,2}}\Pi_{h \in I(\alpha)}\bar{U}_{h}^{l_{h,2}}}{(n_{0,2}+n_{1,2}+
\sum_{h \in I(\alpha)} l_{h,2} + \alpha)!} \label{norm_product_b_psi_2<}
\end{multline}
\begin{lemma} For all integers $\alpha,n_{0},n_{1} \geq 0$, all $l_{h} \geq 0$, all $0 \leq n_{0,2} \leq n_{0}$,
all $0 \leq n_{1,2} \leq n_{1}$, all $0 \leq l_{h,2} \leq l_{h}$ for $h \in I(\alpha)$, we have that
\begin{equation}
\frac{n_{0}!n_{1}!\Pi_{h \in I(\alpha)}l_{h}!}{n_{0,2}!n_{1,2}!\Pi_{h \in I(\alpha)}l_{h,2}!}
\frac{(n_{0,2}+n_{1,2}+\sum_{h \in I(\alpha)} l_{h,2} + \alpha)!}{(n_{0}+n_{1}+\sum_{h \in I(\alpha)} l_{h} + \alpha)!} \leq 1.
\label{frac_factorial_ineq_1}
\end{equation}
\end{lemma}
\begin{proof} For any integers $a \leq b$ and $\alpha \geq 0$ one has
\begin{equation}
\frac{(a+\alpha)!}{(b+\alpha)!} \leq \frac{a!}{b!} \label{frac_factorial_ineq_ab}
\end{equation}
by using the factorization $(a+\alpha)!=(a+\alpha)(a+\alpha-1)\cdots(a+1)a!$. Therefore, one gets the inequality
\begin{multline}
\frac{n_{0}!n_{1}!\Pi_{h \in I(\alpha)}l_{h}!}{n_{0,2}!n_{1,2}!\Pi_{h \in I(\alpha)}l_{h,2}!}
\frac{(n_{0,2}+n_{1,2}+\sum_{h \in I(\alpha)} l_{h,2} + \alpha)!}{(n_{0}+n_{1}+\sum_{h \in I(\alpha)} l_{h} + \alpha)!} \\
\leq \frac{n_{0}!n_{1}!\Pi_{h \in I(\alpha)}l_{h}!}{n_{0,2}!n_{1,2}!\Pi_{h \in I(\alpha)}l_{h,2}!}
\frac{(n_{0,2}+n_{1,2}+\sum_{h \in I(\alpha)} l_{h,2})!}{(n_{0}+n_{1}+\sum_{h \in I(\alpha)} l_{h})!} \label{frac_factorial_ineq_2}
\end{multline}
Now, from the identity $(A+B)^{n_{0}+n_{1}+\sum_{h \in I(\alpha)} l_{h}} = (A+B)^{n_0}(A+B)^{n_1}\times
\Pi_{h \in I(\alpha)}(A+B)^{l_h}$ and the binomial formula, we deduce that
\begin{multline*}
\frac{n_{0}!n_{1}!\Pi_{h \in I(\alpha)}l_{h}!}{n_{0,1}!n_{0,2}!n_{1,1}!n_{1,2}! \Pi_{h \in I(\alpha)} l_{h,1}! l_{h,2}!} \\
\leq \frac{(n_{0}+n_{1}+\sum_{h \in I(\alpha)} l_{h})!}{(n_{0,1} + n_{1,1} + \sum_{h \in I(\alpha)} l_{h,1})!
(n_{0,2} + n_{1,2} + \sum_{h \in I(\alpha)} l_{h,2})!}
\end{multline*}
for all $n_{0,1}+n_{0,2}=n_{0}$, $n_{1,1}+n_{1,2}=n_{1}$, $l_{h,1}+l_{h,2}=l_{h}$. Therefore, we deduce that
\begin{multline}
\frac{n_{0}!n_{1}!\Pi_{h \in I(\alpha)}l_{h}!}{n_{0,2}!n_{1,2}!\Pi_{h \in I(\alpha)}l_{h,2}!}
\frac{(n_{0,2}+n_{1,2}+\sum_{h \in I(\alpha)} l_{h,2} + \alpha)!}{(n_{0}+n_{1}+\sum_{h \in I(\alpha)} l_{h} + \alpha)!} \\
\leq \frac{n_{0,1}!n_{1,1}!\Pi_{h \in I(\alpha)} l_{h,1}!}{(n_{0,1}+n_{1,1}+\sum_{h \in I(\alpha)} l_{h,1})!} \leq 1, \label{frac_factorial_ineq_3}
\end{multline}
and the lemma follows from the inequalities (\ref{frac_factorial_ineq_2}), (\ref{frac_factorial_ineq_3}).
\end{proof}
Finally, the inequality (\ref{norm_product_b_psi_1<}) follows from (\ref{norm_product_b_psi_2<}) and
(\ref{frac_factorial_ineq_1}).
\end{proof}

\begin{prop} Let $\alpha,\alpha'$ be integers such that $\alpha' \geq 0$ and $\alpha'+1 < \alpha$. Let $j \in I(\alpha')$ and $k \in \{0,1\}$.
We have that
\begin{multline}
||\partial_{U_j}\Psi(V_{0},V_{1},(U_{h})_{h \in I(\alpha')})||_{\rho,\alpha,\bar{V}_{0},\bar{V}_{1},(\bar{U}_{h})_{h \in I(\alpha)}}\\
\leq \frac{ \exp( -\sigma \rho \frac{\alpha - \alpha'}{(\alpha+1)^{b}} ) }{\bar{U}_{j} \Pi_{l=1}^{\alpha - \alpha'-1}(\alpha - l +1)}
||\Psi(V_{0},V_{1},(U_{h})_{h \in I(\alpha')})||_{\rho,\alpha',\bar{V}_{0},\bar{V}_{1},(\bar{U}_{h})_{h \in I(\alpha')}},
\label{norm_partial_Uj_Psi_1<}
\end{multline}
\begin{multline}
||\partial_{V_{k}}\Psi(V_{0},V_{1},(U_{h})_{h \in I(\alpha')})||_{\rho,\alpha,\bar{V}_{0},\bar{V}_{1},(\bar{U}_{h})_{h \in I(\alpha)}}\\
\leq \frac{ \exp( -\sigma \rho \frac{\alpha - \alpha'}{(\alpha+1)^{b}} ) }{\bar{V}_{k} \Pi_{l=1}^{\alpha - \alpha'-1}(\alpha - l +1)}
||\Psi(V_{0},V_{1},(U_{h})_{h \in I(\alpha')})||_{\rho,\alpha',\bar{V}_{0},\bar{V}_{1},(\bar{U}_{h})_{h \in I(\alpha')}},
\label{norm_partial_Vk_Psi<}
\end{multline}
and
\begin{multline}
||\Psi(V_{0},V_{1},(U_{h})_{h \in I(\alpha')})||_{\rho,\alpha,\bar{V}_{0},\bar{V}_{1},(\bar{U}_{h})_{h \in I(\alpha)}}\\
\leq \frac{ \exp( -\sigma \rho \frac{\alpha - \alpha'}{(\alpha+1)^{b}} ) }{\Pi_{l=1}^{\alpha - \alpha'}(\alpha - l +1)}
||\Psi(V_{0},V_{1},(U_{h})_{h \in I(\alpha')})||_{\rho,\alpha',\bar{V}_{0},\bar{V}_{1},(\bar{U}_{h})_{h \in I(\alpha')}}
\label{norm_alpha_Psi_alpha'<}
\end{multline}
for all $\Psi(V_{0},V_{1},(U_{h})_{h \in I(\alpha')}) \in E_{\rho,\alpha',\bar{V}_{0},\bar{V}_{1},(\bar{U}_{h})_{h \in I(\alpha')}}$.
\end{prop}
\begin{proof} Let $\Psi(V_{0},V_{1},(U_{h})_{h \in I(\alpha')}) \in E_{\rho,\alpha',\bar{V}_{0},\bar{V}_{1},(\bar{U}_{h})_{h \in I(\alpha')}}$ that
we write in the form
\begin{multline*}
\Psi(V_{0},V_{1},(U_{h})_{h \in I(\alpha')}) \\
= \sum_{n_{0},n_{1},l_{h} \geq 0, h \in I(\alpha)} \psi_{n_{0},n_{1},(l_{h})_{h \in I(\alpha')}}
\Pi_{h \in I(\alpha) \setminus I(\alpha')} \delta_{0,l_{h}}
\frac{V_{0}^{n_0}}{n_{0}!}\frac{V_{1}^{n_1}}{n_{1}!}\Pi_{h \in I(\alpha)}\frac{U_{h}^{l_h}}{l_{h}!}.
\end{multline*}
By definition, we get that
\begin{multline*}
||\partial_{U_j}\Psi(V_{0},V_{1},(U_{h})_{h \in I(\alpha')})||_{\rho,\alpha,\bar{V}_{0},\bar{V}_{1},(\bar{U}_{h})_{h \in I(\alpha)}}\\
= \sum_{n_{0},n_{1},l_{h} \geq 0, h \in I(\alpha)} \frac{|\psi_{n_{0},n_{1},(l_{h})_{h \in I(\alpha'),h \neq j},l_{j}+1} \Pi_{h \in
I(\alpha) \setminus I(\alpha')} \delta_{0,l_h}|}{\exp(\sigma r_{b}(\alpha) \rho)} \frac{\bar{V}_{0}^{n_0}\bar{V}_{1}^{n_1}
\Pi_{h \in I(\alpha)}\bar{U}_{h}^{l_h}}{ (n_{0}+n_{1}+\sum_{h \in I(\alpha)} l_{h} + \alpha)! }.
\end{multline*}
We give upper bounds for this latter expression,
\begin{multline}
||\partial_{U_j}\Psi(V_{0},V_{1},(U_{h})_{h \in I(\alpha')})||_{\rho,\alpha,\bar{V}_{0},\bar{V}_{1},(\bar{U}_{h})_{h \in I(\alpha)}} \\
=  \sum_{n_{0},n_{1},l_{h} \geq 0, h \in I(\alpha')} \left(
\frac{(n_{0}+n_{1}+\sum_{h \in I(\alpha'),h \neq j} l_{h} + l_{j}+1+\alpha')!}{(n_{0}+n_{1}+\sum_{h \in I(\alpha')} l_{h} + \alpha)!}
\frac{1}{\bar{U}_{j}\exp(\sigma \rho( r_{b}(\alpha) - r_{b}(\alpha') ) )} \right)\\
\times \frac{|\psi_{n_{0},n_{1},(l_{h})_{h \in I(\alpha'),h \neq j},l_{j}+1}|}{\exp(\sigma r_{b}(\alpha') \rho)} \frac{\bar{V}_{0}^{n_0}\bar{V}_{1}^{n_1}
\Pi_{h \in I(\alpha'),h \neq j}\bar{U}_{h}^{l_h} \bar{U}_{j}^{l_{j}+1} }{ (n_{0}+n_{1}+\sum_{h \in I(\alpha'),h \neq j} l_{h} + l_{j} +1 + \alpha')! }
\label{norm_partial_Uj_Psi_2<}
\end{multline}
\begin{lemma} We have
\begin{multline}
\frac{(n_{0}+n_{1}+\sum_{h \in I(\alpha'),h \neq j} l_{h} + l_{j}+1+\alpha')!}{(n_{0}+n_{1}+\sum_{h \in I(\alpha')} l_{h} + \alpha)!}
\frac{1}{\exp(\sigma \rho( r_{b}(\alpha) - r_{b}(\alpha') ) )} \\
\leq \frac{ \exp( -\sigma \rho \frac{\alpha - \alpha'}{(\alpha+1)^{b}} ) }{\Pi_{l=1}^{\alpha - \alpha'-1}(\alpha - l +1)}
\label{ineq_frac_factorial_exponential}
\end{multline}
\end{lemma}
\begin{proof} We notice that
$$ r_{b}(\alpha) - r_{b}(\alpha') = \sum_{n=\alpha'+1}^{\alpha} \frac{1}{(n+1)^{b}} \geq \frac{ \alpha - \alpha' }{(\alpha + 1)^{b}} $$
and, with the help of (\ref{frac_factorial_ineq_ab}), that for all integers $a \geq 0$,
$$ \frac{(a+1+\alpha')!}{(a+\alpha)!} \leq \frac{1}{\Pi_{l=1}^{\alpha - \alpha'-1}(\alpha - l +1)}.$$
The lemma follows.
\end{proof}
We get that the inequality (\ref{norm_partial_Uj_Psi_1<})  follows from (\ref{norm_partial_Uj_Psi_2<}) together with
(\ref{ineq_frac_factorial_exponential}). Finally, using similar arguments, one gets also the inequalities
(\ref{norm_partial_Vk_Psi<}) and (\ref{norm_alpha_Psi_alpha'<}).
\end{proof}

In the next two propositions, we study norm estimates for linear operators acting on the Banach space $G_{(\rho,\bar{V}_{0},\bar{V}_{1},(\bar{U}_{h})_{h \geq 0},\bar{W})}$.
\begin{prop} Let a formal series $b(V_{0},V_{1},U_{0},W) \in \mathbb{C}[[V_{0},V_{1},U_{0},W]]$ be absolutely convergent on
the polydisc $D(0,\bar{V}_{0}) \times D(0,\bar{V}_{1}) \times D(0,\bar{U}_{0}) \times D(0,\bar{W})$. Let
$\Psi(V_{0},V_{1},(U_{h})_{h \geq 0},W)$ belonging to $G_{(\rho,\bar{V}_{0},\bar{V}_{1},(\bar{U}_{h})_{h \geq 0},\bar{W})}$. 
Then, the product $b(V_{0},V_{1},U_{0},W)\Psi(V_{0},V_{1},(U_{h})_{h \geq 0},W)$ belongs to
$G_{(\rho,\bar{V}_{0},\bar{V}_{1},(\bar{U}_{h})_{h \geq 0},\bar{W})}$ and the
following inequality
\begin{multline}
||b(V_{0},V_{1},U_{0},W)\Psi(V_{0},V_{1},(U_{h})_{h \geq 0},W)||_{(\rho,\bar{V}_{0},\bar{V}_{1},(\bar{U}_{h})_{h \geq 0},\bar{W})} \\
\leq |b|(\bar{V}_{0},\bar{V}_{1},\bar{U}_{0},\bar{W})||\Psi(V_{0},V_{1},(U_{h})_{h \geq 0},W)||_{(\rho,\bar{V}_{0},\bar{V}_{1},(\bar{U}_{h})_{h \geq 0},\bar{W})} \label{norm_b_Psi_W<|b|_norm_Psi_W}
\end{multline}
holds.
\end{prop}
\begin{proof} Let
\begin{multline*}
b(V_{0},V_{1},U_{0},W) = \sum_{\alpha \geq 0} b_{\alpha}(V_{0},V_{1},U_{0}) \frac{W^{\alpha}}{\alpha !},\\
\Psi(V_{0},V_{1},(U_{h})_{h \geq 0},W) = \sum_{\alpha \geq 0} \Psi_{\alpha}(V_{0},V_{1},(U_{h})_{h \in I(\alpha)})
\frac{W^{\alpha}}{\alpha !}.
\end{multline*}
By definition, we get
\begin{multline}
||b(V_{0},V_{1},U_{0},W)\Psi(V_{0},V_{1},(U_{h})_{h \geq 0},W)||_{(\rho,\bar{V}_{0},\bar{V}_{1},(\bar{U}_{h})_{h \geq 0},\bar{W})}\\
= \sum_{\alpha \geq 0} || \sum_{\alpha_{1} + \alpha_{2}= \alpha} \alpha !
\frac{b_{\alpha_{1}}(V_{0},V_{1},U_{0})}{\alpha_{1}!}\frac{\Psi_{\alpha_2}(V_{0},V_{1},(U_{h})_{h \in I(\alpha_{2})})}{\alpha_{2}!}||_{\rho,
\alpha,\bar{V}_{0},\bar{V}_{1},(\bar{U}_{h})_{h \in I(\alpha)}} \bar{W}^{\alpha}. \label{norm_b_Psi_W=}
\end{multline}
\begin{lemma} We have
\begin{multline}
||b_{\alpha_{1}}(V_{0},V_{1},U_{0})\Psi_{\alpha_2}(V_{0},V_{1},(U_{h})_{h \in I(\alpha_{2})})||_{\rho,
\alpha,\bar{V}_{0},\bar{V}_{1},(\bar{U}_{h})_{h \in I(\alpha)}} \\
\leq \frac{\alpha_{2}!}{\alpha !} |b_{\alpha_1}|(\bar{V}_{0},\bar{V}_{1},\bar{U}_{0})
||\Psi_{\alpha_2}(V_{0},V_{1},(U_{h})_{h \in I(\alpha_{2})})||_{\rho,\alpha_{2},\bar{V}_{0},\bar{V}_{1},
(\bar{U}_{h})_{h \in I(\alpha_{2})}}.
\label{norm_balpha1_Psialpha2_1<}
\end{multline}
\end{lemma}
\begin{proof} We can write
$$
b_{\alpha_1}(V_{0},V_{1},U_{0}) = \sum_{n_{0},n_{1},l_{h} \geq 0, h \in I(\alpha)} b_{\alpha_{1},n_{0},n_{1},l_{0}}
\Pi_{h \in I(\alpha) \setminus \{ 0 \}} \delta_{0,l_h}
\frac{V_{0}^{n_0}}{n_{0}!}\frac{V_{1}^{n_1}}{n_{1}!}\Pi_{h \in I(\alpha)}\frac{U_{h}^{l_h}}{l_{h}!}
$$
and
\begin{multline*}
\Psi_{\alpha_2}(V_{0},V_{1},(U_{h})_{h \in I(\alpha_{2})}) \\
= \sum_{n_{0},n_{1},l_{h} \geq 0, h \in I(\alpha)} \psi_{\alpha_{2},n_{0},n_{1},(l_{h})_{h \in I(\alpha_{2})}}
\Pi_{h \in I(\alpha) \setminus I(\alpha_{2})} \delta_{0,l_h}
\frac{V_{0}^{n_0}}{n_{0}!}\frac{V_{1}^{n_1}}{n_{1}!}\Pi_{h \in I(\alpha)}\frac{U_{h}^{l_h}}{l_{h}!}.
\end{multline*}
By remembering (\ref{norm_product_b_psi_1<}) of Proposition 5, we deduce that
\begin{multline}
||b_{\alpha_{1}}(V_{0},V_{1},U_{0})\Psi_{\alpha_2}(V_{0},V_{1},(U_{h})_{h \in I(\alpha_{2})})||_{\rho,
\alpha,\bar{V}_{0},\bar{V}_{1},(\bar{U}_{h})_{h \in I(\alpha)}} \leq \\
|b_{\alpha_1}|(\bar{V}_{0},\bar{V}_{1},\bar{U}_{0}) (\sum_{n_{0},n_{1},l_{h} \geq 0, h \in I(\alpha_{2})}
\frac{|\psi_{\alpha_{2},n_{0},n_{1},(l_{h})_{h \in I(\alpha_{2})}}|}{\exp( \sigma r_{b}(\alpha) \rho )}
\frac{\bar{V}_{0}^{n_0}\bar{V}_{1}^{n_1}
\Pi_{h \in I(\alpha_{2})}\bar{U}_{h}^{l_h}}{(n_{0}+n_{1}+\sum_{h \in I(\alpha_{2})} l_{h} + \alpha)! }).
\label{norm_balpha1_Psialpha2_2<}
\end{multline}
\begin{lemma} We have
\begin{equation}
\frac{1}{(n_{0}+n_{1}+\sum_{h \in I(\alpha_{2})} l_{h} + \alpha)!} \leq \frac{ \alpha_{2}! }{ \alpha !}
\frac{1}{(n_{0} + n_{1} + \sum_{h \in I(\alpha_{2})} l_{h} + \alpha_{2})!}. \label{frac_factorial_alpha_alpha2<}
\end{equation}
\end{lemma}
\begin{proof} We write
$$ \frac{1}{(n_{0}+n_{1}+\sum_{h \in I(\alpha_{2})} l_{h} + \alpha)!} =
\frac{ (n_{0}+n_{1}+\sum_{h \in I(\alpha_{2})} l_{h} + \alpha_{2})!}{(n_{0}+n_{1}+\sum_{h \in I(\alpha_{2})} l_{h} + \alpha)!}
\frac{1}{(n_{0}+n_{1}+\sum_{h \in I(\alpha_{2})} l_{h} + \alpha_{2})!} $$
and we use the inequality
$$ \frac{(a+\alpha_{2})!}{(a+\alpha)!} \leq \frac{\alpha_{2}!}{\alpha !} $$
for all $\alpha = \alpha_{1}+\alpha_{2}$ and all $a \in \mathbb{N}$ which follows from (\ref{frac_factorial_ineq_ab}). This yields the
lemma.
\end{proof}
Using the fact that $\exp( \sigma r_{b}(\alpha) \rho ) \geq \exp( \sigma r_{b}(\alpha_{2}) \rho )$ and gathering the inequalities
(\ref{norm_balpha1_Psialpha2_2<}) and (\ref{frac_factorial_alpha_alpha2<}) yields (\ref{norm_balpha1_Psialpha2_1<}). 
\end{proof}
Finally, using (\ref{norm_b_Psi_W=}) with (\ref{norm_balpha1_Psialpha2_1<}), one gets
\begin{multline}
||b(V_{0},V_{1},U_{0},W)\Psi(V_{0},V_{1},(U_{h})_{h \geq 0},W)||_{(\rho,\bar{V}_{0},\bar{V}_{1},(\bar{U}_{h})_{h \geq 0},\bar{W})}\\
\leq \sum_{\alpha \geq 0} (\sum_{\alpha_{1} + \alpha_{2} = \alpha} \frac{|b_{\alpha_1}|(\bar{V}_{0},\bar{V}_{1},\bar{U}_{0})}{\alpha_{1}!}
||\Psi_{\alpha_2}(V_{0},V_{1},(U_{h})_{h \in I(\alpha_{2})})||_{\rho,\alpha_{2},\bar{V}_{0},\bar{V}_{1},(\bar{U}_{h})_{h \in I(\alpha_{2})}})
\bar{W}^{\alpha}
\end{multline}
from which the inequality (\ref{norm_b_Psi_W<|b|_norm_Psi_W}) follows.
\end{proof}

\begin{prop}
1) Let $S,k \geq 0$ be integers such that
\begin{equation}
S \geq k+1 + \max(b(d_{1,k}+2)+3,d+1+b(d+d_{1,k}+1)). \label{relation_S_k_1}
\end{equation}
Then, there exists a constant $C_{8.1}>0$ (which is independent of $\rho >1$) such that
\begin{multline}
||B_{1,k}(V_{0},V_{1},U_{0},W)\partial_{W}^{-S+k}
\mathbb{D}_{\mathbf{A}}\Psi(V_{0},V_{1},(U_{h})_{h \geq 0},W)||_{(\rho,\bar{V}_{0},\bar{V}_{1},(\bar{U}_{h})_{h \geq 0},\bar{W})}
\\
\leq C_{8.1}\bar{W}^{S-k}
||\Psi(V_{0},V_{1},(U_{h})_{h \geq 0},W)||_{(\rho,\bar{V}_{0},\bar{V}_{1},(\bar{U}_{h})_{h \geq 0},\bar{W})}
\label{norm_intW_DA_Psi_W<norm_Psi_W}
\end{multline}
for all $\Psi(V_{0},V_{1},(U_{h})_{h \geq 0},W) \in G_{(\rho,\bar{V}_{0},\bar{V}_{1},(\bar{U}_{h})_{h \geq 0},\bar{W})}$.\\
2) Let $S,k \geq 0$ be integers such that
\begin{equation}
S \geq k+3+b(2+d_{2,k}). \label{relation_S_k_2}
\end{equation}
Then, there exists a constant $C_{8.2}>0$ (which is independent of $\rho >1$) such that
\begin{multline}
||B_{2,k}(V_{0},V_{1},U_{0},W)\partial_{W}^{-S+k}
\mathbb{D}_{\mathbf{B}}\Psi(V_{0},V_{1},(U_{h})_{h \geq 0},W)||_{(\rho,\bar{V}_{0},\bar{V}_{1},(\bar{U}_{h})_{h \geq 0},\bar{W})}
\\
\leq C_{8.2}\bar{W}^{S-k}
||\Psi(V_{0},V_{1},(U_{h})_{h \geq 0},W)||_{(\rho,\bar{V}_{0},\bar{V}_{1},(\bar{U}_{h})_{h \geq 0},\bar{W})}
\label{norm_intW_DB_Psi_W<norm_Psi_W}
\end{multline}
for all $\Psi(V_{0},V_{1},(U_{h})_{h \geq 0},W) \in G_{(\rho,\bar{V}_{0},\bar{V}_{1},(\bar{U}_{h})_{h \geq 0},\bar{W})}$.
\end{prop}
\begin{proof}  1) We show the first inequality (\ref{norm_intW_DA_Psi_W<norm_Psi_W}). We expand
$$ B_{1,k}(V_{0},V_{1},U_{0},W) = \sum_{\alpha \geq 0} B_{1,k,\alpha}(V_{0},V_{1},U_{0}) \frac{W^{\alpha}}{\alpha!}.$$
By definition, we have
\begin{multline}
||B_{1,k}(V_{0},V_{1},U_{0},W)\partial_{W}^{-S+k}
\mathbb{D}_{\mathbf{A}}\Psi(V_{0},V_{1},(U_{h})_{h \geq 0},W)||_{(\rho,\bar{V}_{0},\bar{V}_{1},(\bar{U}_{h})_{h \geq 0},\bar{W})}
 \\ = \sum_{\alpha \geq 0} || \sum_{\alpha_{1}+\alpha_{2}=\alpha,\alpha_{2} \geq S-k} \alpha!
\frac{B_{1,k,\alpha_{1}}(V_{0},V_{1},U_{0})}{\alpha_{1}!}\\
 \times (\sum_{j \in I(\alpha_{2}-S+k)}
\frac{\mathbf{A}_{j,\alpha_{2}-S+k+1}(V_{0},V_{1},(U_{h})_{h \in I(\alpha_{2}-S+k+1)})}{\alpha_{2}!}\\
 \times (\partial_{U_j}\Psi_{\alpha_{2}-S+k})(V_{0},V_{1},(U_{h})_{h \in I(\alpha_{2}-S+k)}))
 ||_{\rho,\alpha,\bar{V}_{0},\bar{V}_{1},(\bar{U}_{h})_{h \in I(\alpha)}} \bar{W}^{\alpha}. \label{norm_partial_int_DA_Psi_defin}
\end{multline}
Now, using Lemma 3, we deduce that
\begin{multline}
||B_{1,k}(V_{0},V_{1},U_{0},W)\partial_{W}^{-S+k}\mathbb{D}_{\mathbf{A}}\Psi(V_{0},V_{1},
(U_{h})_{h \geq 0},W)||_{(\rho,\bar{V}_{0},\bar{V}_{1},(\bar{U}_{h})_{h \geq 0},\bar{W})}\\
\leq \sum_{\alpha \geq 0} (\sum_{\alpha_{1}+\alpha_{2}=\alpha,\alpha_{2} \geq S-k}
\frac{|B_{1,k,\alpha_{1}}|(\bar{V}_{0},\bar{V}_{1},\bar{U}_{0})}{\alpha_{1}!}||\sum_{j \in I(\alpha_{2}-S+k)}
\mathbf{A}_{j,\alpha_{2}-S+k+1}(V_{0},V_{1},(U_{h})_{h \in I(\alpha_{2}-S+k+1)})\\
 \times (\partial_{U_j}\Psi_{\alpha_{2}-S+k})(V_{0},V_{1},(U_{h})_{h \in I(\alpha_{2}-S+k)})||_{\rho,
\alpha_{2},\bar{V}_{0},\bar{V}_{1},(\bar{U}_{h})_{h \in I(\alpha_{2})}}) \bar{W}^{\alpha}
 \label{maj_1_norm_partial_int_DA_Psi_defin}
\end{multline}
In the next lemma, we give estimates for the coefficients of the series $\mathbf{A}_{j,\alpha}$ and $|B_{1,k,\alpha}|$.
\begin{lemma}
1) The coefficients of the Taylor series of $\mathbf{A}_{j,\alpha_{2}-S+k+1}(V_{0},V_{1},(U_{h})_{h \in I(\alpha_{2}-S+k+1)})$
\begin{multline*}
\mathbf{A}_{j,\alpha_{2}-S+k+1}(V_{0},V_{1},(U_{h})_{h \in I(\alpha_{2}-S+k+1)}) =\\
\sum_{n_{0},n_{1},l_{h} \geq 0, h \in I(\alpha_{2}-S+k+1)} A_{j,\alpha_{2}-S+k+1,n_{0},n_{1},(l_{h})_{h \in I(\alpha_{2}-S+k+1)}}
\frac{V_{0}^{n_0}}{n_{0}!}\frac{V_{1}^{n_1}}{n_{1}!}\Pi_{h \in I(\alpha_{2}-S+k+1)} \frac{U_{h}^{l_h}}{l_{h}!}
\end{multline*}
satisfy the next estimates. There exist constants $a,\delta > 0$, with $\delta > \bar{\delta}$, $a_{p}>0$, $0 \leq p \leq d$ such that
\begin{multline}
\frac{ A_{j,\alpha_{2}-S+k+1,n_{0},n_{1},(l_{h})_{h \in I(\alpha_{2}-S+k+1)}} }{ n_{0}!n_{1}!\Pi_{h \in I(\alpha_{2}-S+k+1)}l_{h}! } \\
\leq
\frac{a \nu (\alpha_{2}-S+k+1)^{2}(\rho+\delta) + (d+1)\max_{0 \leq p \leq d}a_{p}(\rho+\delta)^{d}\mathcal{P}_{d}(\alpha_{2}-S+k)}{\delta^{n_{0}+n_{1}+\sum_{h \in I(\alpha_{2}-S+k+1)} l_{h} }} \label{maj_Aj_alpha_n_l}
\end{multline}
for all $\alpha_{2} \geq S-k$, all $j \in I(\alpha_{2}-S+k)$, all $n_{0},n_{1},l_{h} \geq 0$, $h \in I(\alpha_{2}-S+k+1)$ where $\mathcal{P}_{d}$
is defined in (\ref{def_P(j)}).\\
2) The coefficients of the Taylor series of $|B_{1,k,\alpha_{1}}|(\bar{V}_{0},\bar{V}_{1},\bar{U}_{0})$
$$
|B_{1,k,\alpha_{1}}|(\bar{V}_{0},\bar{V}_{1},\bar{U}_{0}) = \sum_{n_{0},n_{1},l_{0} \geq 0} b_{1,k,\alpha_{1},n_{0},n_{1},l_{0}}
\frac{\bar{V}_{0}^{n_0}}{n_{0}!}\frac{\bar{V}_{1}^{n_1}}{n_{1}!}\frac{\bar{U}_{0}^{l_0}}{l_{0}!}
$$
satisfy the following inequalities. There exist constants $\delta > \bar{\delta}$, $D_{1,k},\hat{D}_{1,k}>0$ with
\begin{equation}
\frac{b_{1,k,\alpha_{1},n_{0},n_{1},l_{0}}}{n_{0}!n_{1}!l_{0}!} \leq
\frac{D_{1,k} (\rho + \delta)^{d_{1,k}} \alpha_{1}! \hat{D}_{1,k}^{\alpha_1}}{\delta^{n_{0}+n_{1}+l_{0}}}
\label{maj_b1k_alpha_n_l}
\end{equation}
for all $\alpha_{1} \geq 0$, all $n_{0},n_{1},l_{0} \geq 0$.
\end{lemma}
\begin{proof} We first treat the estimates for $\mathbf{A}_{j,\alpha}$. From the Cauchy formula in several variables, one can write
\begin{multline}
\frac{\partial_{v_0}^{n_0}\partial_{v_1}^{n_1}\Pi_{h \in I(\alpha_{2}-S+k+1)} \partial_{u_h}^{l_h}A_{j}(v_{0},v_{1},(u_{h})_{h \in I(\alpha_{2}-S+k+1)})}{n_{0}!n_{1}!\Pi_{h \in I(\alpha_{2}-S+k+1)}l_{h}!}\\
= (\frac{1}{2i\pi})^{\alpha_{2}-S+k+4} \int_{C(v_{0},\delta)} \int_{C(v_{1},\delta)} \Pi_{h \in I(\alpha_{2}-S+k+1)} \int_{C(u_{h},\delta)}
A_{j}(\chi_{0},\chi_{1},(\xi_{h})_{h \in I(\alpha_{2}-S+k+1)}) \\
\times \frac{d\chi_{0}d\chi_{1}  \Pi_{h \in I(\alpha_{2}-S+k+1)} d\xi_{h}}{(\chi_{0}-v_{0})^{n_{0}+1}(\chi_{1}-v_{1})^{n_{1}+1}
\Pi_{h \in I(\alpha_{2}-S+k+1)} (\xi_{h} - u_{h})^{l_{h}+1} } \label{Cauchy_form_Aj}
\end{multline}
for all $|v_{0}|<R$, $|v_{1}|<R$, $|u_{h}| < \rho$, $h \in I(\alpha_{2}-S+k+1)$ and $j \in I(\alpha_{2}-S+k)$ where $R$ is introduced
in Section 2.2. The integration is made along positively oriented circles with radius $\delta > 0$,
$C(v_{0},\delta)$,$C(v_{1},\delta)$ and $C(u_{h},\delta)$ for $h \in I(\alpha_{2}-S+k+1)$. We choose the real number
$\delta>\bar{\delta}$ in
such a way that $R + \delta < R'$ where $R'$ is defined in Section 2.1 and $\bar{\delta}$ at the beginning of Section 3.1. Now, since
the functions $a(\chi_{0},\chi_{1})$ and $a_{p}(\chi_{0},\chi_{1})$ are holomorphic on $D(0,R')^{2}$,
the number $\nu>0$ (see (\ref{sup_diff_z_X<rho})) can be chosen large enough such that there exist real numbers
$a,a_{p}>0$, for $0 \leq p \leq d$, with
\begin{equation}
\sup_{|\chi_{0}| < R+\delta,|\chi_{1}| < R+\delta}|\frac{\partial_{\chi_{1}}^{l_1}a(\chi_{0},\chi_{1})}{l_{1}! \nu^{l_1}}| \leq a \ \ , \ \
\sup_{|\chi_{0}| < R+\delta,|\chi_{1}| < R+\delta}|\frac{\partial_{\chi_{1}}^{l_0}a_{p}(\chi_{0},\chi_{1})}{l_{0}! \nu^{l_0}}| \leq a_{p} 
\end{equation}
for all $l_{0},l_{1} \geq 0$. We recall also that for any integers $k,n \geq 1$, the number of tuples
$(b_{1},\ldots,b_{k}) \in \mathbb{N}^{k}$ such that $b_{1}+\cdots+b_{k}=n$ is $(n+k-1)!/((k-1)!n!)$. From these latter statements and
the definition of $A_{j}$ given by (\ref{defin_A_j}), we deduce that
\begin{equation}
|A_{j}(\chi_{0},\chi_{1},(\xi_{h})_{h \in I(\alpha_{2}-S+k+1)})| \leq a \nu (j+1)^{2}(\rho+\delta) +
(d+1)\max_{0 \leq p \leq d}a_{p}(\rho+\delta)^{d}\mathcal{P}_{d}(j) \label{maj_Aj_chi_xi}
\end{equation}
(since $\rho > 1$), where
\begin{equation}
\mathcal{P}_{d}(j)= \frac{(j+d)!}{j!} = \Pi_{l=1}^{d}(j+l) \label{def_P(j)}
\end{equation}
is a polynomial of degree $d$ in $j$ with positive coefficients, for all $|\chi_{0}| < R+\delta$, $|\chi_{1}| < R+\delta$,
$|\xi_{h}| < \rho + \delta$, $h \in I(\alpha_{2}-S+k+1)$ and $j \in I(\alpha_{2}-S+k)$. Gathering (\ref{Cauchy_form_Aj}) and
(\ref{maj_Aj_chi_xi}) yields (\ref{maj_Aj_alpha_n_l}).

Again, from the Cauchy formula in several variables, one can write
\begin{multline}
\frac{\partial_{v_0}^{n_0}\partial_{v_1}^{n_1}\partial_{u_0}^{l_0}b_{1,k,\alpha_{1}}(v_{0},v_{1},u_{0})}{n_{0}!n_{1}!l_{0}!}
= (\frac{1}{2i\pi})^{3} \int_{C(v_{0},\delta)} \int_{C(v_{1},\delta)} \int_{C(u_{0},\delta)} b_{1,k,\alpha_{1}}(\chi_{0},\chi_{1},\xi_{0})\\
\times \frac{d\chi_{0}d\chi_{1}d\xi_{0}}{(\chi_{0}-v_{0})^{n_{0}+1}(\chi_{1}-v_{1})^{n_{1}+1}(\xi_{0} - u_{0})^{l_{0}+1} }
\label{Cauchy_form_b1k}
\end{multline}
for all $|v_{0}|<R$, $|v_{1}|<R$, $|u_{0}| < \rho$. Again, one chooses the real number $\delta>\bar{\delta}$ in
such a way that $R + \delta < R'$. By construction of $b_{1,k,\alpha_{1}}$ in Section 2.2, we deduce that there exist two constants
$D_{1,k},\hat{D}_{1,k}>0$ such that
\begin{equation}
|b_{1,k,\alpha_{1}}(\chi_{0},\chi_{1},\xi_{0})| \leq D_{1,k}(\rho+\delta)^{d_{1,k}} \alpha_{1}! \hat{D}_{1,k}^{\alpha_{1}}
\label{maj_b1k_chi_xi}
\end{equation}
for all $\alpha_{1} \geq 0$, all $|\chi_{0}| < R+\delta$, $|\chi_{1}| < R+\delta$, $|\xi_{0}| < \rho + \delta$. Gathering 
(\ref{Cauchy_form_b1k}) and (\ref{maj_b1k_chi_xi}) yields (\ref{maj_b1k_alpha_n_l}).
\end{proof}
>From (\ref{maj_b1k_alpha_n_l}), we deduce that
\begin{equation}
\frac{|B_{1,k,\alpha_{1}}|(\bar{V}_{0},\bar{V}_{1},\bar{U}_{0})}{\alpha_{1}!} \leq
\frac{D_{1,k}(\rho+\delta)^{d_{1,k}}\hat{D}_{1,k}^{\alpha_1}}{(1-\frac{\bar{V}_{0}}{\delta})
(1-\frac{\bar{V}_{1}}{\delta})(1-\frac{\bar{U}_{0}}{\delta})}. \label{maj_B1k_alpha_barV0_V1_U0}
\end{equation}
On the other hand, from the Proposition 5, we deduce that
\begin{multline}
||\mathbf{A}_{j,\alpha_{2}-S+k+1}(V_{0},V_{1},(U_{h})_{h \in I(\alpha_{2}-S+k+1)})\\
 \times (\partial_{U_j}\Psi_{\alpha_{2}-S+k})(V_{0},V_{1},(U_{h})_{h \in I(\alpha_{2}-S+k)})||_{\rho,\alpha_{2},\bar{V}_{0},\bar{V}_{1},(\bar{U}_{h})_{h \in I(\alpha_{2})}} \\
 \leq |\mathbf{A}_{j,\alpha_{2}-S+k+1}|(\bar{V}_{0},\bar{V}_{1},(\bar{U}_{h})_{h \in I(\alpha_{2}-S+k+1)})\\
 \times ||(\partial_{U_j}\Psi_{\alpha_{2}-S+k})(V_{0},V_{1},(U_{h})_{h \in I(\alpha_{2}-S+k)})||_{\rho,\alpha_{2},\bar{V}_{0},
\bar{V}_{1},(\bar{U}_{h})_{h \in I(\alpha_{2})}}. \label{maj_norm_product_Aj_alpha_partial_Psi_alpha}
\end{multline}
>From (\ref{maj_Aj_alpha_n_l}), we deduce that
\begin{multline}
 |\mathbf{A}_{j,\alpha_{2}-S+k+1}|(\bar{V}_{0},\bar{V}_{1},(\bar{U}_{h})_{h \in I(\alpha_{2}-S+k+1)}) \\
 \leq \frac{a \nu (\alpha_{2}-S+k+1)^{2}(\rho+\delta) + (d+1)\max_{0 \leq p \leq d}a_{p}(\rho+\delta)^{d}\mathcal{P}_{d}(\alpha_{2}-S+k)}
 {(1-\frac{\bar{V}_{0}}{\delta}) (1-\frac{\bar{V}_{1}}{\delta}) \Pi_{h \in I(\alpha_{2} - S+k+1)} (1-\frac{\bar{U}_{h}}{\delta}) }
 \label{maj_series_Aj_alpha}
\end{multline}
for all $j \in I(\alpha_{2}-S+k)$. Now, from the definition of $\bar{U}_{h}=\bar{\delta}/(h^{b}+1)$, where $b > 1$, we know that there
exists $\kappa > 0$ such that
\begin{equation}
\Pi_{h \in I(\alpha)} (1-\frac{\bar{U}_{h}}{\delta})  \geq \kappa \label{inf_prod_bar_Uh_conv}
\end{equation}
for all $\alpha \geq 0$. From Proposition 6, we have that
\begin{multline}
||(\partial_{U_j}\Psi_{\alpha_{2}-S+k})(V_{0},V_{1},(U_{h})_{h \in I(\alpha_{2}-S+k)})||_{\rho,\alpha_{2},\bar{V}_{0},\bar{V}_{1},(\bar{U}_{h})_{h \in I(\alpha_{2})}}
 \leq \frac{ \exp( -\sigma \rho \frac{S-k}{(\alpha_{2}+1)^{b}} ) }{\bar{U}_{j} \Pi_{l=1}^{S-k-1}(\alpha_{2} - l +1)}\\
 \times
||\Psi_{\alpha_{2}-S+k}(V_{0},V_{1},(U_{h})_{h \in I(\alpha_{2}-S+k)})||_{\rho,\alpha_{2}-S+k,\bar{V}_{0},\bar{V}_{1},
(\bar{U}_{h})_{h \in I(\alpha_{2}-S+k)}}. \label{maj_norm_partial_Psi_alpha}
\end{multline}
Collecting the estimates (\ref{maj_series_Aj_alpha}), (\ref{inf_prod_bar_Uh_conv}) and (\ref{maj_norm_partial_Psi_alpha}), we get
from (\ref{maj_norm_product_Aj_alpha_partial_Psi_alpha}) that 
\begin{multline}
||\sum_{j \in I(\alpha_{2}-S+k)} \mathbf{A}_{j,\alpha_{2}-S+k+1}(V_{0},V_{1},(U_{h})_{h \in I(\alpha_{2}-S+k+1)})\\
 \times (\partial_{U_j}\Psi_{\alpha_{2}-S+k})(V_{0},V_{1},(U_{h})_{h \in I(\alpha_{2}-S+k)})||_{\rho,\alpha_{2},\bar{V}_{0},\bar{V}_{1},(\bar{U}_{h})_{h \in I(\alpha_{2})}} \\
 \leq \mathcal{A}_{\rho,\alpha_{2}} ||\Psi_{\alpha_{2}-S+k}(V_{0},V_{1},(U_{h})_{h \in I(\alpha_{2}-S+k)})||_{\rho,\alpha_{2}-S+k,\bar{V}_{0},\bar{V}_{1},(\bar{U}_{h})_{h \in I(\alpha_{2}-S+k)}} \label{maj_norm_sumof_product_Aj_alpha_partial_Psi_alpha}
\end{multline}
where
\begin{multline*}
\mathcal{A}_{\rho,\alpha_{2}} = \frac{a \nu (\alpha_{2}-S+k+1)^{2}(\rho+\delta) + (d+1)\max_{0 \leq p \leq d}a_{p}(\rho+\delta)^{d}\mathcal{P}_{d}(\alpha_{2}-S+k)}{(1-\frac{\bar{V}_{0}}{\delta}) (1-\frac{\bar{V}_{1}}{\delta}) \kappa }\\
\times \frac{ \exp( -\sigma \rho \frac{S-k}{(\alpha_{2}+1)^{b}} ) }{\Pi_{l=1}^{S-k-1}(\alpha_{2} - l +1)} \frac{1}{\bar{\delta}}
(\alpha_{2} - S+k+1)((\alpha_{2}-S+k)^{b}+1).
\end{multline*}
Now, we recall the following classical estimates. Let $\delta,m_{1},m_{2}>0$ positive real numbers, then
\begin{equation}
\sup_{x \geq 0} (x+\delta)^{m_1}\exp(-m_{2}x) \leq (\frac{m_1}{m_2})^{m_1}\exp(-m_{1})\exp(\delta m_{2}) \label{xexpx<}
\end{equation}
holds. Hence,
\begin{multline}
(\rho + \delta)^{d_{1,k}}\mathcal{A}_{\rho,\alpha_{2}} \leq
\left( \frac{a \nu (\alpha_{2}-S+k+1)^{2}(\alpha_{2} + 1)^{b(1+d_{1,k})}(\frac{\exp(-1)(1+d_{1,k})}{\sigma (S-k)})^{1+d_{1,k}}
\exp(\delta \sigma (S-k)) }{(1-\frac{\bar{V}_{0}}{\delta}) (1-\frac{\bar{V}_{1}}{\delta}) \kappa } \right.\\
\left. +  \frac{ (d+1)\max_{0 \leq p \leq d}a_{p}\mathcal{P}_{d}(\alpha_{2}-S+k) (\alpha_{2}+1)^{b(d+d_{1,k})}(\frac{(d+d_{1,k})\exp(-1)}{\sigma(S-k)})^{d+d_{1,k}}\exp(\delta \sigma (S-k)) }{(1-\frac{\bar{V}_{0}}{\delta}) (1-\frac{\bar{V}_{1}}{\delta}) \kappa } \right)\\
\times \frac{ (\alpha_{2} - S+k+1)((\alpha_{2}-S+k)^{b}+1) }{ \bar{\delta} \Pi_{l=1}^{S-k-1}(\alpha_{2} - l +1) }.
\end{multline}
Under the assumptions (\ref{relation_S_k_1}), one gets a constant $\tilde{C}_{8.1}>0$ (depending on
$a,\max_{0 \leq p \leq d}a_{p},\delta,\bar{\delta},b,d$,\\$d_{1,k},\sigma,\nu,S,k,\kappa,\bar{V}_{0},\bar{V}_{1}$) such that
\begin{equation}
(\rho + \delta)^{d_{1,k}}\mathcal{A}_{\rho,\alpha_{2}} \leq \tilde{C}_{8.1} \label{A_rho_alpha_bounded}
\end{equation}
for all $\rho \geq 0$, all $\alpha_{2} \geq S-k$. Finally, gathering
(\ref{maj_1_norm_partial_int_DA_Psi_defin}), (\ref{maj_B1k_alpha_barV0_V1_U0}), (\ref{maj_norm_sumof_product_Aj_alpha_partial_Psi_alpha}), (\ref{A_rho_alpha_bounded}), one gets that
\begin{multline}
||B_{1,k}(V_{0},V_{1},U_{0},W)\partial_{W}^{-S+k}
\mathbb{D}_{\mathbf{A}}
\Psi(V_{0},V_{1},(U_{h})_{h \geq 0},W)||_{(\rho,\bar{V}_{0},\bar{V}_{1},(\bar{U}_{h})_{h \geq 0},\bar{W})}\\
\leq \sum_{\alpha \geq 0} ( \sum_{\alpha_{1}+\alpha_{2}=\alpha,\alpha_{2} \geq S-k}
\frac{D_{1,k}}{(1-\frac{\bar{V}_{0}}{\delta})(1-\frac{\bar{V}_{1}}{\delta})(1-\frac{\bar{U}_{0}}{\delta})} \hat{D}_{1,k}^{\alpha_1}\\
\times \tilde{C}_{8.1}||\Psi_{\alpha_{2}-S+k}(V_{0},V_{1},(U_{h})_{h \in I(\alpha_{2}-S+k)})||_{\rho,\alpha_{2}-S+k,\bar{V}_{0},\bar{V}_{1},(\bar{U}_{h})_{h \in I(\alpha_{2}-S+k)}}) \bar{W}^{\alpha_{1}+\alpha_{2}-S+k} \bar{W}^{S-k}\\
= \frac{\tilde{C}_{8.1} D_{1,k}}{(1-\frac{\bar{V}_{0}}{\delta})(1-\frac{\bar{V}_{1}}{\delta})(1-\frac{\bar{U}_{0}}{\delta})(1-
\hat{D}_{1,k}\bar{W})}
\\
\times \bar{W}^{S-k}||\Psi(V_{0},V_{1},(U_{h})_{h \geq 0},W)||_{(\rho,\bar{V}_{0},\bar{V}_{1},(\bar{U}_{h})_{h \geq 0},\bar{W})}
\end{multline}
provided that $\bar{V}_{0} < \delta$, $\bar{V}_{1}<\delta$, $\bar{U}_{0}<\delta$ and $\bar{W} < 1/\hat{D}_{1,k}$, which yields
(\ref{norm_intW_DA_Psi_W<norm_Psi_W}).

2) Now, we turn towards the estimates (\ref{norm_intW_DB_Psi_W<norm_Psi_W}) which will follow from the same arguments as in 1).
Using Lemma 3, we get that
\begin{multline}
||B_{2,k}(V_{0},V_{1},U_{0},W)\partial_{W}^{-S+k}\mathbb{D}_{\mathbf{B}}\Psi(V_{0},V_{1},
(U_{h})_{h \geq 0},W)||_{(\rho,\bar{V}_{0},\bar{V}_{1},(\bar{U}_{h})_{h \geq 0},\bar{W})}\\
\leq \sum_{\alpha \geq 0} (\sum_{\alpha_{1}+\alpha_{2}=\alpha,\alpha_{2} \geq S-k}
\frac{|B_{2,k,\alpha_{1}}|(\bar{V}_{0},\bar{V}_{1},\bar{U}_{0})}{\alpha_{1}!}||\sum_{j \in I(\alpha_{2}-S+k)}
\mathbf{B}_{j,\alpha_{2}-S+k+1}(V_{0},V_{1},(U_{h})_{h \in I(\alpha_{2}-S+k+1)})\\
 \times (\partial_{U_j}\Psi_{\alpha_{2}-S+k})(V_{0},V_{1},(U_{h})_{h \in I(\alpha_{2}-S+k)})||_{\rho,\alpha_{2},\bar{V}_{0},\bar{V}_{1},(\bar{U}_{h})_{h \in I(\alpha_{2})}} \bar{W}^{\alpha}
 \label{maj_1_norm_partial_int_DB_Psi_defin}
\end{multline}
In the next lemma, we give estimates for the coefficients of the series $\mathbf{B}_{j,\alpha}$ and $|B_{2,k,\alpha}|$.
\begin{lemma} 1) The coefficients of the Taylor series of
$\mathbf{B}_{j,\alpha_{2}-S+k+1}(V_{0},V_{1},(U_{h})_{h \in I(\alpha_{2}-S+k+1)})$
\begin{multline*}
\mathbf{B}_{j,\alpha_{2}-S+k+1}(V_{0},V_{1},(U_{h})_{h \in I(\alpha_{2}-S+k+1)}) =\\
\sum_{n_{0},n_{1},l_{h} \geq 0, h \in I(\alpha_{2}-S+k+1)} B_{j,\alpha_{2}-S+k+1,n_{0},n_{1},(l_{h})_{h \in I(\alpha_{2}-S+k+1)}}
\frac{V_{0}^{n_0}}{n_{0}!}\frac{V_{1}^{n_1}}{n_{1}!}\Pi_{h \in I(\alpha_{2}-S+k+1)} \frac{U_{h}^{l_h}}{l_{h}!}
\end{multline*}
satisfy the next estimates. There exist a constant $\delta > 0$, with $\delta > \bar{\delta}$ such that
\begin{equation}
\frac{ B_{j,\alpha_{2}-S+k+1,n_{0},n_{1},(l_{h})_{h \in I(\alpha_{2}-S+k+1)}} }{ n_{0}!n_{1}!\Pi_{h \in I(\alpha_{2}-S+k+1)}l_{h}! }
\leq
\frac{\nu(\alpha_{2}-S+k+1)(\rho+\delta)}{\delta^{n_{0}+n_{1}+\sum_{h \in I(\alpha_{2}-S+k+1)} l_{h} }} \label{maj_Bj_alpha_n_l}
\end{equation}
for all $\alpha_{2} \geq S-k$, all $j \in I(\alpha_{2}-S+k)$, all $n_{0},n_{1},l_{h} \geq 0$, $h \in I(\alpha_{2}-S+k+1)$.\\
2) The coefficients of the Taylor series of $|B_{2,k,\alpha_{1}}|(\bar{V}_{0},\bar{V}_{1},\bar{U}_{0})$
$$
|B_{2,k,\alpha_{1}}|(\bar{V}_{0},\bar{V}_{1},\bar{U}_{0}) = \sum_{n_{0},n_{1},l_{0} \geq 0} b_{2,k,\alpha_{1},n_{0},n_{1},l_{0}}
\frac{\bar{V}_{0}^{n_0}}{n_{0}!}\frac{\bar{V}_{1}^{n_1}}{n_{1}!}\frac{\bar{U}_{0}^{l_0}}{l_{0}!}
$$
satisfy the following inequalities. There exist constants $\delta > \bar{\delta}$, $D_{2,k},\hat{D}_{2,k}>0$ with
\begin{equation}
\frac{b_{2,k,\alpha_{1},n_{0},n_{1},l_{0}}}{n_{0}!n_{1}!l_{0}!} \leq
\frac{D_{2,k} (\rho + \delta)^{d_{2,k}} \alpha_{1}! \hat{D}_{2,k}^{\alpha_1}}{\delta^{n_{0}+n_{1}+l_{0}}}
\label{maj_b2k_alpha_n_l}
\end{equation}
for all $\alpha_{1} \geq 0$, all $n_{0},n_{1},l_{0} \geq 0$.
\end{lemma}
\begin{proof} 1) From the Cauchy formula in several variables, one can check that
\begin{multline}
\frac{\partial_{v_0}^{n_0}\partial_{v_1}^{n_1}\Pi_{h \in I(\alpha_{2}-S+k+1)} \partial_{u_h}^{l_h}B_{j}(v_{0},v_{1},(u_{h})_{h \in I(\alpha_{2}-S+k+1)})}{n_{0}!n_{1}!\Pi_{h \in I(\alpha_{2}-S+k+1)}l_{h}!}\\
= (\frac{1}{2i\pi})^{\alpha_{2}-S+k+4} \int_{C(v_{0},\delta)} \int_{C(v_{1},\delta)} \Pi_{h \in I(\alpha_{2}-S+k+1)} \int_{C(u_{h},\delta)}
B_{j}(\chi_{0},\chi_{1},(\xi_{h})_{h \in I(\alpha_{2}-S+k+1)}) \\
\times \frac{d\chi_{0}d\chi_{1}  \Pi_{h \in I(\alpha_{2}-S+k+1)} d\xi_{h}}{(\chi_{0}-v_{0})^{n_{0}+1}(\chi_{1}-v_{1})^{n_{1}+1}
\Pi_{h \in I(\alpha_{2}-S+k+1)} (\xi_{h} - u_{h})^{l_{h}+1} } \label{Cauchy_form_Bj}
\end{multline}
for all $|v_{0}|<R$, $|v_{1}|<R$, $|u_{h}| < \rho$, $h \in I(\alpha_{2}-S+k+1)$ and $j \in I(\alpha_{2}-S+k)$. We choose the real number
$\delta>\bar{\delta}$ in
such a way that $R + \delta < R'$. From the definition given in  (\ref{defin_B_j}), we get that
\begin{equation}
|B_{j}(\chi_{0},\chi_{1},(\xi_{h})_{h \in I(\alpha_{2}-S+k+1)})| \leq \nu (j+1) (\rho + \delta) \label{maj_Bj_chi_xi}
\end{equation}
for all $|\chi_{0}| < R+\delta$, $|\chi_{1}| < R+\delta$, $|\xi_{h}| < \rho + \delta$,
$h \in I(\alpha_{2}-S+k+1)$ and $j \in I(\alpha_{2}-S+k)$. Gathering (\ref{Cauchy_form_Bj}) and (\ref{maj_Bj_chi_xi}) yields
(\ref{maj_Bj_alpha_n_l}).\\
2) The proof is exactly the same as 2) in Lemma 5.
\end{proof}
>From (\ref{maj_b2k_alpha_n_l}), we deduce that
\begin{equation}
\frac{|B_{2,k,\alpha_{1}}|(\bar{V}_{0},\bar{V}_{1},\bar{U}_{0})}{\alpha_{1}!} \leq
\frac{D_{2,k}(\rho+\delta)^{d_{2,k}}\hat{D}_{2,k}^{\alpha_1}}{(1-\frac{\bar{V}_{0}}{\delta})
(1-\frac{\bar{V}_{1}}{\delta})(1-\frac{\bar{U}_{0}}{\delta})}. \label{maj_B2k_alpha_barV0_V1_U0}
\end{equation}
Using Propositions 5,6, we deduce that
\begin{multline}
||\sum_{j \in I(\alpha_{2}-S+k)} \mathbf{B}_{j,\alpha_{2}-S+k+1}(V_{0},V_{1},(U_{h})_{h \in I(\alpha_{2}-S+k+1)})\\
 \times (\partial_{U_j}\Psi_{\alpha_{2}-S+k})(V_{0},V_{1},(U_{h})_{h \in I(\alpha_{2}-S+k)})||_{\rho,\alpha_{2},\bar{V}_{0},\bar{V}_{1},(\bar{U}_{h})_{h \in I(\alpha_{2})}} \\
 \leq \mathcal{B}_{\rho,\alpha_{2}} ||\Psi_{\alpha_{2}-S+k}(V_{0},V_{1},
(U_{h})_{h \in I(\alpha_{2}-S+k)})||_{\rho,\alpha_{2}-S+k,\bar{V}_{0},\bar{V}_{1},(\bar{U}_{h})_{h \in I(\alpha_{2}-S+k)}}
\label{maj_norm_sumof_product_Bj_alpha_partial_Psi_alpha}
\end{multline}
where
$$
\mathcal{B}_{\rho,\alpha_{2}} = \frac{\nu (\alpha_{2}-S+k+1)(\rho+\delta)}{(1-\frac{\bar{V}_{0}}{\delta}) (1-\frac{\bar{V}_{1}}{\delta}) \kappa }\\
\times \frac{ \exp( -\sigma \rho \frac{S-k}{(\alpha_{2}+1)^{b}} ) }{\Pi_{l=1}^{S-k-1}(\alpha_{2} - l +1)} \frac{1}{\bar{\delta}}
(\alpha_{2} - S+k+1)((\alpha_{2}-S+k)^{b}+1)
$$
and where $\kappa$ is introduced in (\ref{inf_prod_bar_Uh_conv}). Using the estimates (\ref{xexpx<}), we get
\begin{multline}
(\rho + \delta)^{d_{2,k}}\mathcal{B}_{\rho,\alpha_{2}} \leq \left( \frac{\nu (\alpha_{2}-S+k+1)(\alpha_{2} + 1)^{b(1+d_{2,k})}(\frac{\exp(-1)(1+d_{2,k})}{\sigma (S-k)})^{1+d_{2,k}}
\exp(\delta \sigma (S-k)) }{(1-\frac{\bar{V}_{0}}{\delta}) (1-\frac{\bar{V}_{1}}{\delta}) \kappa } \right)\\
\times \frac{ (\alpha_{2} - S+k+1)((\alpha_{2}-S+k)^{b}+1) }{ \bar{\delta} \Pi_{l=1}^{S-k-1}(\alpha_{2} - l +1) }.
\end{multline}
Under the assumptions (\ref{relation_S_k_2}), one gets a constant $\tilde{C}_{8.2}>0$ (depending on
$\delta,\bar{\delta},b,d_{2,k},\sigma,\nu,S,k,\kappa,\\ \bar{V}_{0},\bar{V}_{1}$) such that
\begin{equation}
(\rho + \delta)^{d_{2,k}}\mathcal{B}_{\rho,\alpha_{2}} \leq \tilde{C}_{8.2} \label{B_rho_alpha_bounded}
\end{equation}
for all $\rho \geq 0$, all $\alpha_{2} \geq S-k$. Finally, gathering (\ref{maj_1_norm_partial_int_DB_Psi_defin}),
(\ref{maj_B2k_alpha_barV0_V1_U0}), (\ref{maj_norm_sumof_product_Bj_alpha_partial_Psi_alpha}) and
(\ref{B_rho_alpha_bounded}), we get (\ref{norm_intW_DB_Psi_W<norm_Psi_W}).
\end{proof}

\begin{prop} 1) Let $S,k \geq 0$ be integers such that
\begin{equation}
S \geq k+1+b\max(d_{1,k},d_{2,k}). \label{relation_S_k_3}
\end{equation}
Then, for $m \in \{0,1\}$, there exists a constant $C_{9}>0$ (which is independent of $\rho > 1$) such that
\begin{multline}
||B_{m+1,k}(V_{0},V_{1},U_{0},W)\partial_{W}^{-S+k}\partial_{V_m}\Psi(V_{0},V_{1},(U_{h})_{h \geq 0},W)||_{(\rho,\bar{V}_{0},\bar{V}_{1},(\bar{U}_{h})_{h \geq 0},\bar{W})}
\\
\leq C_{9}\bar{W}^{S-k}||\Psi(V_{0},V_{1},(U_{h})_{h \geq 0},W)||_{(\rho,\bar{V}_{0},\bar{V}_{1},(\bar{U}_{h})_{h \geq 0},\bar{W})}
\label{norm_intW_diffVm_Psi_W<norm_Psi_W}
\end{multline}
for all $\Psi(V_{0},V_{1},(U_{h})_{h \geq 0},W) \in G_{(\rho,\bar{V}_{0},\bar{V}_{1},(\bar{U}_{h})_{h \geq 0},\bar{W})}$.\\
2) Let $S,k \geq 0$ be integers such that
\begin{equation}
S \geq k+bd_{3,k}. \label{relation_S_k_3_2}
\end{equation}
Then, there exists a constant $C_{9.1}>0$ (which is independent of $\rho > 1$) such that
\begin{multline}
||B_{3,k}(V_{0},V_{1},U_{0},W)\partial_{W}^{-S+k}\Psi(V_{0},V_{1},(U_{h})_{h \geq 0},W)||_{(\rho,\bar{V}_{0},\bar{V}_{1},
(\bar{U}_{h})_{h \geq 0},\bar{W})}
\\
\leq C_{9.1}\bar{W}^{S-k}||\Psi(V_{0},V_{1},(U_{h})_{h \geq 0},W)||_{(\rho,\bar{V}_{0},\bar{V}_{1},(\bar{U}_{h})_{h \geq 0},
\bar{W})}
\label{norm_intW_Psi_W<norm_Psi_W}
\end{multline}
for all $\Psi(V_{0},V_{1},(U_{h})_{h \geq 0},W) \in G_{(\rho,\bar{V}_{0},\bar{V}_{1},(\bar{U}_{h})_{h \geq 0},\bar{W})}$.\\
\end{prop}
\begin{proof} 1) We expand
$$ B_{m+1,k}(V_{0},V_{1},U_{0},W) = \sum_{\alpha \geq 0} B_{m+1,k,\alpha}(V_{0},V_{1},U_{0}) \frac{W^{\alpha}}{\alpha!}.$$
By definition, we have
\begin{multline}
||B_{m+1,k}(V_{0},V_{1},U_{0},W)\partial_{W}^{-S+k}\partial_{V_m}\Psi(V_{0},V_{1},(U_{h})_{h \geq 0},W)||_{(\rho,\bar{V}_{0},\bar{V}_{1},(\bar{U}_{h})_{h \geq 0},\bar{W})}
 \\ = \sum_{\alpha \geq 0} || \sum_{\alpha_{1}+\alpha_{2}=\alpha,\alpha_{2} \geq S-k} \alpha! \frac{B_{m+1,k,\alpha_{1}}(V_{0},V_{1},U_{0})}{\alpha_{1}!}\\
 \times \frac{\partial_{V_m}\Psi_{\alpha_{2}-S+k}(V_{0},V_{1},(U_{h})_{h \in I(\alpha_{2}-S+k})}{\alpha_{2}!}
||_{\rho,\alpha,\bar{V}_{0},\bar{V}_{1},(\bar{U}_{h})_{h \in I(\alpha)}} \bar{W}^{\alpha}. \label{norm_partial_intW_diff_Vm_Psi_defin}
 \end{multline}
Now,  using Lemma 3, we deduce that
\begin{multline}
||B_{m+1,k}(V_{0},V_{1},U_{0},W)\partial_{W}^{-S+k}
\partial_{V_m}\Psi(V_{0},V_{1},(U_{h})_{h \geq 0},W)||_{(\rho,\bar{V}_{0},\bar{V}_{1},(\bar{U}_{h})_{h \geq 0},\bar{W})}\\
 \leq \sum_{\alpha \geq 0} (\sum_{\alpha_{1}+\alpha_{2}=\alpha,\alpha_{2} \geq S-k}
\frac{|B_{m+1,k,\alpha_{1}}|(\bar{V}_{0},\bar{V}_{1},\bar{U}_{0})}{\alpha_{1}!}\\
\times ||(\partial_{V_m}\Psi_{\alpha_{2}-S+k})(V_{0},V_{1},(U_{h})_{h \in I(\alpha_{2}-S+k)})||_{\rho,\alpha_{2},
\bar{V}_{0},\bar{V}_{1},(\bar{U}_{h})_{h \in I(\alpha_{2})}}) \bar{W}^{\alpha}
 \label{maj_1_norm_partial_intW_diff_Vm_Psi_defin}
\end{multline}
>From Proposition 6, we know that
\begin{multline}
||(\partial_{V_m}\Psi_{\alpha_{2}-S+k})(V_{0},V_{1},(U_{h})_{h \in I(\alpha_{2}-S+k)})||_{\rho,\alpha_{2},\bar{V}_{0},\bar{V}_{1},(\bar{U}_{h})_{h \in I(\alpha_{2})}}
 \leq \frac{ \exp( -\sigma \rho \frac{S-k}{(\alpha_{2}+1)^{b}} ) }{\bar{V}_{m} \Pi_{l=1}^{S-k-1}(\alpha_{2} - l +1)}\\
 \times ||\Psi_{\alpha_{2}-S+k}(V_{0},V_{1},(U_{h})_{h \in I(\alpha_{2}-S+k)})||_{\rho,\alpha_{2}-S+k,
\bar{V}_{0},\bar{V}_{1},(\bar{U}_{h})_{h \in I(\alpha_{2}-S+k)}}. \label{maj_norm_partialVm_Psi_alpha}
\end{multline}
>From (\ref{maj_B1k_alpha_barV0_V1_U0}), (\ref{maj_B2k_alpha_barV0_V1_U0}), (\ref{maj_1_norm_partial_intW_diff_Vm_Psi_defin})
and (\ref{maj_norm_partialVm_Psi_alpha}), we get that
\begin{multline}
||B_{m+1,k}(V_{0},V_{1},U_{0},W)\partial_{W}^{-S+k}\partial_{V_m}\Psi(V_{0},V_{1},(U_{h})_{h \geq 0},W)||_{(\rho,\bar{V}_{0},\bar{V}_{1},(\bar{U}_{h})_{h \geq 0},\bar{W})}\\
\leq \sum_{\alpha \geq 0} ( \sum_{\alpha_{1}+\alpha_{2}=\alpha,\alpha_{2} \geq S-k} \hat{D}_{m+1,k}^{\alpha_1}\\
\times \mathcal{C}_{\rho,\alpha_{2}}||\Psi_{\alpha_{2}-S+k}(V_{0},V_{1},(U_{h})_{h \in I(\alpha_{2}-S+k)})||_{\rho,\alpha_{2}-S+k,\bar{V}_{0},\bar{V}_{1},(\bar{U}_{h})_{h \in I(\alpha_{2}-S+k)}}) \bar{W}^{\alpha} \label{maj_2_norm_partial_intW_diff_Vm_Psi_defin}
\end{multline}
where
$$
\mathcal{C}_{\rho,\alpha_{2}} =
\frac{D_{m+1,k}(\rho+\delta)^{d_{m+1,k}}}{(1-\frac{\bar{V}_{0}}{\delta})(1-\frac{\bar{V}_{1}}{\delta})(1-\frac{\bar{U}_{0}}{\delta})}
\times \frac{ \exp( -\sigma \rho \frac{S-k}{(\alpha_{2}+1)^{b}} ) }{\bar{V}_{m} \Pi_{l=1}^{S-k-1}(\alpha_{2} - l +1)}.
$$
Using the estimates (\ref{xexpx<}), we deduce that
$$
\mathcal{C}_{\rho,\alpha_{2}} \leq \frac{D_{m+1,k}(\frac{d_{m+1,k}\exp(-1)}{\sigma (S-k)})^{d_{m+1,k}} \exp(\delta \sigma (S-k)) }{(1-\frac{\bar{V}_{0}}{\delta})(1-\frac{\bar{V}_{1}}{\delta})(1-\frac{\bar{U}_{0}}{\delta})}
\times \frac{ (\alpha_{2}+1)^{bd_{m+1,k}} }{\bar{V}_{m} \Pi_{l=1}^{S-k-1}(\alpha_{2} - l +1)}.
$$
Under the assumption (\ref{relation_S_k_3}), we get a constant $\tilde{C}_{9}>0$ (depending on $D_{m+1,k},d_{m+1,k},S,k,\delta,\sigma,\bar{V}_{0},\\ \bar{V}_{1},\bar{U}_{0},b$) such that
\begin{equation}
\mathcal{C}_{\rho,\alpha_{2}} \leq \tilde{C}_{9} \label{C_rho_alpha_bounded}
\end{equation}
for all $\rho > 1$, all $\alpha_{2} \geq S-k$. Finally, collecting (\ref{maj_2_norm_partial_intW_diff_Vm_Psi_defin}) and
(\ref{C_rho_alpha_bounded}), we get
\begin{multline}
||B_{m+1,k}(V_{0},V_{1},U_{0},W)\partial_{W}^{-S+k}\partial_{V_m}\Psi(V_{0},V_{1},(U_{h})_{h \geq 0},W)||_{(\rho,\bar{V}_{0},\bar{V}_{1},(\bar{U}_{h})_{h \geq 0},\bar{W})}\\
\leq \sum_{\alpha \geq 0} ( \sum_{\alpha_{1}+\alpha_{2}=\alpha,\alpha_{2} \geq S-k} \hat{D}_{m+1,k}^{\alpha_1}\\
\times \tilde{C}_{9}||\Psi_{\alpha_{2}-S+k}(V_{0},V_{1},(U_{h})_{h \in I(\alpha_{2}-S+k)})||_{\rho,\alpha_{2}-S+k,\bar{V}_{0},
\bar{V}_{1},(\bar{U}_{h})_{h \in I(\alpha_{2}-S+k)}}) \bar{W}^{\alpha_{1}+\alpha_{2}-S+k}\bar{W}^{S-k}\\
= \frac{ \tilde{C}_{9} }{1 - \hat{D}_{m+1,k}\bar{W}} \bar{W}^{S-k} ||\Psi(V_{0},V_{1},(U_{h})_{h \geq 0},W)||_{(\rho,\bar{V}_{0},\bar{V}_{1},(\bar{U}_{h})_{h \geq 0},\bar{W})}
\end{multline}
which yields (\ref{norm_intW_diffVm_Psi_W<norm_Psi_W}).\\
2) We expand
$$ B_{3,k}(V_{0},V_{1},U_{0},W) = \sum_{\alpha \geq 0} B_{3,k,\alpha}(V_{0},V_{1},U_{0}) \frac{W^{\alpha}}{\alpha!}.$$
By definition, we have
\begin{multline}
||B_{3,k}(V_{0},V_{1},U_{0},W)\partial_{W}^{-S+k}\Psi(V_{0},V_{1},(U_{h})_{h \geq 0},W)||_{(\rho,\bar{V}_{0},\bar{V}_{1},
(\bar{U}_{h})_{h \geq 0},\bar{W})}
 \\ = \sum_{\alpha \geq 0} || \sum_{\alpha_{1}+\alpha_{2}=\alpha,\alpha_{2} \geq S-k} \alpha!
\frac{B_{3,k,\alpha_{1}}(V_{0},V_{1},U_{0})}{\alpha_{1}!}\\
 \times \frac{\Psi_{\alpha_{2}-S+k}(V_{0},V_{1},(U_{h})_{h \in I(\alpha_{2}-S+k})}{\alpha_{2}!} ||_{\rho,\alpha,\bar{V}_{0},
\bar{V}_{1},(\bar{U}_{h})_{h \in I(\alpha)}} \bar{W}^{\alpha}. \label{norm_partial_intW_Psi_defin}
 \end{multline}
Now,  using Lemma 3, we deduce that
\begin{multline}
||B_{3,k}(V_{0},V_{1},U_{0},W)\partial_{W}^{-S+k}\Psi(V_{0},V_{1},(U_{h})_{h \geq 0},W)||_{(\rho,\bar{V}_{0},\bar{V}_{1},
(\bar{U}_{h})_{h \geq 0},\bar{W})}\\
=\sum_{\alpha \geq 0} (\sum_{\alpha_{1}+\alpha_{2}=\alpha,\alpha_{2} \geq S-k}
\frac{|B_{3,k,\alpha_{1}}|(\bar{V}_{0},\bar{V}_{1},\bar{U}_{0})}{\alpha_{1}!}\\
\times ||\Psi_{\alpha_{2}-S+k}(V_{0},V_{1},(U_{h})_{h \in I(\alpha_{2}-S+k)})||_{\rho,\alpha_{2},\bar{V}_{0},\bar{V}_{1},
(\bar{U}_{h})_{h \in I(\alpha_{2})}} \bar{W}^{\alpha}.
 \label{maj_1_norm_partial_intW_Psi_defin}
\end{multline}
>From Proposition 6, we know that
\begin{multline}
||\Psi_{\alpha_{2}-S+k}(V_{0},V_{1},(U_{h})_{h \in I(\alpha_{2}-S+k)}||_{\rho,\alpha_{2},\bar{V}_{0},\bar{V}_{1},
(\bar{U}_{h})_{h \in I(\alpha_{2})}}
 \leq \frac{ \exp( -\sigma \rho \frac{S-k}{(\alpha_{2}+1)^{b}} ) }{\Pi_{l=1}^{S-k}(\alpha_{2} - l +1)}\\
 \times ||\Psi_{\alpha_{2}-S+k}(V_{0},V_{1},(U_{h})_{h \in I(\alpha_{2}-S+k)})||_{\rho,\alpha_{2}-S+k,\bar{V}_{0},\bar{V}_{1},
(\bar{U}_{h})_{h \in I(\alpha_{2}-S+k)}}. \label{maj_norm_Psi_alpha}
\end{multline}
On the other hand, the coefficients of the Taylor series of $|B_{3,k,\alpha_{1}}|(\bar{V}_{0},\bar{V}_{1},\bar{U}_{0})$
$$
|B_{3,k,\alpha_{1}}|(\bar{V}_{0},\bar{V}_{1},\bar{U}_{0}) = \sum_{n_{0},n_{1},l_{0} \geq 0} b_{3,k,\alpha_{1},n_{0},n_{1},l_{0}}
\frac{\bar{V}_{0}^{n_0}}{n_{0}!}\frac{\bar{V}_{1}^{n_1}}{n_{1}!}\frac{\bar{U}_{0}^{l_0}}{l_{0}!}
$$
satisfy the following inequalities. There exist constants $\delta > \bar{\delta}$, $D_{3,k},\hat{D}_{3,k}>0$ with
\begin{equation}
\frac{b_{3,k,\alpha_{1},n_{0},n_{1},l_{0}}}{n_{0}!n_{1}!l_{0}!} \leq
\frac{D_{3,k} (\rho + \delta)^{d_{3,k}} \alpha_{1}! \hat{D}_{3,k}^{\alpha_1}}{\delta^{n_{0}+n_{1}+l_{0}}}
\label{maj_b3k_alpha_n_l}
\end{equation}
for all $\alpha_{1} \geq 0$, all $n_{0},n_{1},l_{0} \geq 0$. The proof copies 2) from Lemma 5. From
(\ref{maj_b3k_alpha_n_l}), we deduce that
\begin{equation}
\frac{|B_{3,k,\alpha_{1}}|(\bar{V}_{0},\bar{V}_{1},\bar{U}_{0})}{\alpha_{1}!} \leq
\frac{D_{3,k}(\rho+\delta)^{d_{3,k}}\hat{D}_{3,k}^{\alpha_1}}{(1-\frac{\bar{V}_{0}}{\delta})(1-\frac{\bar{V}_{1}}{\delta})
(1-\frac{\bar{U}_{0}}{\delta})}. \label{maj_B3k_alpha_barV0_V1_U0}
\end{equation}
>From (\ref{maj_B3k_alpha_barV0_V1_U0}), (\ref{maj_1_norm_partial_intW_Psi_defin})
and (\ref{maj_norm_Psi_alpha}), we get that
\begin{multline}
||B_{3,k}(V_{0},V_{1},U_{0},W)\partial_{W}^{-S+k}\Psi(V_{0},V_{1},(U_{h})_{h \geq 0},W)||_{(\rho,\bar{V}_{0},\bar{V}_{1},
(\bar{U}_{h})_{h \geq 0},\bar{W})}\\
\leq \sum_{\alpha \geq 0} ( \sum_{\alpha_{1}+\alpha_{2}=\alpha,\alpha_{2} \geq S-k} \hat{D}_{3,k}^{\alpha_1}\\
\times \mathcal{D}_{\rho,\alpha_{2}}||\Psi_{\alpha_{2}-S+k}(V_{0},V_{1},(U_{h})_{h \in I(\alpha_{2}-S+k)})||_{\rho,\alpha_{2}-S+k,
\bar{V}_{0},\bar{V}_{1},(\bar{U}_{h})_{h \in I(\alpha_{2}-S+k)}}) \bar{W}^{\alpha} \label{maj_2_norm_partial_intW_Psi_defin}
\end{multline}
where
$$
\mathcal{D}_{\rho,\alpha_{2}} =
\frac{D_{3,k}(\rho+\delta)^{d_{3,k}}}{(1-\frac{\bar{V}_{0}}{\delta})(1-\frac{\bar{V}_{1}}{\delta})(1-\frac{\bar{U}_{0}}{\delta})}
\times \frac{ \exp( -\sigma \rho \frac{S-k}{(\alpha_{2}+1)^{b}} ) }{\Pi_{l=1}^{S-k}(\alpha_{2} - l +1)}.
$$
Using the estimates (\ref{xexpx<}), we deduce that
$$
\mathcal{D}_{\rho,\alpha_{2}} \leq \frac{D_{3,k}(\frac{d_{3,k}\exp(-1)}{\sigma (S-k)})^{d_{3,k}} \exp(\delta \sigma (S-k)) }{(1-
\frac{\bar{V}_{0}}{\delta})(1-\frac{\bar{V}_{1}}{\delta})(1-\frac{\bar{U}_{0}}{\delta})}
\times \frac{ (\alpha_{2}+1)^{bd_{3,k}} }{\Pi_{l=1}^{S-k}(\alpha_{2} - l +1)}.
$$
Under the assumption (\ref{relation_S_k_3_2}), we get a constant $\tilde{C}_{9.1}>0$ (depending on $D_{3,k},d_{3,k},S,k,\delta,\sigma,
\bar{V}_{0},\\ \bar{V}_{1},\bar{U}_{0},b$) such that
\begin{equation}
\mathcal{D}_{\rho,\alpha_{2}} \leq \tilde{C}_{9.1} \label{D_rho_alpha_bounded}
\end{equation}
for all $\rho > 1$, all $\alpha_{2} \geq S-k$. Finally, collecting (\ref{maj_2_norm_partial_intW_Psi_defin}) and
(\ref{D_rho_alpha_bounded}), we get
\begin{multline}
||B_{3,k}(V_{0},V_{1},U_{0},W)\partial_{W}^{-S+k}\Psi(V_{0},V_{1},(U_{h})_{h \geq 0},W)||_{(\rho,\bar{V}_{0},\bar{V}_{1},
(\bar{U}_{h})_{h \geq 0},\bar{W})}\\
\leq \sum_{\alpha \geq 0} ( \sum_{\alpha_{1}+\alpha_{2}=\alpha,\alpha_{2} \geq S-k} \hat{D}_{3,k}^{\alpha_1}\\
\times \tilde{C}_{9.1}||\Psi_{\alpha_{2}-S+k}(V_{0},V_{1},(U_{h})_{h \in I(\alpha_{2}-S+k)})||_{\rho,\alpha_{2}-S+k,\bar{V}_{0},
\bar{V}_{1},(\bar{U}_{h})_{h \in I(\alpha_{2}-S+k)}} \bar{W}^{\alpha_{1}+\alpha_{2}-S+k}\bar{W}^{S-k})\\
= \frac{ \tilde{C}_{9.1} }{1 - \hat{D}_{3,k}\bar{W}} \bar{W}^{S-k}
||\Psi(V_{0},V_{1},(U_{h})_{h \geq 0},W)||_{(\rho,\bar{V}_{0},\bar{V}_{1},(\bar{U}_{h})_{h \geq 0},\bar{W})}
\end{multline}
which yields (\ref{norm_intW_Psi_W<norm_Psi_W}).\\
\end{proof}

\subsection{A functional partial differential equation in the Banach spaces of infinitely many variables
$G_{(\rho,\bar{V}_{0},\bar{V}_{1},(\bar{U})_{h \geq 0},\bar{W})}$.}

In the next proposition, we solve a functional fixed point equation within the Banach spaces of formal series introduced in the previous
subsection.

\begin{prop} We make the following assumptions
\begin{multline}
S \geq k+1 + \max(b(d_{1,k}+2)+3,d+1+b(d+d_{1,k}+1)) \ \ , \ \ S \geq k+3+b(2+d_{2,k}),\\
S \geq k+1+b\max(d_{1,k},d_{2,k}) \ \ , \ \ S \geq k + bd_{3,k} \label{shape_functional_eq_in_G_rho_VU}
\end{multline}
for all $k \in \mathcal{S}$. Then, for given $\bar{V}_{0},\bar{V}_{1},\bar{\delta}>0$, there exists $\bar{W} > 0$
(independent of $\rho > 1$) such that, for all
$\tilde{I}(V_{0},V_{1},(U_{h})_{h \geq 0},W) \in G_{(\rho,\bar{V}_{0},\bar{V}_{1},(\bar{U}_{h})_{h \geq 0},\bar{W})}$, the
functional equation
\begin{multline}
\Psi(V_{0},V_{1},(U_{h})_{h \geq 0},W)
= \sum_{k \in \mathcal{S}} B_{1,k}(V_{0},V_{1},U_{0},W)\partial_{W}^{-S+k}\partial_{V_0}
\Psi(V_{0},V_{1},(U_{h})_{h \geq 0},W)\\
+ B_{1,k}(V_{0},V_{1},U_{0},W)\partial_{W}^{-S+k}\mathbb{D}_{\mathbf{A}}\Psi(V_{0},V_{1},(U_{h})_{h \geq 0},W)\\
+ \sum_{k \in \mathcal{S}} B_{2,k}(V_{0},V_{1},U_{0},W)\partial_{W}^{-S+k}\partial_{V_1}
\Psi(V_{0},V_{1},(U_{h})_{h \geq 0},W)\\
+ B_{2,k}(V_{0},V_{1},U_{0},W)\partial_{W}^{-S+k}\mathbb{D}_{\mathbf{B}}\Psi(V_{0},V_{1},(U_{h})_{h \geq 0},W) \\
+ \sum_{k \in \mathcal{S}} B_{3,k}(V_{0},V_{1},U_{0},W)\partial_{W}^{-S+k}\Psi(V_{0},V_{1},(U_{h})_{h \geq 0},W)\\
+ \tilde{I}(V_{0},V_{1},(U_{h})_{h \geq 0},W) \label{functional_eq_in_G_rho_VU}
\end{multline}
has a unique solution $\Psi(V_{0},V_{1},(U_{h})_{h \geq 0},W) \in
G_{(\rho,\bar{V}_{0},\bar{V}_{1},(\bar{U}_{h})_{h \geq 0},\bar{W})}$.
Moreover, we have that
\begin{equation}
||\Psi(V_{0},V_{1},(U_{h})_{h \geq 0},W)||_{(\rho,\bar{V}_{0},\bar{V}_{1},(\bar{U}_{h})_{h \geq 0},\bar{W})} \leq 2
||\tilde{I}(V_{0},V_{1},(U_{h})_{h \geq 0},W)||_{(\rho,\bar{V}_{0},\bar{V}_{1},(\bar{U}_{h})_{h \geq 0},\bar{W})}.
\label{norm_sol_functional_eq_in_G_rho_VU<}
\end{equation}
\end{prop}
\begin{proof} We consider the map $\mathfrak{M}$ from the space $\mathbb{G}[[V_{0},V_{1},(U_{h})_{h \geq 0},W]]$ of
formal series (introduced in Definition 1) into itself defined as follows:
\begin{multline}
\mathfrak{M}( \Delta(V_{0},V_{1},(U_{h})_{h \geq 0},W) ) = \sum_{k \in \mathcal{S}} B_{1,k}(V_{0},V_{1},U_{0},W)\partial_{W}^{-S+k}\partial_{V_0}
\Delta(V_{0},V_{1},(U_{h})_{h \geq 0},W)\\
+ B_{1,k}(V_{0},V_{1},U_{0},W)\partial_{W}^{-S+k}\mathbb{D}_{\mathbf{A}}\Delta(V_{0},V_{1},(U_{h})_{h \geq 0},W)\\
+ \sum_{k \in \mathcal{S}} B_{2,k}(V_{0},V_{1},U_{0},W)\partial_{W}^{-S+k}\partial_{V_1}
\Delta(V_{0},V_{1},(U_{h})_{h \geq 0},W)\\
+ B_{2,k}(V_{0},V_{1},U_{0},W)\partial_{W}^{-S+k}\mathbb{D}_{\mathbf{B}}\Delta(V_{0},V_{1},(U_{h})_{h \geq 0},W)\\
+ \sum_{k \in \mathcal{S}} B_{3,k}(V_{0},V_{1},U_{0},W)\partial_{W}^{-S+k}\Delta(V_{0},V_{1},(U_{h})_{h \geq 0},W)
\label{map_M}
\end{multline}
for all $\Delta(V_{0},V_{1},(U_{h})_{h \geq 0},W) \in \mathbb{G}[[V_{0},V_{1},(U_{h})_{h \geq 0},W]]$.

In order to prove the proposition, we need the following lemma.

\begin{lemma} Let $\mathrm{id}$ the identity map $x \mapsto x$ from $\mathbb{G}[[V_{0},V_{1},(U_{h})_{h \geq 0},W]]$ into itself.
Then, for a well chosen $\bar{W} > 0$, the map $\mathrm{id} - \mathfrak{M}$ defines an invertible map
such that $(\mathrm{id} - \mathfrak{M})^{-1}$ is defined from
$G_{(\rho,\bar{V}_{0},\bar{V}_{1},(\bar{U}_{h})_{h \geq 0},\bar{W})}$ into itself. Moreover, we have that
\begin{multline}
||(\mathrm{id} - \mathfrak{M})^{-1}(\Xi(V_{0},V_{1},(U_{h})_{h \geq 0},W))||_{(\rho,\bar{V}_{0},\bar{V}_{1},
(\bar{U}_{h})_{h \geq 0},\bar{W})}\\
\leq 2||\Xi(V_{0},V_{1},(U_{h})_{h \geq 0},W)||_{(\rho,\bar{V}_{0},\bar{V}_{1},(\bar{U}_{h})_{h \geq 0},\bar{W})}
\label{norm_inverse_id_minus_M<}
\end{multline}
for all $\Xi(V_{0},V_{1},(U_{h})_{h \geq 0},W) \in G_{(\rho,\bar{V}_{0},\bar{V}_{1},(\bar{U}_{h})_{h \geq 0},\bar{W})}$.
\end{lemma}
\begin{proof} Taking care of the constraints (\ref{shape_functional_eq_in_G_rho_VU}), we get from Propositions 8 and 9 a constant
$C_{10}>0$ (depending on the constants introduced above and also on the aforementioned propositions:
$a,\max_{0 \leq p \leq d}a_{p}$,$\delta,\bar{\delta}$,
$b,d$,$\max_{k \in \mathcal{S}} d_{1,k}$,\\$\max_{k \in \mathcal{S}} D_{1,k}$,
$\max_{k \in \mathcal{S}} d_{2,k},\max_{k \in \mathcal{S}} D_{2,k}$,
$\max_{k \in \mathcal{S}} d_{3,k},\max_{k \in \mathcal{S}} D_{3,k}$, $\sigma,\nu$,$S,\mathcal{S}$,$\kappa,\bar{V}_{0},\bar{V}_{1}$
but independent of $\rho > 1$)
such that
\begin{multline*}
||\mathfrak{M}(\Delta(V_{0},V_{1},(U_{h})_{h \geq 0},W))||_{(\rho,\bar{V}_{0},\bar{V}_{1},(\bar{U}_{h})_{h \geq 0},\bar{W})}\\
\leq C_{10} (\sum_{k \in \mathcal{S}} \bar{W}^{S-k}) ||\Delta(V_{0},V_{1},(U_{h})_{h \geq 0},W)||_{(\rho,\bar{V}_{0},\bar{V}_{1},
(\bar{U}_{h})_{h \geq 0},\bar{W})}
\end{multline*}
for all $\Delta(V_{0},V_{1},(U_{h})_{h \geq 0},W) \in G_{(\rho,\bar{V}_{0},\bar{V}_{1},(\bar{U}_{h})_{h \geq 0},\bar{W})}$
with $0 \leq \bar{W} \leq \min_{m \in \{0,1,2\},k \in \mathcal{S}} 1/(2\hat{D}_{m+1,k})$. Since $S > k$ for all $k \in \mathcal{S}$, we
can choose $\bar{W}>0$ such that 
$$ C_{10} \sum_{k \in \mathcal{S}} \bar{W}^{S-k} < \frac{1}{2} $$
together with $\bar{W} \leq \min_{m \in \{0,1,2\},k \in \mathcal{S}}
1/(2\hat{D}_{m+1,k})$. We deduce that
\begin{multline*}
||\mathfrak{M}(\Delta(V_{0},V_{1},(U_{h})_{h \geq 0},W))||_{(\rho,\bar{V}_{0},\bar{V}_{1},(\bar{U}_{h})_{h \geq 0},\bar{W})}\\
\leq \frac{1}{2} ||\Delta(V_{0},V_{1},(U_{h})_{h \geq 0},W)||_{(\rho,\bar{V}_{0},\bar{V}_{1},(\bar{U}_{h})_{h \geq 0},\bar{W})}
\end{multline*}
for all $\Delta(V_{0},V_{1},(U_{h})_{h \geq 0},W) \in G_{(\rho,\bar{V}_{0},\bar{V}_{1},(\bar{U}_{h})_{h \geq 0},\bar{W})}$. This
yields the estimates (\ref{norm_inverse_id_minus_M<}).
\end{proof}
Finally, let $\tilde{I}(V_{0},V_{1},(U_{h})_{h \geq 0},W) \in G_{(\rho,\bar{V}_{0},\bar{V}_{1},(\bar{U}_{h})_{h \geq 0},\bar{W})}$ for
$\bar{W}>0$ chosen as in Lemma 7. We define
$$ \Psi(V_{0},V_{1},(U_{h})_{h \geq 0},W) = (\mathrm{id} - \mathfrak{M})^{-1}(\tilde{I}(V_{0},V_{1},(U_{h})_{h \geq 0},W)).$$
By construction, $\Psi(V_{0},V_{1},(U_{h})_{h \geq 0},W)$ belongs to $G_{(\rho,\bar{V}_{0},\bar{V}_{1},(\bar{U}_{h})_{h \geq 0},
\bar{W})}$
and solves the equation (\ref{functional_eq_in_G_rho_VU}) with the estimates (\ref{norm_sol_functional_eq_in_G_rho_VU<}).
\end{proof}

\section{Analytic solutions with growth estimates of linear partial differential equations in $\mathbb{C}^{3}$.}

We are now in position to state the main result of our work.

\begin{theo} Let $b_{m,k}(t,z,u_{0},w)$ be the functions defined in (\ref{b_mk_t_z_u0_w_defin}) for $m=1,2,3$ and $k \in \mathcal{S}$.
Let us assume that there exists $b>1$ such that
\begin{multline}
S \geq k+1 + \max(b(d_{1,k}+2)+3,d+1+b(d+d_{1,k}+1)) \ \ , \ \ S \geq k+3+b(2+d_{2,k}),\\
S \geq k+1+b\max(d_{1,k},d_{2,k}) \ \ , \ \ S \geq k + bd_{3,k} \label{shape_lin_PDE_in_C3}
\end{multline}
for all $k \in \mathcal{S}$. For all $0 \leq j \leq S-1$, we consider functions $\omega_{j}(t,z)$ which are assumed to be holomorphic and
bounded on the product $D(0,R')^{2}$.

Then, there exist constants $\sigma,\bar{W},C_{12}>0$ such that the problem
\begin{multline}
\partial_{w}^{S}Y(t,z,w) = \sum_{k \in \mathcal{S}} (b_{1,k}(t,z,X(t,z),w)\partial_{t}\partial_{w}^{k}Y(t,z,w)\\
+ b_{2,k}(t,z,X(t,z),w)\partial_{z}\partial_{w}^{k}Y(t,z,w)+ b_{3,k}(t,z,X(t,z),w)\partial_{w}^{k}Y(t,z,w)) \label{lin_pde_Y}
\end{multline}
with initial data
\begin{equation}
(\partial_{w}^{j}Y)(t,z,0) = \omega_{j}(t,z) \ \ , \ \ 0 \leq j \leq S-1, \label{lin_pde_Y_init_cond}
\end{equation}
has a solution $Y(t,z,w)$ which is holomorphic on $\mathrm{Int}(K) \times D(0,\bar{W}/2)$ and which fulfills the following estimates
\begin{equation}
\sup_{(t,z) \in \mathrm{Int}(K),w \in D(0,\bar{W}/2)} |Y(t,z,w)| \leq C_{12}\exp( \sigma \zeta(b) \rho )
\\+ \sum_{j=0}^{S-1} \sup_{(t,z) \in \mathrm{Int}(K)} |\omega_{j}(t,z)|\frac{(\bar{W}/2)^{j}}{j!}
\end{equation}
where $\zeta(b) = \sum_{n \geq 0} 1/(n+1)^{b}$, for any compact set $K \subset D(0,R)^{2} \setminus \Theta$ with non-empty interior
$\mathrm{Int}(K)$ for some $R<R'$ and any $\rho>1$ which satisfies (\ref{sup_diff_z_X<rho}). We stress that the constants
$\sigma,\bar{W},C_{12}>0$ do not depend neither on $K$ nor on $\rho>1$.
\end{theo}

\begin{proof} By convention, we will put $\omega_{j}(t,z) \equiv 0$ for all $j \geq S$. On the other hand, we specialize the functions
 $\tilde{\omega}_{\alpha}$ which were introduced in (\ref{tilde_omega_defin}) in order that
\begin{multline}
\tilde{\omega}_{\alpha}(v_{0},v_{1},(u_{h})_{h \in I(\alpha)}) = \hat{\omega}_{\alpha}(v_{0},v_{1},u_{0})\\
= \sum_{k \in \mathcal{S}} \sum_{\alpha_{1}+\alpha_{2}=\alpha} \alpha!( \frac{b_{1,k,\alpha_{1}}(v_{0},v_{1},u_{0})}{\alpha_{1}!}
\frac{ \partial_{v_0}\omega_{\alpha_{2}+k}(v_{0},v_{1}) }{\alpha_{2}!} + \frac{b_{2,k,\alpha_{1}}(v_{0},v_{1},u_{0})}{\alpha_{1}!}
\frac{ \partial_{v_1}\omega_{\alpha_{2}+k}(v_{0},v_{1}) }{\alpha_{2}!}\\
 + \frac{b_{3,k,\alpha_{1}}(v_{0},v_{1},u_{0})}{\alpha_{1}!}
\frac{\omega_{\alpha_{2}+k}(v_{0},v_{1}) }{\alpha_{2}!}).
\end{multline}
By construction and using the definition (\ref{b_mk_alpha_n_l_and_tilde_omega_alpha_n_l_defin}), we can write with the
help of the Kronecker symbol,
\begin{equation}
\tilde{\omega}_{\alpha,n_{0},n_{1},(l_{h})_{h \in I(\alpha)}} = \hat{\omega}_{\alpha,n_{0},n_{1},l_{0}} \times \Pi_{h \in I(\alpha) \setminus \{0 \}}
\delta_{0,l_{h}}
\label{tilde_omega_alpha_n_l_factorize}
\end{equation}
where
$$ \hat{\omega}_{\alpha,n_{0},n_{1},l_{0}} = \sup_{|v_{0}|<R,|v_{1}|<R,|u_{0}|<\rho} |\partial_{v_0}^{n_0}
\partial_{v_1}^{n_1}\partial_{u_0}^{l_0} \hat{\omega}_{\alpha}(v_{0},v_{1},u_{0})|.
$$
\begin{lemma} There exist $\tilde{V}_{0},\tilde{V}_{1},\tilde{W}>0$ such that the formal series
\begin{multline}
\tilde{\Omega}(V_{0},V_{1},(U_{h})_{h \geq 0},W) \\
= \sum_{\alpha \geq 0}
\left( \sum_{n_{0},n_{1},l_{h} \geq 0, h \in I(\alpha)} \tilde{\omega}_{\alpha,n_{0},n_{1},(l_{h})_{h \in I(\alpha)}}
\frac{V_{0}^{n_0}}{n_{0}!}\frac{V_{1}^{n_1}}{n_{1}!}\Pi_{h \in I(\alpha)}\frac{U_{h}^{l_h}}{l_{h}!} \right)
\frac{W^{\alpha}}{\alpha!}
\end{multline}
belongs to $G_{(\rho,\tilde{V}_{0},\tilde{V}_{1},(\bar{U}_{h})_{h \geq 0},\tilde{W})}$. Moreover, there exists a constant
$C_{11}>0$ (independent of $\rho$) such that
\begin{equation}
||\tilde{\Omega}(V_{0},V_{1},(U_{h})_{h \geq 0},W)||_{(\rho,\tilde{V}_{0},\tilde{V}_{1},(\bar{U}_{h})_{h \geq 0},\tilde{W})}
\leq C_{11}.
\label{norm_tilde_Omega_V_U_C11}
\end{equation}
\end{lemma}
\begin{proof} Let $k \in \mathcal{S}$. By construction of $b_{m,k}(t,z,u_{0},w)$, we get couples of constants
$D_{1,k},\hat{D}_{1,k}>0$, $D_{2,k},\hat{D}_{2,k}>0$ and $D_{3,k},\hat{D}_{3,k}>0$ such that
\begin{multline}
|b_{1,k,\alpha_{1}}(\chi_{0},\chi_{1},\xi_{0})| \leq D_{1,k}(\rho + \delta)^{d_{1,k}}\alpha_{1}! (\hat{D}_{1,k})^{\alpha_1},\\
|b_{2,k,\alpha_{1}}(\chi_{0},\chi_{1},\xi_{0})| \leq D_{2,k}(\rho + \delta)^{d_{2,k}}\alpha_{1}! (\hat{D}_{2,k})^{\alpha_1} \ \ , \ \
|b_{3,k,\alpha_{1}}(\chi_{0},\chi_{1},\xi_{0})| \leq D_{3,k}(\rho + \delta)^{d_{3,k}}\alpha_{1}! (\hat{D}_{3,k})^{\alpha_1}
\label{b_1k_alpha_chi_xi_and_b_2k_alpha_chi_xi_bounds}
\end{multline}
for all $\alpha_{1} \geq 0$, all $|\chi_{0}| < R+\delta < R'$, $|\chi_{1}| < R+\delta < R'$, $|\xi_{0}|< \rho + \delta$. Moreover, we also
get couples of constants $E_{1,k},\hat{E}_{1,k}>0$, $E_{2,k},\hat{E}_{2,k}>0$ and $E_{3,k},\hat{E}_{3,k}>0$ such that
\begin{multline}
|\partial_{\chi_{0}}\omega_{\alpha_{2}+k}(\chi_{0},\chi_{1})| \leq E_{1,k}\alpha_{2}! (\hat{E}_{1,k})^{\alpha_2},\\
|\partial_{\chi_{1}}\omega_{\alpha_{2}+k}(\chi_{0},\chi_{1})| \leq E_{2,k}\alpha_{2}! (\hat{E}_{2,k})^{\alpha_2} \ \ , \ \
|\omega_{\alpha_{2}+k}(\chi_{0},\chi_{1})| \leq E_{3,k}\alpha_{2}! (\hat{E}_{3,k})^{\alpha_2}
\label{partial_chi_omega_alpha_chi_bounds}
\end{multline}
for all $\alpha_{2} \geq 0$, all $|\chi_{0}| < R+\delta < R'$, $|\chi_{1}| < R+\delta < R'$. From
(\ref{b_1k_alpha_chi_xi_and_b_2k_alpha_chi_xi_bounds}) and (\ref{partial_chi_omega_alpha_chi_bounds}) we deduce
\begin{multline}
|\hat{\omega}_{\alpha}(\chi_{0},\chi_{1},\xi_{0})|\\
\leq \sum_{k \in \mathcal{S}} \sum_{\alpha_{1}+\alpha_{2}=\alpha} \alpha!( D_{1,k}E_{1,k}(\rho+\delta)^{d_{1,k}}
(\hat{D}_{1,k})^{\alpha_{1}}
(\hat{E}_{1,k})^{\alpha_{2}} + D_{2,k}E_{2,k}(\rho+\delta)^{d_{2,k}}(\hat{D}_{2,k})^{\alpha_{1}} (\hat{E}_{2,k})^{\alpha_{2}}\\
+ D_{3,k}E_{3,k}(\rho+\delta)^{d_{3,k}}(\hat{D}_{3,k})^{\alpha_{1}} (\hat{E}_{3,k})^{\alpha_{2}}  )
\end{multline}
for all $\alpha \geq 0$, all $|\chi_{0}| < R+\delta < R'$, $|\chi_{1}| < R+\delta < R'$, $|\xi_{0}|< \rho + \delta$.
>From the Cauchy formula in several variables, one can write
\begin{multline*}
\frac{\partial_{v_0}^{n_0}\partial_{v_1}^{n_1}\partial_{u_0}^{l_0}\hat{\omega}_{\alpha}(v_{0},v_{1},u_{0})}{n_{0}!n_{1}!l_{0}!}
= (\frac{1}{2i\pi})^{3} \int_{C(v_{0},\delta)} \int_{C(v_{1},\delta)} \int_{C(u_{0},\delta)} \hat{\omega}_{\alpha}(\chi_{0},\chi_{1},\xi_{0})\\
\times \frac{d\chi_{0}d\chi_{1}d\xi_{0}}{(\chi_{0}-v_{0})^{n_{0}+1}(\chi_{1}-v_{1})^{n_{1}+1}(\xi_{0} - u_{0})^{l_{0}+1} }
\end{multline*}
for all $|v_{0}|<R$, $|v_{1}|<R$, $|u_{0}| < \rho$. We deduce that
\begin{multline}
\frac{\hat{\omega}_{\alpha,n_{0},n_{1},l_{0}}}{n_{0}!n_{1}!l_{0}!} \leq \frac{1}{\delta^{n_{0}+n_{1}+l_{0}}}\\
\times \sum_{k \in \mathcal{S}} \sum_{\alpha_{1}+\alpha_{2}=\alpha} \alpha!( D_{1,k}E_{1,k}(\rho+\delta)^{d_{1,k}}(\hat{D}_{1,k})^{\alpha_{1}}
(\hat{E}_{1,k})^{\alpha_{2}} + D_{2,k}E_{2,k}(\rho+\delta)^{d_{2,k}}(\hat{D}_{2,k})^{\alpha_{1}} (\hat{E}_{2,k})^{\alpha_{2}}\\
+ D_{3,k}E_{3,k}(\rho+\delta)^{d_{3,k}}(\hat{D}_{3,k})^{\alpha_{1}} (\hat{E}_{3,k})^{\alpha_{2}})
\label{hat_omega_alpha_n_l_Cauchy_maj}
\end{multline}
for all $\alpha \geq 0$, all $n_{0},n_{1},l_{0} \geq 0$. Using (\ref{tilde_omega_alpha_n_l_factorize}), we get that
\begin{multline}
||\tilde{\Omega}(V_{0},V_{1},(U_{h})_{h \geq 0},W)||_{(\rho,\tilde{V}_{0},\tilde{V}_{1},(\bar{U}_{h})_{h \geq 0},\tilde{W})}\\
= \sum_{\alpha \geq 0} ( \sum_{n_{0},n_{1},l_{0} \geq 0} \frac{ |\hat{w}_{\alpha,n_{0},n_{1},l_{0}}| }{\exp( \sigma r_{b}(\alpha) \rho ) }
\frac{\tilde{V}_{0}^{n_0}\tilde{V}_{1}^{n_1}\bar{U}_{0}^{l_0}}{(n_{0}+n_{1}+l_{0}+\alpha)!}) \tilde{W}^{\alpha}.
\end{multline}
>From (\ref{hat_omega_alpha_n_l_Cauchy_maj}), (\ref{xexpx<}) and with the help of the classical estimates
$$ (n_{0}+n_{1}+l_{0}+\alpha)! \geq n_{0}!n_{1}!l_{0}!\alpha!, $$
for all $n_{0},n_{1},l_{0},\alpha \geq 0$, we get a constant $C_{11,1}>0$ (depending on
$D_{1,k}$,$d_{1,k}$,$E_{1,k}$,$D_{2,k}$,$d_{2,k}$,$E_{2,k}$,\\$D_{3,k}$,$d_{3,k}$,$E_{3,k}$ for all
$k \in \mathcal{S}$,$\sigma$,$\delta$) such that
\begin{multline}
||\tilde{\Omega}(V_{0},V_{1},(U_{h})_{h \geq 0},W)||_{(\rho,\tilde{V}_{0},\tilde{V}_{1},(\bar{U}_{h})_{h \geq 0},\tilde{W})}\\
\leq \frac{C_{11,1}}{(1 - \frac{\tilde{V}_{0}}{\delta})(1 - \frac{\tilde{V}_{1}}{\delta})(1 - \frac{\bar{U}_{0}}{\delta})}
\sum_{k \in \mathcal{S}} \frac{1}{(1 - \hat{D}_{1,k}\tilde{W})(1 - \hat{E}_{1,k}\tilde{W})} \\
+ \frac{1}{(1 - \hat{D}_{2,k}\tilde{W})(1 - \hat{E}_{2,k}\tilde{W})} +
\frac{1}{(1 - \hat{D}_{3,k}\tilde{W})(1 - \hat{E}_{3,k}\tilde{W})}.
\label{norm_tilde_Omega_V_U<11.1}
\end{multline}
We choose
\begin{multline}
0 < \tilde{W} < \min_{k \in \mathcal{S}}(1/(2\hat{D}_{1,k}),1/(2\hat{D}_{2,k}),1/(2\hat{D}_{3,k}),1/(2\hat{E}_{1,k}),
1/(2\hat{E}_{2,k}),1/(2\hat{E}_{3,k})),\\
0 < \tilde{V}_{0} < \delta/2, 0 < \tilde{V}_{1} < \delta/2, 0< \bar{U}_{0} < \delta/2.
\end{multline}
>From (\ref{norm_tilde_Omega_V_U<11.1}) we deduce the inequality (\ref{norm_tilde_Omega_V_U_C11}).
\end{proof}
Under the assumption (\ref{shape_lin_PDE_in_C3}), we get from Proposition 10 four constants
$0<\bar{V}_{0}<\tilde{V}_{0}$, $0<\bar{V}_{1}<\tilde{V}_{1}$, $0<\bar{U}_{0}$ and $0<\bar{W}<\tilde{W}$
(independent of $\rho$) such that the functional equation
\begin{multline}
\Psi(V_{0},V_{1},(U_{h})_{h \geq 0},W)
= \sum_{k \in \mathcal{S}} ( B_{1,k}(V_{0},V_{1},U_{0},W)\partial_{W}^{-S+k}\partial_{V_0}
\Psi(V_{0},V_{1},(U_{h})_{h \geq 0},W)\\
+ B_{1,k}(V_{0},V_{1},U_{0},W)\partial_{W}^{-S+k}\mathbb{D}_{\mathbf{A}}\Psi(V_{0},V_{1},(U_{h})_{h \geq 0},W) )\\
+ \sum_{k \in \mathcal{S}} ( B_{2,k}(V_{0},V_{1},U_{0},W)\partial_{W}^{-S+k}\partial_{V_1}
\Psi(V_{0},V_{1},(U_{h})_{h \geq 0},W)\\
+ B_{2,k}(V_{0},V_{1},U_{0},W)\partial_{W}^{-S+k}\mathbb{D}_{\mathbf{B}}\Psi(V_{0},V_{1},(U_{h})_{h \geq 0},W) ) \\
+ \sum_{k \in \mathcal{S}} B_{3,k}(V_{0},V_{1},U_{0},W)\partial_{W}^{-S+k}\Psi(V_{0},V_{1},(U_{h})_{h \geq 0},W)\\
+ \tilde{\Omega}(V_{0},V_{1},(U_{h})_{h \geq 0},W) \label{functional_eq_in_G_rho_VU_tilde_Omega}
\end{multline}
has a unique solution $\Psi(V_{0},V_{1},(U_{h})_{h \geq 0},W)$ belonging to $G_{(\rho,\bar{V}_{0},\bar{V}_{1},
(\bar{U}_{h})_{h \geq 0},\bar{W})}$ which satisfies moreover the estimates
\begin{multline}
||\Psi(V_{0},V_{1},(U_{h})_{h \geq 0},W)||_{(\rho,\bar{V}_{0},\bar{V}_{1},(\bar{U}_{h})_{h \geq 0},\bar{W})} \leq
2||\tilde{\Omega}(V_{0},V_{1},(U_{h})_{h \geq 0},W)||_{(\rho,\bar{V}_{0},\bar{V}_{1},(\bar{U}_{h})_{h \geq 0},\bar{W})}\\
\leq 2C_{11}.
\label{norm_Psi_V_U_W<2C11}
\end{multline}
Now, from Proposition 4, we know that the sequence $\varphi_{\alpha,n_{0},n_{1},(l_{h})_{h \in I(\alpha)}}$ introduced in (\ref{varphi_alpha_n_l_defin}) satisfies the inequality
\begin{equation}
\varphi_{\alpha,n_{0},n_{1},(l_{h})_{h \in I(\alpha)}} \leq \psi_{\alpha,n_{0},n_{1},(l_{h})_{h \in I(\alpha)}}
\label{varphi_alpha_n_l<Psi_alpha_n_l_in_theo}
\end{equation}
for all $\alpha \geq 0$, all $n_{0},n_{1},l_{h} \geq 0$, for $h \in I(\alpha)$. Gathering (\ref{norm_Psi_V_U_W<2C11}) and
(\ref{varphi_alpha_n_l<Psi_alpha_n_l_in_theo}), and from the definition of the Banach spaces in Section 3.1, we get, in particular, for
$n_{0}=n_{1}=l_{h}=0$, for all $h \in I(\alpha)$, all $\alpha \geq 0$, that
\begin{multline}
\sup_{|v_{0}|<R,|v_{1}|<R,|u_{h}|<\rho,h \in I(\alpha)}|\phi_{\alpha}(v_{0},v_{1},(u_{h})_{h \in I(\alpha)})| \leq
\psi_{\alpha,0,0,(0)_{h \in I(\alpha)}}\\
\leq 2C_{11}\exp(\sigma r_{b}(\alpha) \rho)(\frac{1}{\bar{W}})^{\alpha} \alpha! \leq
2C_{11}\exp(\sigma \zeta(b) \rho)(\frac{1}{\bar{W}})^{\alpha} \alpha!
\label{sup_phi_alpha_v_u<exp_rho}
\end{multline}
for all $\alpha \geq 0$ and where $\zeta(b) = \sum_{n \geq 0} 1/(n+1)^{b}$. From (\ref{sup_phi_alpha_v_u<exp_rho}), we get
that the formal series $U(t,z,w)$ introduced in (\ref{U_defin}) actually defines a holomorphic function (denoted again by
$U(t,z,w)$) on $\mathrm{Int}(K) \times D(0,\bar{W}/2)$ for which the estimates
\begin{equation}
\sup_{(t,z) \in \mathrm{Int}(K),w \in D(0,\bar{W}/2)} |U(t,z,w)| \leq 4C_{11}\exp( \sigma \zeta(b) \rho )
\label{sup_U_tzw<exp_rho}
\end{equation}
hold and which satisfies the equation (\ref{ID_U}) on $\mathrm{Int}(K) \times D(0,\bar{W}/2)$.

Finally, we define the function
$$ Y(t,z,w) = \partial_{w}^{-S}U(t,z,w) + \sum_{j=0}^{S-1} \omega_{j}(t,z) \frac{w^j}{j!}.$$
By construction, $Y(t,z,w)$ defines a holomorphic function on
$\mathrm{Int}(K) \times D(0,\bar{W}/2)$ with bounds estimates
\begin{multline}
\sup_{(t,z) \in \mathrm{Int}(K),w \in D(0,\bar{W}/2)} |Y(t,z,w)| \leq 4(\frac{\bar{W}}{2})^{S}C_{11}\exp( \sigma \zeta(b) \rho )
\\+ \sum_{j=0}^{S-1} \sup_{(t,z) \in \mathrm{Int}(K)} |\omega_{j}(t,z)|\frac{(\bar{W}/2)^{j}}{j!}
\end{multline}
and solves the problem (\ref{lin_pde_Y}), (\ref{lin_pde_Y_init_cond}). This yields the result.
\end{proof}


\begin{thebibliography}{99}

\bibitem{al} S. Alinhac, \emph{Probl\`emes de Cauchy pour des op\'erateurs singuliers}. (French) Bull. Soc. Math. France 102 (1974),
289--315.

\bibitem{bolepa} A. Bove, J. Lewis, C. Parenti, \emph{Propagation of singularities for fuchsian operators}. Lecture Notes in
Mathematics, 984. Springer-Verlag, Berlin, 1983. ii+161 pp.

\bibitem{copata} O. Costin, H. Parkh and Y. Takei, \emph{Borel summability of the heat equation with variable coefficients,} J. Diff. Eq.
Volume 252, 4, pp. 3076--3092 (2012).

\bibitem{de} L. Debnath, \emph{Nonlinear partial differential equations for scientists and engineers.} Birkh\"{a}user Boston, Inc., Boston, MA, 1997. 593 pp.

\bibitem{di} S. Dineen, \emph{Complex analysis on infinite dimensional spaces}, Springer, 1999.

\bibitem{ev} L. Evans, \emph{Partial differential equations.} Second edition. Graduate Studies in Mathematics,
19. American Mathematical Society, Providence, RI, 2010. 749 pp. 

\bibitem{geta} R. G\'erard, H. Tahara, \emph{Singular nonlinear partial differential equations}. Aspects of Mathematics. Friedr. Vieweg
and Sohn, Braunschweig, 1996. viii+269 pp.

\bibitem{ha} Y. Hamada, \emph{The singularities of the solutions of the Cauchy problem.} Publ. Res. Inst. Math. Sci. 5 1969 21--40.

\bibitem{halewa} Y. Hamada, J. Leray, C. Wagschal, \emph{Syst\`emes d'\'equations aux d\'eriv\'ees partielles \`a caract\'eristiques
multiples: probl\`eme de Cauchy ramifi\'e; hyperbolicit\'e partielle.} J. Math. Pures Appl. (9) 55 (1976), no. 3, 297--352.

\bibitem{ig} K. Igari, \emph{On the branching of singularities in complex domains.} Proc. Japan Acad. Ser. A Math. Sci. 70
(1994), no. 5, 128--130.

\bibitem{lamasa} A. Lastra, S. Malek, J. Sanz, \emph{On $q$-asymptotics for linear $q$-difference-differential equations with Fuchsian
and irregular singularities.} J. Differential Equations 252 (2012), no. 10, 5185--5216.

\bibitem{le} J. Leray, \emph{Probl\`eme de Cauchy. I. Uniformisation de la solution du probl\`eme lin\'eaire analytique de Cauchy
pr\`es de la vari\'et\'e qui porte les donn\'ees de Cauchy.} Bull. Soc. Math. France 85 1957 389--429.

\bibitem{ma1} S. Malek, \emph{On the summability of formal solutions of linear partial differential equations.} J. Dyn. Control Syst. 11
(2005), no. 3, 389--403.

\bibitem{ma2} S. Malek, \emph{On functional linear partial differential equations in Gevrey spaces of holomorphic functions.} Ann. Fac.
Sci. Toulouse Math. (6) 16 (2007), no. 2, 285--302.

\bibitem{ma3} S. Malek, \emph{On Gevrey functional solutions of partial differential equations with Fuchsian and irregular singularities.}
J. Dyn. Control Syst. 15 (2009), no. 2, 277--305.

\bibitem{mast} S. Malek, C. Stenger, \emph{Complex singularity analysis of holomorphic solutions of linear partial differential equations},
Adv. Dyn. Syst. Appl. vol 6. no. 2 (2011).

\bibitem{malf} W. Malfliet, \emph{Solitary wave solutions of nonlinear wave equations.} Amer. J. Phys. 60 (1992),  no. 7, 650--654.

\bibitem{man} T. Mandai, \textit{The method of Frobenius to Fuchsian partial differential equations}. J. Math. Soc. Japan 52 (2000),
no. 3, 645--672.

\bibitem{ohsatato} Y. Ohta, J. Satsuma, D. Takahashi, D. Tokihiro, \emph{An elementary introduction to Sato theory.} Recent
developments in soliton theory.  Progr. Theoret. Phys. Suppl.  No. 94  (1988), 210--241.

\bibitem{ou1} S. Ouchi, \emph{An integral representation of singular solutions of linear partial differential equations in the complex
domain.} J. Fac. Sci. Univ. Tokyo Sect. IA Math. 27 (1980), no. 1, 37--85.

\bibitem{ou2} S. Ouchi, \emph{The behaviour near the characteristic surface of singular solutions of linear partial differential equations
in the complex domain.} Proc. Japan Acad. Ser. A Math. Sci. 65 (1989), no. 4, 102--105.

\bibitem{stsh} B. Sternin, V. Shatalov, \emph{Differential equations on complex manifolds.} Mathematics and its Applications, 276.
Kluwer Academic Publishers Group, Dordrecht, 1994. xii+504 pp.

\bibitem{ta} H. Tahara, \emph{Coupling of two partial differential equations and it's applications}, Publ. RIMS. Kyoto Univ. 43 (2007).

\bibitem{wa} C. Wagschal, \emph{Sur le probl\`eme de Cauchy ramifi\'e.} J. Math. Pures Appl. (9) 53 (1974), 147--163.


\end{thebibliography}
\end{document}